\documentclass{amsproc}
\usepackage{amsmath, euscript}
\usepackage{cases}
\usepackage{mathrsfs}
\usepackage{bbm}
\usepackage{amssymb}
\usepackage{txfonts}
\usepackage{amscd}
\usepackage{amsfonts,latexsym,amsmath,amsthm,amsxtra,mathdots,amssymb,latexsym,mathabx}
\usepackage[all,cmtip]{xy}
\usepackage{color}
\usepackage{multicol}
\usepackage{hyperref}
\usepackage{tikz}

\allowdisplaybreaks

\newtheorem{acknowledgement}{Acknowledgements}

\def\mod{\mathrm{mod}\ }

    \newcommand{\BC}{{\mathbb {C}}} \newcommand{\BD}{{\mathbb {D}}}
     
     \newcommand{\BH}{{\mathbb {H}}}
     \newcommand{\BJ}{{\mathbb {J}}}
     \newcommand{\BL}{{\mathbb {L}}}
     \newcommand{\BN}{{\mathbb {N}}}
     
    \newcommand{\BQ}{{\mathbb {Q}}} \newcommand{\BR}{{\mathbb {R}}}
    \newcommand{\BS}{{\mathbb {S}}} 
    \newcommand{\BU}{{\mathbb {U}}} 
     \newcommand{\BX}{{\mathbb {X}}}
     \newcommand{\BZ}{{\mathbb {Z}}}

    \newcommand{\fr}{{\mathfrak{r}}} \newcommand{\fC}{{\mathfrak{C}}}
    \newcommand{\GL}{{\mathrm{GL}}}
    \renewcommand{\Im}{{\mathfrak{Im}\,}}

    \newcommand{\fc}{{\mathfrak{c}}} 
     
    \renewcommand{\Re}{{\mathfrak{Re}\,}}

\def\frc{\mathfrak{c}}

\def\-{^{-1}}

\def\ux{\boldsymbol x}
\def\ualpha{\boldsymbol r}
\def\udelta{\boldsymbol \delta}
\def\ulambda{\boldsymbol \lambda}
\def\uLambda{\boldsymbol \varLambda}
\def\usigma{\boldsymbol \varsigma}
\def\unu{\boldsymbol \nu}

\def\urho{\boldsymbol \varrho}
\def\ue{\boldsymbol e}

\def\-{^{-1}}
\def\ut{\boldsymbol t}
\def\ualpha{\boldsymbol \alpha}
\def\lp {\left (}
\def\rp {\right )}
\def\EC {\EuScript C}

\def\Voronoi{Vorono\" \i \hskip 3 pt}

\newcommand{\delete}[1]{}

    \newcommand{\sgn}{{\mathrm{sgn}}}

    \newcommand{\ds}{\displaystyle}
    \newcommand{\sstyle}{\scriptstyle}

    \newcommand{\ra}{\rightarrow}

    \theoremstyle{plain}

    \newtheorem{thm}{Theorem}[section] \newtheorem{cor}[thm]{Corollary}
    \newtheorem{lem}[thm]{Lemma}  \newtheorem{prop}[thm]{Proposition}
    
     \newtheorem{defn}[thm]{Definition}
    \newtheorem{term}[thm]{Terminology}
    \newtheorem {rem}[thm]{Remark}

        \newtheorem*{observ}{Observation}
        \newtheorem*{assum}{Assumption}

    \numberwithin{equation}{section}

\makeatletter
\newsavebox\myboxA
\newsavebox\myboxB
\newlength\mylenA

\newcommand*\xoverline[2][0.75]{%
	\sbox{\myboxA}{$\m@th#2$}%
	\setbox\myboxB\null
	\ht\myboxB=\ht\myboxA%
	\dp\myboxB=\dp\myboxA%
	\wd\myboxB=#1\wd\myboxA
	\sbox\myboxB{$\m@th\overline{\copy\myboxB}$}
	\setlength\mylenA{\the\wd\myboxA}
	\addtolength\mylenA{-\the\wd\myboxB}%
	\ifdim\wd\myboxB<\wd\myboxA%
	\rlap{\hskip 0.5\mylenA\usebox\myboxB}{\usebox\myboxA}%
	\else
	\hskip -0.5\mylenA\rlap{\usebox\myboxA}{\hskip 0.5\mylenA\usebox\myboxB}%
	\fi}
\makeatother

\begin{document}

\title[Theory of Bessel Functions of High Rank - I]{Theory of Bessel Functions of High Rank - 
I: \\ Fundamental Bessel Functions}


\author{Zhi Qi}
\address{Department of Mathematics\\ The Ohio State University\\100 Math Tower\\231 West 18th Avenue\\Columbus, OH 43210\\USA}
\email{qi.91@buckeyemail.osu.edu}

\begin{abstract}
In this article we introduce a new category of special functions called  \textit{fundamental Bessel functions} arising from the \Voronoi summation formula for $\GL_n (\BR)$. The fundamental Bessel functions of rank one and two are the oscillatory exponential functions $e^{\pm i x}$ and the classical Bessel functions respectively. 
The main implements and subjects of our study of  fundamental Bessel functions are their formal integral representations and Bessel equations. 
\end{abstract}

\maketitle

\begin{footnotesize}
\tableofcontents
\end{footnotesize}

\section{Introduction}


\subsection{Background}\label{sec: Background}

$ $ 
\textit{Hankel transforms} 
(of high rank) are introduced as an important constituent of the \textit{\Voronoi  summation formula} by Miller and Schmid in \cite{Miller-Schmid-2004, Miller-Schmid-2006, Miller-Schmid-2009}.
This summation formula is a fundamental analytic tool in number theory and has its roots in representation theory. 

In this article, we shall develop the analytic theory of \textit{fundamental Bessel functions}\footnote{The Bessel functions  studied here are called \textit{fundamental} in order to be distinguished from the Bessel functions for $\GL_n (\BR)$. 
The latter should be regarded as the foundation of harmonic analysis on $\GL_n (\BR)$. Some evidences show that fundamental Bessel functions are actually the building blocks of the Bessel functions for $\GL_n (\BR)$. See \cite[\S 3.2]{Qi-Thesis} for $\GL_3 (\BR)$ (and $\GL_3 (\BC)$).

Throughout this article, we shall drop the adjective {\it fundamental} for brevity. Moreover, the usual Bessel functions will be referred to as classical Bessel functions.}. These Bessel functions  constitute the integral kernels of  Hankel transforms. Thus, to motivate our study, we shall start with introducing   Hankel transforms and their number theoretic and representation theoretic background.

\subsubsection{Two expressions of a Hankel transform}

Let $n $ be a positive integer, and let $(\ulambda, \udelta) = (\lambda_1, ..., \lambda_n, \delta_1, ..., \delta_n) \in \BC^n \times (\BZ/2 \BZ)^n$. \\

The first expression of the Hankel transform  of rank $n$ associated with $(\ulambda, \udelta)$ is based on   signed Mellin transforms as follows.

Let $\mathscr S (\BR )$ denote the space of Schwartz functions on $\BR $. For $\lambda \in \BC$, $j \in \BN = \{ 0, 1, 2, ... \} $ and $\eta \in \BZ/2\BZ$, let $\upsilon$ be a smooth function on $\BR^\times = \BR \smallsetminus \{ 0 \}$ such that $\sgn (x)^{\eta } \lp \log |x| \rp^{-j} |x|^{- \lambda } \upsilon (x) \in \mathscr S (\BR)$.
For $\delta \in \BZ/ 2\BZ$, the \textit{signed Mellin transform $\EuScript M_{\delta} \upsilon$ with order $\delta$} of  $\upsilon$ is defined by
\begin{equation} \label{1def: Mellin transform}
\EuScript M_{\delta} \upsilon (s) = \int_{\BR^\times} \upsilon (x) \sgn (x)^\delta |x|^{s } d^\times x.
\end{equation}
Here $d^\times x = |x|\- d x$ is the standard multiplicative Haar measure on $\BR^\times$. The Mellin inversion formula is
\begin{equation*}
\upsilon (x) = \sum_{\delta \in \BZ/ 2\BZ } \frac { \sgn (x)^\delta} {4 \pi i} \int_{(\sigma)} \EuScript M_\delta \upsilon (s) |x|^{-s} d s,  \hskip 10 pt \sigma > - \Re \lambda,
\end{equation*}
where the contour of integration $ (\sigma)$ is the vertical line   from $\sigma - i \infty$ to $\sigma + i \infty$.

Let $\mathscr S (\BR ^\times)$ denote the space of smooth functions on $\BR^\times $ whose derivatives are rapidly decreasing at both zero and infinity.
We associate with $\upsilon \in \mathscr S (\BR ^\times)$ 
a function $\Upsilon$ on $\BR ^\times$ satisfying the following two identities
\begin{equation}\label{1eq: Hankel transform identity}
\EuScript M_\delta \Upsilon (s ) =  \left( \prod_{l      = 1}^n G_{\delta_{l     } + \delta} (s - \lambda_l     ) \right) \EuScript M_\delta \upsilon ( 1 - s),  \hskip 10 pt \delta \in \BZ/2 \BZ,
\end{equation}
where $G_{\delta} (s)$ denotes the gamma factor
\begin{equation} \label{1eq: definition of G delta}
G_\delta (s) = i^\delta \pi^{ \frac 1 2 - s} \frac {\Gamma \lp \frac 1 2 ( {s + \delta} ) \rp} {\Gamma \lp \frac 1 2 ( {1 - s + \delta} ) \rp} = 
\left\{ \begin{split}
& 2(2 \pi)^{-s} \Gamma (s) \cos \left(\frac {\pi s} 2 \right), \hskip 10 pt \text { if } \delta = 0,\\
& 2 i (2 \pi)^{-s} \Gamma (s) \sin  \left(\frac {\pi s} 2 \right), \hskip 9 pt \text { if } \delta = 1,
\end{split} \right.
\end{equation}
where for the second equality we apply the duplication formula and Euler's reflection formula of the Gamma function,
\begin{equation*}
\Gamma (1-s) \Gamma (s) = \frac \pi {\sin (\pi s)}, \hskip 10 pt \Gamma (s) \Gamma \lp s + \frac 1 2 \rp = 2^{1-2s} \sqrt \pi \Gamma (2 s).
\end{equation*}
$\Upsilon$ is called the \textit{Hankel transform of index $(\ulambda, \udelta)$} of $\upsilon$\footnote{Note that if $\upsilon$ is the $f$ in \cite{Miller-Schmid-2009} then $|x| \Upsilon((-)^n x)$ is their $F(x)$.}.
According to \cite[\S 6]{Miller-Schmid-2006}, $\Upsilon$ is smooth on $\BR^\times$ and decays rapidly at infinity, along with all its derivatives. At the origin,  $\Upsilon$ has singularities of some very particular type. Indeed,  $\Upsilon (x) \in \sum_{l      = 1}^n \sgn (x)^{\delta_l     } |x|^{-\lambda_l     }  \mathscr S (\BR)$ when no two components of $\ulambda$ differ by an integer, and in the nongeneric case powers of $\log |x|$ will be included. 

By the Mellin inversion,
\begin{equation}\label{1eq: Psi defined by Gamma functions}
\Upsilon (x) 
= \sum_{\delta \in \BZ/ 2\BZ } \frac { \sgn (x)^\delta} {4 \pi i} \int_{(\sigma)} \left( \prod_{l      = 1}^n G_{\delta_{l     } + \delta} (s - \lambda_l      ) \right) \EuScript M_\delta \upsilon (1 - s) |x|^{ - s} d s,
\end{equation}
for $\sigma > \max  \left\{ \Re  \lambda_l      \right \} $.\\

In \cite{Miller-Schmid-2009} there is an alternative description of $\Upsilon$ defined by the \textit{Fourier type transform}, in symbolic notion, as follows
\begin{equation}\label{1eq: generalized Fourier transform}
\begin{split}
\Upsilon (x) =  \frac 1 { |x|} \int_{\BR^{\times \, n}} \upsilon \left( \frac {x_1 ... x_n} { x} \right) \left( \prod_{l      = 1}^{n} \left( \sgn (x_l     )^{\delta_l     } |x_l     |^{- \lambda_l     } e ( x_l     ) \right) \right) d x_n d x_{n-1} ... d x_1,
\end{split}
\end{equation}
with $e(x) = e^{2\pi i x}$. The integral  in  (\ref{1eq: generalized Fourier transform}) converges when performed as iterated integral in the order $d x_n d x_{n-1} ... d x_1$, starting from $x_n$, then $x_{n-1}$, ..., and finally $x_1$, provided $\Re  \lambda_1 >  ... > \Re  \lambda_{n-1} > \Re  \lambda_n$, and it has meaning for arbitrary values of $\ulambda \in \BC^n$ by analytic continuation. 

According to \cite{Miller-Schmid-2009}, though \textit{less suggestive than} (\ref{1eq: generalized Fourier transform}), the expression (\ref{1eq: Psi defined by Gamma functions}) of Hankel transforms is \textit{more useful in applications}. 
Indeed, {all} the applications of the \Voronoi summation formula in analytic number theory so far are based on (\ref{1eq: Psi defined by Gamma functions}) with exclusive use of Stirling's asymptotic formula of the Gamma function (see Appendix \ref{appendix: asymptotic}). On the other hand, there is no occurrence of the Fourier type integral transform \eqref{1eq: generalized Fourier transform} in the literature other than Miller and Schmid's foundational work. It will however be shown in this article that the expression \eqref{1eq: generalized Fourier transform} should {\it not}  be  of only aesthetic interest.

\begin{assum}
Subsequently, we shall always assume that the index $\ulambda$ satisfies  $\sum_{l      = 1}^n \lambda_l      = 0$\footnote{This condition is just a matter of normalization. Equivalently, the corresponding representations  of $\GL_n (\BR)$ are trivial on the positive component of the center. With this condition on $\ulambda$, the associated Bessel functions can be expressed in a simpler way.}. Accordingly, we define the complex hyperplane $\BL^{n-1} = \left\{\ulambda \in \BC^n : \sum_{l      = 1}^n \lambda_l      = 0 \right \}$.
\end{assum}

\subsubsection{Background of Hankel transforms in number theory and representation theory}

For $n = 1$, the number theoretic background lies on the local theory in Tate's thesis at the real place. 
Actually, in view of \eqref{1eq: generalized Fourier transform}, 
the Hankel transform of rank one and index $(\ulambda, \udelta) = (0, \delta)$ is essentially the (inverse) Fourier transform,
\begin{equation}\label{1eq: n=1, Hankel}
\Upsilon (x) = \int_{\BR} \upsilon (y) \sgn ( xy)^\delta  e( xy) d y.
\end{equation}
The  \Voronoi  summation formula of rank one is the summation formula of Poisson. Recall that Riemann's proof of the functional equation of his $\zeta$-function relies on the Poisson summation formula, whereas Tate's thesis reinterprets this using the Poisson summation formula for the adele ring.





For $n=2$, the Hankel transform associated with a $\GL_2$-automorphic form has been present in the literature as part of the \Voronoi summation formula for $\GL_2$ for decades. See, for instance,
\cite[Proposition 1]{Harcos-Michel} and the references there. 
According to \cite[Proposition 1]{Harcos-Michel} (see also Remark \ref{2rem: n=2}), we have
\begin{equation} \label{1eq: n=2, Hankel}
\Upsilon (x) = \int_{\BR^\times} \upsilon (y) J_F (xy) dy, \hskip 10 pt x \in \BR^\times,
\end{equation}
where, if $F$ is a Maa\ss  \hskip 3 pt form of eigenvalue $\frac 1 4 + t^2$ and weight $k$,
\begin{align}
\nonumber J _F (x) & = - \frac \pi {\cosh (\pi t)} \left( Y_{ 2i t} (4 \pi \sqrt x) + Y_{- 2i t} (4 \pi \sqrt x) \right) \\
\nonumber & = \frac { \pi i} {\sinh (\pi t)} \left( J_{ 2i t} (4 \pi \sqrt x) - J_{- 2i t} (4 \pi \sqrt x) \right) \\
\label{1eq: n=2, Maass form, Bessel functions, k even}  & = \pi i \lp e^{-\pi t } H^{(1)}_{ 2i t} (4\pi \sqrt x) - e^{ \pi t} H^{(2)}_{ 2i t} (4 \pi \sqrt x) \rp, \\ 
\nonumber J_F (- x) & = 4 \cosh (\pi t) K_{2 i t} (4 \pi \sqrt x) \\
\nonumber & = \frac { \pi i} {\sinh (\pi t)} \left( I_{ 2i t} (4 \pi \sqrt x) - I_{- 2i t} (4 \pi \sqrt x) \right), \hskip 10 pt x > 0,
\end{align}
for $k$ even,
\begin{align}
\nonumber J _F (x) & = - \frac \pi {\sinh (\pi t)} \left( Y_{ 2i t} (4 \pi \sqrt x) - Y_{- 2i t} (4 \pi \sqrt x) \right) \\
\nonumber & =  \frac { \pi i} {\cosh (\pi t)} \left( J_{ 2i t} (4 \pi \sqrt x) + J_{- 2i t} (4 \pi \sqrt x) \right) \\
\label{1eq: n=2, Maass form, Bessel functions, k odd} & = 
\pi i \lp e^{-\pi t } H^{(1)}_{ 2i t} (4\pi \sqrt x) + e^{ \pi t} H^{(2)}_{ 2i t} (4 \pi \sqrt x) \rp \\
\nonumber  J_F (- x) &  = 4 \sinh (\pi t) K_{2 i t} (4 \pi \sqrt x) \\
\nonumber & = \frac { \pi i} {\cosh (\pi t)} \left(I_{ 2i t} (4 \pi \sqrt x) - I_{- 2i t} (4 \pi \sqrt x) \right), \hskip 10 pt x > 0,
\end{align}
for $k$ odd \footnote{For this case there are two insignificant typos in \cite[Proposition 1]{Harcos-Michel}.},
and if $F$ is a holomorphic cusp form of weight $k$,
\begin{equation}\label{1eq: n=2, holomorphic form, Bessel functions}
J_F (x) = 2 \pi i^k J_{k-1} (4 \pi \sqrt x), \hskip 10 pt J_F (-x) = 0, \hskip 10 pt x > 0.
\end{equation}
Thus the integral kernel $J _F$ has an expression in terms of Bessel functions, where, in standard notation, $J_{\nu}$, $Y_\nu$, $H^{(1)}_\nu$, $H^{(2)}_\nu$, $I_\nu$ and $K_\nu$ are the various Bessel functions (see for instance \cite{Watson}). 
Here,  the following connection formulae {\rm (\cite[3.61 (3, 4, 5, 6), 3.7 (6)]{Watson})} have been applied  in  \eqref{1eq: n=2, Maass form, Bessel functions, k even} and \eqref{1eq: n=2, Maass form, Bessel functions, k odd},
\begin{align}
& Y_\nu (x) = \frac {J_\nu (x) \cos (\pi \nu) - J_{- \nu} (x)}{\sin ( \pi \nu)}, \hskip 10 pt Y_{-\nu} (x) = \frac {J_\nu (x) - J_{- \nu} (x) \cos (\pi \nu)}{\sin ( \pi \nu)},\\
& \label{2eq: connection formulae}
H^{(1)}_\nu (x) = \frac {J_{-\nu} (x) - e^{- \pi i \nu} J_\nu (x) }{i \sin ( \pi \nu)}, \hskip 14 pt H^{(2)}_\nu (x) = \frac { e^{\pi i \nu} J_\nu (x) - J_{-\nu} (x)}{i \sin ( \pi \nu)},\\
& \label{1eq: connection formula K} K_{\nu} (x) =  \frac {\pi \lp I_{-\nu} (x) - I_\nu (x) \rp} {2 \sin (\pi \nu)}.
\end{align}
The theory of Bessel functions has been extensively studied since the early 19th century, and we refer the reader to Watson's beautiful book \cite{Watson} for an encyclopedic treatment. 

For $n \geq 3$, Hankel transforms are formulated in Miller and Schmid \cite{Miller-Schmid-2006, Miller-Schmid-2009}, given that $(\ulambda, \udelta)$ is a certain parameter of a cuspidal $\GL_n(\BZ)$-automorphic representation of $\GL_n (\BR)$. It is the archimedean ingredient that relates the weight functions on two sides of the identity in the \Voronoi summation formula for $\GL_n( \BZ) $. For $n=1, 2$ the Poisson and the \Voronoi summation formula are also interpreted from their perspective in \cite{Miller-Schmid-2004-1}.

Using the global theory of $\GL_n \times \GL_1$-Rankin-Selberg $L$-functions, Inchino and Templier \cite{Ichino-Templier} extend Miller and Schmid's work and prove the \Voronoi summation formula for any irreducible cuspidal automorphic representation of $\GL_n$ over an arbitrary number field for $n \geq 2$. 
According to \cite{Ichino-Templier}, the two defining identities \eqref{1eq: Hankel transform identity} of the associated Hankel transform follow from renormalizing the corresponding local functional equations of the $\GL_n \times \GL_1$-Rankin-Selberg zeta integrals over $\BR$. 

\subsubsection{Bessel kernels}

In the case $n \geq 3$, when applying the \Voronoi summation formula, it might have been realized by many authors that, similar to (\ref{1eq: n=1, Hankel}, \ref{1eq: n=2, Hankel}), Hankel transforms of rank $n$ should also admit integral kernels, that is,
\begin{equation*} 
\Upsilon (x) = \int_{\BR ^\times} \upsilon (y) J_{(\ulambda, \udelta)} (xy ) d y.
\end{equation*} 
We shall call $J_{(\ulambda, \udelta)}$ the \textit{{\rm ({\it fundamental})} Bessel kernel of index $(\ulambda, \udelta)$}.

Actually, it will be seen in \S \ref{sec: Definition of Bessel functions} that an  expression of $J_{(\ulambda, \udelta)} (\pm x)$, $x \in \BR_+ = (0, \infty)$, in terms of   certain Mellin-Barnes type integrals involving the Gamma function (see (\ref{2eq: Bessel kernel}, \ref{2eq: Bessel kernel as the inverse Mellin transform of the gamma factor})) may be easily derived from the first expression \eqref{1eq: Psi defined by Gamma functions} of the Hankel transform of index $(\ulambda, \udelta)$. Moreover, the analytic continuation of $J_{(\ulambda, \udelta)} (\pm x) $ from $ \BR_+ $ onto the Riemann surface $\BU$, the universal cover of $\BC \smallsetminus \{ 0 \}$, can be realized as a Barnes type integral via modifying the integral contour of a  Mellin-Barnes type integral (see Remark {\rm \ref{rem: Barnes integral}}). In the literature, we have seen applications of the asymptotic expansion of $J_{(\ulambda, \udelta)}  (\pm x)$ obtained from applying Stirling's asymptotic formula  of the Gamma function to the Mellin-Barnes type integral (see Appendix \ref{appendix: asymptotic}). There are however two limitations of this method. Firstly, it is {\it only}  applicable when $\ulambda$ is regarded  as fixed constant and hence the dependence on $\ulambda$ of the error term can not be clarified. Secondly, it is {\it not} applicable to a Barnes type integral and therefore the domain of the asymptotic expansion can not be extended from $\BR_+$.  
In this direction from \eqref{1eq: Psi defined by Gamma functions}, it seems that we can not proceed any further.

The novelty of this article is an approach to Bessel kernels starting from the second expression (\ref{1eq: generalized Fourier transform}) of  Hankel transforms. This approach is more accessible,  at least in symbolic notions, in view of the simpler form of (\ref{1eq: generalized Fourier transform}) compared to \eqref{1eq: Psi defined by Gamma functions}. Once we can make sense of the symbolic notions in (\ref{1eq: generalized Fourier transform}), some well-developed methods from analysis and differential equations may be exploited so that we are able to understand   Bessel kernels to a much greater extent. 


\subsection{Outline of article}


\subsubsection{Bessel functions and their formal integral representations}

First of all,  in \S \ref{sec: Definition of Bessel functions}, we introduce the   \textit{Bessel function  $J (x; \usigma, \ulambda)$ of indices $\ulambda \in \BL^{n-1}$ and $\usigma \in \{+, -\}^n$}. It turns out that the Bessel kernel  $J_{(\ulambda, \udelta)} (\pm x)$ can be formulated as a signed sum of   $J \big(2 \pi x^{\frac 1 n}; \usigma, \ulambda \big)$, $x \in \BR_+$.  Our task is therefore understanding each Bessel function  $J (x; \usigma, \ulambda)$. 

In \S \ref{sec: formal integral representation of the Bessel functions}, with some manipulations on the Fourier type expression (\ref{1eq: generalized Fourier transform}) of   the Hankel transform of index $(\ulambda, \udelta)$ in a symbolic manner,   we obtain a \textit{formal} integral representation of the Bessel function $J (x; \usigma, \ulambda)$. If we define $\unu = (\nu_1, ..., \nu_{n-1}) \in \BC^{n-1}$ by $\nu_{l     } = \lambda_l      - \lambda_n$, with $l      = 1, ..., n-1$, then the formal integral is given by
\begin{equation}\label{1eq: formal integral J nu (x; sigma)}
J_{\unu} (x; \usigma) = \int_{\BR_+^{n-1}} \left(\prod_{l      = 1}^{n-1} t_l     ^{ \nu_l      - 1} \right) e^{i x \left(\varsigma_n t_1 ... t_{n-1} + \sum_{l      = 1}^{n-1} \varsigma_l      t_l     \- \right)} d t_{n-1} ... dt_1.
\end{equation}

Justification of this formal integral representation is the main subject of \S \ref{sec: Rigorous interpretations of formal integral representation} and \S \ref{sec: Relating Bessel functions}. For this, we partition the formal integral $J_{\unu} (x; \usigma)$ according to some partition of unity on $\BR^{n-1}_+$, and then 
repeatedly apply \textit{two} kinds of partial integration operators on each resulting integral. In this way, $J_{\unu} (x; \usigma)$ can be transformed into a finite sum of absolutely convergent multiple integrals. This sum of integrals is regarded as the rigorous definition of  $J_{\unu} (x; \usigma)$. However, the simplicity of the expression (\ref{1eq: formal integral J nu (x; sigma)}) is sacrificed after these technical procedures. Furthermore, it is shown that
\begin{equation}\label{1eq: J(x; sigma; lambda) = J nu (s; sigma)}
J (x; \usigma, \ulambda) = J_{\unu} (x; \usigma),
\end{equation}
where $ J_{\unu} (x; \usigma)$ on the right is now rigorously understood.

\subsubsection{Asymptotics via stationary phase}

In \S \ref{sec: Bessel functions of K-type and H-Bessel functions}, we either adapt  techniques or   apply  results from {the method of stationary phase} to study the asymptotic behaviour of each integral in the rigorous definition of $J_{\unu} (x; \usigma)$, and hence  $J_{\unu} (x; \usigma)$ itself, for large argument. Even in the classical case $n=2$, our method is entirely new, as the coefficients in the asymptotic expansions are formulated in a way that is quite different from what is known in the literature (see \S \ref{sec: remark on the coefficients in the asymptotics}).

When all the components of $\usigma$ are identically $\pm$, we denote $J (x; \usigma, \ulambda)$, respectively  $J_{\unu} (x; \usigma)$, by $H^\pm (x; \ulambda)$, respectively $H_{\unu}^\pm (x )$, and call it an \textit{$H$-Bessel function}\footnote{If a statement or a formula includes $\pm$ or $\mp$, then it should be read with $\pm$ and $\mp$ simultaneously replaced by  either  $+$ and $-$ or $-$ and $+$.}. 
This pair of $H$-Bessel functions will be of paramount significance in our treatment.

It is shown that $H^{\pm} (x; \ulambda) = H_{\unu}^\pm (x )$ admits an analytic continuation from $\BR_+$ onto the half-plane $\BH^{\pm} = \left\{ z \in \BC \smallsetminus \{0\} : 0 \leq \pm \arg z \leq \pi \right\}$. We have the asymptotic expansion
\begin{equation}\label{1eq: asymptotic expansion 1}
\begin{split}
H^{\pm}  (z; \ulambda) = n^{- \frac 1 2} & (\pm 2 \pi i)^{ \frac { n-1} 2} e^{ \pm i n z} z^{ - \frac { n-1} 2} \\
& \lp \sum_{m=0}^{M-1} ( \pm i )^{ - m} B_{m} (\ulambda) z^{- m} + O_{\mathfrak R, M, n} \left( \mathfrak C^{2 M} |z|^{-M + \frac { n-1} 2} \right) \rp,
\end{split}
\end{equation}
for all $z \in \BH^\pm$ such that $ |z| \geq \fC$, where $\mathfrak C =  \max  \left\{ | \lambda_l      | \right \} + 1$, $\mathfrak R = \max  \left\{ |\Re \lambda_l      | \right \}$, $M \geq 0$, $B_{m} (\ulambda) $ is a certain symmetric polynomial in $\ulambda$ of  degree $2m $, with $B_{0} (\ulambda) = 1$. In particular, these two $H$-Bessel functions oscillate and decay proportionally to $x^{-\frac { n-1} 2}$ on $\BR_+$. 

All the other Bessel functions are called \textit{$K$-Bessel functions} and are shown to be Schwartz functions at infinity. 

\subsubsection{Bessel equations}

The differential equation, namely \textit{Bessel equation}, satisfied by the Bessel function $J (x; \usigma, \ulambda)$ is discovered in \S \ref{sec: Recurrence relations and differential equations of the Bessel functions}.  

Given $\ulambda \in \BL^{n-1}$, there are exactly two Bessel equations
\begin{equation}\label{1eq: Bessel equations}
\sum_{j = 1}^{n} V_{n, j} (\ulambda) x^{j} w^{(j)} + \left( V_{n, 0} (\ulambda) - \varsigma (in)^n  x^n \right) w  = 0, \hskip 10 pt \varsigma \in \{+, - \},
\end{equation}
where $V_{n, j} (\ulambda)$ is some explicitly given symmetric polynomial in $\ulambda$ of degree $n-j$. We call $\varsigma$ the \textit{sign} of the Bessel equation (\ref{1eq: Bessel equations}). $J (x; \usigma, \ulambda)$ satisfies the Bessel equation of sign $S_n(\usigma) = \prod_{l      = 1}^n \varsigma_{l     }$.

The entire \S \ref{sec: Bessel equations} is devoted to the study of Bessel equations. Let $\BU$ denote the Riemann surface   associated with $\log z$, that is, the universal cover of $\BC \smallsetminus \{ 0 \}$. Replacing $x$ by $z$ to stand for complex variable in the Bessel equation (\ref{1eq: Bessel equations}),   the domain is extended from $\BR_+$ to $\BU$. According to the theory of linear ordinary differential equations with analytic coefficients, $J (x; \usigma, \ulambda)$ admits an analytic continuation onto $\BU$.

Firstly, since zero is a regular singularity, the Frobenius method may be exploited to find a solution $J_{l     } (z; \varsigma, \ulambda)$ of (\ref{1eq: Bessel equations}), for each $l      = 1, ..., n$, defined by the following series,
\begin{equation*}
J_{l     } (z; \varsigma, \ulambda) = \sum_{m=0}^\infty \frac { (\varsigma i^n)^m  z^{ n (- \lambda_{l      } + m)} } { \prod_{k = 1}^n \Gamma \lp  { \lambda_{ k } - \lambda_{l     }}  + m + 1 \rp}.
\end{equation*}
$J_{l     } (z; \varsigma, \ulambda)$ are called  {\it Bessel functions of the first kind}, since
they generalize the Bessel functions $J_{\nu} (z)$ and the modified Bessel functions $I_{\nu} (z)$ of the first kind.

It turns out that each $J (z; \usigma, \ulambda)$ may be expressed in terms of $J_{l     } \left(z; S_n(\usigma), \ulambda\right)$. This leads to the following connection formula
\begin{equation}\label{1eq: connection formula 1}
J(z ; \usigma, \ulambda) = e \lp \pm \frac {\sum_{l      \in L_ \mp (\usigma)} \lambda_{l      }} 2 \rp H^\pm \Big(e^{\pm \pi i \frac { n_{\mp} (\usigma)} n} z; \ulambda \Big),
\end{equation}
where $L_\pm (\usigma) = \{ l      : \varsigma_l      = \pm \}$ and $n_\pm (\usigma) = \left| L_\pm (\usigma) \right|$. Thus the Bessel function $J(z ; \usigma, \ulambda)$ is determined up to a constant by the pair of integers $(n_+ (\usigma), n_- (\usigma) )$, called the \textit{signature} of 
$J(z ; \usigma, \ulambda)$. 

Secondly, $\infty$ is an irregular singularity of rank one. The formal solutions at infinity serve as the asymptotic expansions of some actual solutions of Bessel equations. 

Let $\xi$ be an $n$-th root of $\varsigma 1$. There exists a unique formal solution $\widehat J (z; \ulambda; \xi)$ of the Bessel equation of sign $\varsigma$ in the following form
\begin{equation*} 
\widehat J (z; \ulambda; \xi) = e^{i n \xi z} z^{- \frac { n-1} 2} \sum_{m=0}^\infty B_m (\ulambda; \xi) z^{-m},
\end{equation*}
where $B_m (\ulambda; \xi)$ is a symmetric polynomial in $\ulambda$ of degree $2 m$, with $B_0 (\ulambda; \xi) = 1$. The coefficients of  $B_m (\ulambda; \xi)$ depend only on $m$, $\xi$ and $n$. There exists a {\it unique} solution $J (z; \ulambda; \xi)$ of the Bessel equation of sign $\varsigma$ that possesses $\widehat J (z; \ulambda; \xi)$ as its asymptotic expansion on the sector
\begin{equation*}
\BS_{\xi } = \left\{ z \in \BU : \left| \arg z - \arg ( i \xoverline \xi) \right|  < \frac {\pi} n \right \},
\end{equation*}
or any of its open subsector.

The study of the theory of asymptotic expansions for ordinary differential equations can be traced back to Poincar\'e. There are abundant references on this topic, for instance,  \cite[Chapter 5]{Coddington-Levinson}, \cite[Chapter III-V]{Wasow} and \cite[Chapter 7]{Olver}. However, the author is not aware of any error analysis in the index aspect 
in the literature except for differential equations of second order in \cite{Olver}. Nevertheless, with some effort, a very satisfactory error bound is attainable.

For $0 < \vartheta < \frac 1 2\pi $ define the sector
\begin{equation*}
\BS'_{\xi } (\vartheta) = \left\{ z \in \BU : \left| \arg z - \arg ( i \xoverline \xi) \right|  < \pi + \frac {\pi} n - \vartheta \right \}.
\end{equation*}
The following asymptotic expansion is established in \S \ref{sec: Error Bounds for the asymptotic expansions},
\begin{equation}\label{1eq: asymptotic expansion 2}
J (z; \ulambda; \xi) = e^{i n \xi z} z^{- \frac { n-1} 2} \lp \sum_{m=0}^{M-1} B_m (\ulambda; \xi) z^{-m} + O_{M, n} \lp \fC^{2 M} |z|^{- M } \rp \rp
\end{equation}
for all $z \in \BS'_{\xi} (\vartheta)$ with $ |z| \ggg_{M, \vartheta, n} \fC^2 $. 

For a $2n$-th root of unity $\xi$, $J( z; \ulambda; \xi)$ is called a \textit{Bessel function of the second kind}. We have the following formula that relates all the the Bessel functions of the second kind to either $J (z; \ulambda;   1)$ or $J (z; \ulambda; - 1)$  upon rotating the argument by a $2n$-th root of unity,
\begin{equation}\label{1eq: connection formula 2}
J( z; \ulambda; \xi) = (\pm \xi)^{ \frac {n-1} 2} J(\pm \xi z; \ulambda; \pm 1).
\end{equation}

\subsubsection{\texorpdfstring{Connections between $J (z; \usigma, \ulambda)$ and $J(z; \ulambda; \xi )$}{Connections between $J (z; \varsigma, \lambda)$ and $J(z; \lambda; \xi )$}}

Comparing the asymptotic expansions of $H^{\pm} (z; \ulambda)$ and $J (z; \ulambda; \pm 1)$ in (\ref{1eq: asymptotic expansion 1}) and \eqref{1eq: asymptotic expansion 2}, 
we obtain the identity
\begin{equation}\label{1eq: identity H}
H^{\pm} (z; \ulambda) = n^{- \frac 1 2} (\pm 2 \pi i)^{ \frac { n-1} 2} J (z; \ulambda; \pm 1).
\end{equation}
It follows from (\ref{1eq: connection formula 1}) and (\ref{1eq: connection formula 2}) that 
\begin{equation*}
\begin{split}
J (z; \usigma, \ulambda) = \frac { (\mp 2\pi i)^{\frac { n-1} 2} } {\sqrt n} e \lp \pm \frac {(n-1) n_\pm(\usigma)} {4 n} \mp \frac {\sum_{l      \in L_\pm (\usigma)} \lambda_{l      }} 2 \rp J \Big(z; \ulambda; \mp e^{\mp \pi i \frac { n_\pm(\usigma) } n} \Big).
\end{split}
\end{equation*}
Thus (\ref{1eq: asymptotic expansion 2}) may be applied to improve the  error estimate in the asymptotic expansion \eqref{1eq: asymptotic expansion 1} of the $H$-Bessel function $H^{\pm} (z; \ulambda)$ when  $ |z| \ggg_{M, n} \fC^2$  and also to show the exponential decay of $K$-Bessel functions on $\BR_+$.

\subsubsection{\texorpdfstring{Connections between $J_{l     } (z; \varsigma, \ulambda) $ and $ J (z; \ulambda; \xi) $}{Connections between $J_{l     } (z; \varsigma, \lambda) $ and $ J (z; \lambda; \xi) $}} The identity \eqref{1eq: identity H} also yields connection formulae between the two kinds of Bessel functions,  in terms of a certain Vandermonde matrix and its inverse. \\

\begin{acknowledgement}
The author is grateful to his advisor, Roman Holowinsky, who brought him to the area of analytic number theory, gave him much enlightenment, guidance and encouragement. The author would like to thank James Cogdell, Ovidiu Costin, Stephen Miller and Zhilin Ye for valuable comments and helpful discussions.
\end{acknowledgement}


\section{Preliminaries on Bessel functions} \label{sec: Bessel functions}

In \S \ref{sec: Definition of Bessel functions} and \ref{sec: formal integral representation of the Bessel functions}, we shall introduce the Bessel function $J (x ; \usigma, \ulambda)$, with $\usigma \in \{+, - \}^n$ and $\ulambda \in \BL^{n-1}$. Two expressions of $J (x ; \usigma, \ulambda)$ arise from the two formulae \eqref{1eq: Psi defined by Gamma functions} and \eqref{1eq: generalized Fourier transform} of  the Hankel transform of index $(\ulambda, \udelta)$. The first is a Mellin-Barnes type contour integral and the second is a {formal} multiple integral. In \S \ref{sec: Classical Bessel function} and \ref{sec: special example}, some examples of  $J (x ; \usigma, \ulambda)$ are provided for the purpose of illustration.\\


Let $\upsilon \in \mathscr S (\BR^\times)$ be a Schwartz function on $\BR^\times$. Without loss of generality, we assume  $\upsilon( - y) = (-)^\eta \upsilon (y)$, with $\eta \in \BZ/2 \BZ$.

\subsection{\texorpdfstring{The definition of the Bessel function $J (x ; \usigma, \ulambda)$}{The definition of the Bessel function $J (x ; \varsigma, \lambda)$}} \label{sec: Definition of Bessel functions}
We start with reformulating (\ref{1eq: definition of G delta}) as
\begin{equation*} 
G_\delta (s) = 
(2 \pi)^{-s} \Gamma (s) \left (e\left ( \frac s 4 \right) + (-)^\delta e\left (- \frac s 4 \right) \right).
\end{equation*}
Inserting this formula of $G_{\delta} $ into 
(\ref{1eq: Psi defined by Gamma functions}),  $\Upsilon (x)$ then splits as follows
\begin{equation}\label{2eq: signed sum of Psi (x; sigma)}
\Upsilon (x) = \sgn (x)^{\eta} \sum_{ \usigma \in \{+, - \}^n} \left( \prod_{l      = 1}^n \varsigma_l      ^{\delta_l      + \eta} \right) \Upsilon (|x|; \usigma), 
\end{equation}
with $\usigma = (\varsigma_1, ..., \varsigma_n)$, where 
\begin{equation}\label{2eq: Psi (x; sigma) double integral of Gamma functions}
\Upsilon (x; \usigma) = \frac 1 {2 \pi i}  \int_{(\sigma )} \int_{0}^\infty \upsilon (y) y^{ - s } d y \cdot G(s; \usigma, \ulambda)  ((2 \pi)^n x)^{ - s} ds, \hskip 10 pt x \in \BR_+,
\end{equation}
and
\begin{equation} \label{2eq: definition of G(s;  sigma, lambda)}
G (s; \usigma, \ulambda) = \prod_{l      = 1}^{n} \Gamma (s - \lambda_l      ) e \left(\frac {\varsigma_l      (s - \lambda_l      )} 4 \right).
\end{equation}

Since all the derivatives of $\upsilon $ rapidly decay at both zero and infinity, repeating partial integrations yields the bound
\begin{equation*}
\int_{0}^\infty \upsilon (y) y^{ - s } d y \lll_{\Re s, M, \upsilon} (|\Im s|+1)^{-M},
\end{equation*}
for any nonnegative integer $M $. Hence the iterated double integral in (\ref{2eq: Psi (x; sigma) double integral of Gamma functions}) is convergent due to  Stirling's formula.

Choose  $ \rho <  \frac 1 2 - \frac 1 n $ so that $ \sum_{l      = 1}^n \lp \rho - \Re  \lambda_l      - \frac 1 2 \rp <  -1$. 
Without passing through any pole of $G (s; \usigma, \ulambda)$, we shift the vertical line  $ (\sigma)$ to a contour $\EuScript C$  that starts from $\rho - i \infty$, ends at $\rho + i \infty$, and remains vertical at infinity. After this contour shift, the double integral in (\ref{2eq: Psi (x; sigma) double integral of Gamma functions}) becomes absolutely convergent by Stirling's formula. Changing the order of integration is therefore legitimate and yields
\begin{equation}\label{2eq: Psi (x; sigma) as Hankel transform}
\Upsilon (x; \usigma) = \int_{0}^\infty \upsilon (y) J \big( 2\pi (xy)^{\frac 1 n}; \usigma, \ulambda\big) d y,
\end{equation}
with
\begin{equation}\label{2eq: definition of J (x; sigma)}
J (x ; \usigma, \ulambda) = \frac 1  {2 \pi i}\int_{\EuScript C} G(s; \usigma, \ulambda) x^{- n s} d s.
\end{equation}
For $\ulambda \in \BL^{n-1}$ and $\usigma \in \{+, - \}^n$, the function $J (x; \usigma, \ulambda)$ defined by (\ref{2eq: definition of J (x; sigma)}) is called a \textit{Bessel function} and the integral in \eqref{2eq: definition of J (x; sigma)}  a {\it Mellin-Barnes type integral}. We view $J \big(x^{\frac 1 n}; \usigma, \ulambda\big)$ as the inverse Mellin transform of $ G(s; \usigma, \ulambda)$. 

Suitably choosing the integral contour $\EuScript C$, it may be verified that $J (x; \usigma, \ulambda)$ is a smooth function of $x$ and is analytic with respect to $\ulambda$.

\begin{rem}
The contour of integration  $ (\sigma)$ does not need modification if the components of $\usigma$ are not identical. For further discussions of the integral in the definition {\rm (\ref{2eq: definition of J (x; sigma)})} of $J (x; \usigma, \ulambda)$ see Remark {\rm \ref{rem: Barnes integral}}. 
\end{rem}

\begin{rem}\label{rem: Bessel kernel}
We have
\begin{equation}\label{2eq: Hankel transform, with Bessel kernel}
\Upsilon (x) = \int_{\BR ^\times} \upsilon (y) J_{(\ulambda, \udelta)} (xy ) d y, \hskip 10 pt x \in \BR^\times,
\end{equation}
 for any $\upsilon \in \mathscr S (\BR ^\times)$, where the Bessel kernel $J_{(\ulambda, \udelta)}$ is given by
\begin{equation}\label{2eq: Bessel kernel}
\begin{split}
J_{(\ulambda, \udelta)} \lp \pm x \rp & = \frac 1 2 \sum_{\delta \in \BZ/2\BZ} (\pm)^{ \delta} \sum_{\usigma \in \{+, -\}^n}  \left( \prod_{l      = 1}^n \varsigma_l      ^{\delta_l      + \delta} \right) J \big(2 \pi x^{\frac 1 n}; \usigma, \ulambda \big) \\
& = \sum_{\sstyle \usigma \in \{+, -\}^n \atop \sstyle \prod  \varsigma_{l} = \pm } \left( \prod_{l      = 1}^n \varsigma_l      ^{\delta_l   } \right) J \big(2 \pi x^{\frac 1 n}; \usigma, \ulambda \big),  
\end{split}
\end{equation}
with $x \in \BR_+$.
Moreover,
\begin{equation}\label{2eq: Bessel kernel as the inverse Mellin transform of the gamma factor}
J_{(\ulambda, \udelta)} \lp x \rp = \sum_{\delta \in \BZ/2\BZ} \frac { \sgn (x)^\delta} {4 \pi i} \int_{\EuScript C}
\left( \prod_{l      = 1}^n G_{\delta_{l     } + \delta} (s - \lambda_l     ) \right) |x|^{- s} d s.
\end{equation}
\end{rem}

\subsection{\texorpdfstring{The formal integral representation of  $J (x ; \usigma, \ulambda)$}{The formal integral representation of  $J (x ; \varsigma, \lambda)$}} \label{sec: formal integral representation of the Bessel functions}

In this section, we assume $n \geq 2$. Since we shall manipulate the Fourier type integral transform (\ref{1eq: generalized Fourier transform}) only in a symbolic manner, the restrictions on the index $\ulambda$ that guarantee the convergence of the iterated integral in (\ref{1eq: generalized Fourier transform}) will not be imposed here.

With the parity condition on the weight function $\upsilon$,  (\ref{1eq: generalized Fourier transform}) may be written as
\begin{equation}\label{1eq: Upsilon = sum, 2}
\begin{split}
\Upsilon (x) = \frac { \sgn (x)^{\eta}} {|x|} \sum_{ \usigma \in \{+, - \}^n} & \left( \prod_{l      = 1}^n \varsigma_l      ^{\delta_l      + \eta} \right) \\
& \int_{\BR_+^n} \upsilon \left(\frac { x_1 ... x_n} {|x|} \right) \left( \prod_{l      = 1}^n x_l     ^{- \lambda_l     } e( \varsigma_l      x_l     ) \right) dx_n d x_{n-1} ... d  x_1.
\end{split}
\end{equation}
Comparing \eqref{1eq: Upsilon = sum, 2} with (\ref{2eq: signed sum of Psi (x; sigma)})\footnote{To justify our comparison, we use the fact that the associated $2^n \times 2^n$ matrix 
	is equal to the $n$-th tensor power of $  \begin{pmatrix}
	1 & (-1)^{\eta} \\
	1 & (-1)^{1 + \eta}
	\end{pmatrix}$ and hence is invertible.}, we arrive at 
\begin{align*}
\Upsilon (x; \usigma) = \frac {1} {|x|} \int_{\BR_+^n} \upsilon \left(\frac { x_1 ... x_n} {|x|} \right) \left( \prod_{l      = 1}^n x_l     ^{- \lambda_l     } e( \varsigma_l      x_l     ) \right) dx_n d x_{n-1} ... d  x_1.
\end{align*}
The change of variables $x_{n} = |x| y (x_1 ... x_{n-1})\- $, $ x_{l     } = y_{l     }\- $, $l      = 1, ..., n - 1$, turns this further into
\begin{equation}\label{2eq: Psi (x; sigma) another integral representation}
\begin{split}
\Upsilon (x; \usigma) = \int_{\BR_+^n} \upsilon \left( y \right) (x y)^{- \lambda_n } & \left( \prod_{l      = 1}^{n-1} y_l     ^{ \lambda_l      - \lambda_n - 1} \right) \\
& e \left( \varsigma_n  xy y_1 ... y_{n-1} + \sum_{l      = 1}^{n-1} \varsigma_l      y_l     \- \right) dy d y_{n-1} ... d y_1.
\end{split}
\end{equation}
Comparing now \eqref{2eq: Psi (x; sigma) another integral representation} with (\ref{2eq: Psi (x; sigma) as Hankel transform}), if one {\it formally} changes the order of the integrations, which is \textit{not} permissible since the integral is \textit{not} absolutely convergent, then $J(x; \usigma, \ulambda)$ can be expressed as a symbolic integral as below,
\begin{equation*}
\begin{split}
  \ J ( 2\pi x ; \usigma, \ulambda)  
=  x^{- n \lambda_n } \int_{\BR_+^{n-1}}  \left( \prod_{l      = 1}^{n-1} y_l     ^{ \lambda_l      - \lambda_n - 1} \right) e \left( \varsigma_n  x^n y_1 ... y_{n-1} + \sum_{l      = 1}^{n-1} \varsigma_l      y_l     \- \right) d y_{n-1} ... dy_1.
\end{split}
\end{equation*}
Another change of variables $y_{l     } = t_{l     } x\-$, along with the assumption  $\sum_{l      = 1}^{n} \lambda_l      = 0$,  yields
\begin{equation}\label{2eq: formal integral representation of J (x; sigma)}
J ( x ; \usigma, \ulambda) = \int_{\BR_+^{n-1}} \left(\prod_{l      = 1}^{n-1} t_l     ^{ \lambda_l      - \lambda_n - 1} \right) e^{i x \left(\varsigma_n t_1 ... t_{n-1} + \sum_{l      = 1}^{n-1} \varsigma_l      t_l     \- \right)} d t_{n-1} ... dt_1.
\end{equation}
The above integral is {\it not} absolutely convergent and will be referred to as the \textit{formal integral representation of $J ( x ; \usigma, \ulambda)$}.

\begin{rem}
Before realizing its connection with the Fourier type transform   {\rm (\ref{1eq: generalized Fourier transform})}, the formal integral representation  of $J ( x ; \usigma, \ulambda)$ was derived by the author from {\rm (\ref{1eq: Psi defined by Gamma functions})} based on a symbolic application of the product-convolution principle of the Mellin transform together with the following formula {\rm (\cite[{\bf 3.764}]{G-R})} 
\begin{equation}\label{2eq: Gamma and exponential via Mellin}
\Gamma (s) e \lp \pm \frac  s 4 \rp = \int_0^\infty e^{\pm ix} x^{s } d^\times x, \hskip 10 pt 0 < \Re s < 1.
\end{equation}
Though not specified, this principle is implicitly suggested in Miller and Schmid's work, especially, \cite[Theorem 4.12, Lemma 6.19]{Miller-Schmid-2004} and \cite[(5.22, 5.26)]{Miller-Schmid-2006} {\rm ({\it see also \cite[\S 5]{Qi2}})}.
\end{rem}


\subsection{The classical cases}\label{sec: Classical Bessel function}

\subsubsection{The case $n = 1$}

\begin{prop} \label{prop: n=1}
	Suppose $n = 1$. Choose the contour $\EC$  as in \S {\rm \ref{sec: Definition of Bessel functions}};  $\EC$ starts from $\rho - i \infty$ and ends at $\rho + i \infty$, with  $\rho < - \frac 1 2$, and all the nonpositive integers lie on the left side of  $\EC$. We have
	\begin{equation}\label{2eq: n = 1, Mellin inversion}
	e^{\pm i x} = \frac 1 {2\pi i} \int_{\EC} \Gamma (s) e \lp \pm \frac s 4 \rp x^{-s} d s.
	\end{equation}
	Therefore
	\begin{equation*}
	J (x; \pm, 0) = e^{\pm i x}.
	\end{equation*}
\end{prop}
\begin{proof}
	Let $\Re z > 0$. For $\Re s > 0$, we have the formula
	\begin{equation*}
	\Gamma (s) z^{-s} = \int_0^\infty e^{- z x} x^{s } d^\times x,
	\end{equation*}
	where the integral is absolutely convergent. The Mellin inversion formula  yields 
	\begin{equation*}
	e^{- z x} = \frac 1 {2\pi i} \int_{(\sigma)} \Gamma (s) z^{-s} x^{-s} d s, \hskip 10 pt \sigma > 0.
	\end{equation*} 
	Shifting the contour of integration from $(\sigma)$ to $\EC$, one sees that
	\begin{equation*}
	e^{- z x} = \frac 1 {2\pi i} \int_{\EC} \Gamma (s) z^{-s} x^{-s} d s.
	\end{equation*}
	Choose $z = e^{\mp \lp \frac {1 } 2 \pi - \epsilon \rp i}$, $ \pi > \epsilon > 0$. In view of  Stirling's formula, the convergence of the integral above is uniform in $\epsilon$. Therefore, we obtain \eqref{2eq: n = 1, Mellin inversion} by letting $\epsilon \ra 0$.
\end{proof}

\begin{rem}
	 Observe that the integral in \eqref{2eq: Gamma and exponential via Mellin} is only conditionally convergent, the Mellin inversion formula does not apply in the rigorous sense. Nevertheless, \eqref{2eq: n = 1, Mellin inversion} should be view as the Mellin inversion of \eqref{2eq: Gamma and exponential via Mellin}.
\end{rem}

\begin{rem}
	It follows from the proof of Proposition {\rm \ref{prop: n=1}} that the formula
	\begin{equation}\label{2eq: Mellin n=1}
	e^{- e(a) x } = \frac 1 {2\pi i} \int_{\EC} \Gamma (s) e\lp - a s \rp x^{-s} d s
	\end{equation}
	is valid for any $a \in \left[- \frac 1 4,  \frac 1 4\right]$.
\end{rem}


\subsubsection{The case $n = 2$}

\begin{prop}\label{prop: Classical Bessel functions}
Let $\lambda \in \BC$. Then
\begin{align*}
J (x; \pm, \pm, \lambda, - \lambda) & = \pm \pi i e^{\pm \pi i \lambda} H^{(1, 2)}_{2 \lambda} (2 x), \\
J (x; \pm, \mp, \lambda, - \lambda) & = 2 e^{\mp \pi i \lambda} K_{2 \lambda} (2 x).
\end{align*}
Here $H^{(1)}_{\nu}$ and $H^{(2)}_{\nu}$ are Bessel functions of the third kind, also known as Hankel functions, whereas $K_{\nu}$ is the  modified Bessel function of the second kind, occasionally called the $K$-Bessel function.
\end{prop}

\begin{proof}
The following formulae are derived from \cite[{\bf 6.561} 14-16]{G-R} along with Euler's reflection formula of the Gamma function. 
\begin{equation*} 
\pi \int_0^\infty J_{\nu} (2 \sqrt x) x^{s-1} d x = \Gamma \left( s + \frac \nu 2\right) \Gamma \left( s - \frac \nu 2\right) \sin \left(\pi \left(s - \frac \nu 2 \right) \right)
\end{equation*}
for $- \frac {1} 2 \Re  \nu< \Re s < \frac 1 4$,
\begin{equation*} 
- \pi \int_0^\infty  Y_{\nu} (2 \sqrt x) x^{s-1} d x  = \Gamma \left( s + \frac \nu 2\right) \Gamma \left( s - \frac \nu 2\right) \cos \left(\pi \left(s - \frac \nu 2 \right) \right)
\end{equation*}
for $ \frac {1} 2 |\Re  \nu|< \Re s < \frac 1 4$, and
\begin{equation*} 
2 \int_0^\infty  K_{\nu} (2 \sqrt x) x^{s-1} d x  = \Gamma \left( s + \frac \nu 2\right) \Gamma \left( s - \frac \nu 2\right)
\end{equation*}
for $ \Re s > \frac {1} 2|\Re  \nu|$. For $\Re s$   in the given ranges, these integrals are absolutely convergent.  It follows immediately from the Mellin inversion formula that
\begin{align*}
J (x; \pm, \pm, \lambda, - \lambda) & = \pm \pi i e^{\pm \pi i \lambda} (J_{2 \lambda} (2 x) \pm i Y_{2 \lambda} (2 x)), \hskip 10 pt |\Re \lambda| < \frac 1 4, \\
J (x; \pm, \mp, \lambda, - \lambda) & = 2 e^{\mp \pi i \lambda} K_{2 \lambda} (2 x).
\end{align*}
In view of the analyticity in $\lambda$, the first formula remains valid even if $|\Re \lambda| \geq \frac 1 4$ by the theory of analytic continuation. 
Finally, we conclude the proof by recollecting the formula $H_\nu^{(1, 2)} (x) = J_\nu (x) \pm i Y_\nu (x)$.
\end{proof}

\begin{rem}\label{2rem: n=2}
Let $\lambda = i t$ if $F$ is a Maa\ss  \hskip 3 pt  form of eigenvalue $\frac 1 4 + t^2$ and weight $k$, and let $\lambda = \frac 1 2 {(k - 1)}  $ if $F$ is a holomorphic cusp form of weight $k$. Then $F$ is parametrized by $(\ulambda, \udelta) = (\lambda, - \lambda, k (\mod 2), 0)$ and $J _F = J_{(\ulambda, \udelta)}$. From the formula {\rm (\ref{2eq: Bessel kernel})} of the Bessel kernel, we have 
\begin{align*}
J_{(\ulambda, \udelta)} \lp x \rp & = J (2 \pi \sqrt x; +, +, \lambda, - \lambda) + (-)^k J (2 \pi \sqrt x; - , -, \lambda, - \lambda),\\ 
J_{(\ulambda, \udelta)} \lp - x \rp & = J (2 \pi \sqrt x; +, -, \lambda, - \lambda) + (-)^k J (2 \pi \sqrt x; - , +, \lambda, - \lambda).
\end{align*} 
Thus, Proposition {\rm\ref{prop: Classical Bessel functions}} implies {\rm (\ref{1eq: n=2, Maass form, Bessel functions, k even}, \ref{1eq: n=2, Maass form, Bessel functions, k odd}, \ref{1eq: n=2, holomorphic form, Bessel functions})}.
\end{rem}

When  $x > 0$ and $| \Re  \nu| < 1$, we have the following integral representations of Bessel functions (\cite[6.21 (10, 11), 6.22 (13)]{Watson})
\begin{align*}
H^{(1, 2)}_\nu (x) & = \pm \frac {2 e^{\mp \frac 1 2 \pi i \nu }} {\pi i} \int_0^\infty e^{\pm i x \cosh r} \cosh (\nu r) d r,\\
K_{\nu} (x) & = \frac 1 {\cos \left( \frac {1 } 2 \pi \nu \right)} \int_0^\infty \cos (x \sinh r) \cosh (\nu r) d r.
\end{align*}
The change of variables $t = e^r$ yields
\begin{align*}
 \pm \pi i e^{\pm \frac 1 2 \pi i \nu } H^{(1, 2)}_\nu (2 x)  & = \int_0^\infty t^{\nu - 1} e^{\pm i x (t + t\-)} d t,\\
2 e^{\pm \frac 1 2 \pi i \nu } K_{\nu} (2 x) & =  \int_0^\infty t^{\nu - 1} e^{\pm i x (t - t\-)} d t.
\end{align*}
The integrals in these formulae  are exactly the formal integrals in \eqref{2eq: formal integral representation of J (x; sigma)} in the case $n=2$. They   \textit{conditionally} converge  if  $| \Re  \nu| < 1$, but diverge if otherwise. 

\subsection{A prototypical example} \label{sec: special example}

According to \cite[3.4 (3, 6), 3.71 (13)]{Watson},
\begin{equation*}
J_{\frac 1 2} (x) = \lp \frac 2 {\pi x} \rp^{\frac 1 2} \sin x, \hskip  10 pt
J_{- \frac 1 2} (x) = \lp \frac 2 {\pi x} \rp^{\frac 1 2} \cos x.
\end{equation*}
The connection formulae in \eqref{2eq: connection formulae} (\cite[3.61 (5, 6)]{Watson})
then imply that
\begin{equation*} 
H^{(1)}_{\frac 1 2} (x) = - i \lp \frac 2 {\pi x} \rp^{\frac 1 2} e^{i x}, \hskip 10 pt H^{(2)}_{\frac 1 2} (x) = i  \lp \frac 2 {\pi x} \rp^{\frac 1 2} e^{- i x}.
\end{equation*}
Moreover, \cite[3.71 (13)]{Watson} reads
\begin{equation*}
K_{\frac 1 2} (x) = \lp \frac \pi {2 x} \rp^{\frac 1 2} e^{- x}.
\end{equation*}
Therefore, from the formulae in Proposition \ref{prop: Classical Bessel functions} we have
\begin{equation*}
J \lp x; \pm, \pm, \tfrac 1 4 , - \tfrac 1 4\rp = \lp \frac {\pi  } x \rp^{\frac 1 2} e^{\pm 2 i  x \pm \frac 1 4 \pi i}, \hskip 10 pt J \lp x; \pm, \mp, \tfrac 1 4 , - \tfrac 1 4 \rp = \lp \frac {\pi } x \rp^{\frac 1 2} e^{- 2 x \mp \frac 1 4 \pi i}.
\end{equation*}
These formulae admit generalizations to arbitrary rank.
\begin{prop}
	\label{prop: special example}
	For $\usigma \in \{ +, -\}^n$ we define $L_\pm (\usigma) = \{ l      : \varsigma_l      = \pm \}$ and $n_\pm (\usigma) = \left| L_\pm (\usigma) \right|$.
	Put $\xi (\usigma) = i e^{ {\pi} i \frac {n_- (\usigma) - n_+ (\usigma) } {2n} } =  \mp e^{\mp \pi i \frac { n_{\pm} (\usigma)} n}$. Suppose $\ulambda = \frac 1 n \lp  \frac {n-1} 2, ..., - \frac {n-1} 2 \rp$. 
	Then  
\begin{equation}\label{2eq: special example}
J (x; \usigma, \ulambda) = \frac { c(\usigma)} {\sqrt n} \lp \frac {2\pi} x \rp^{\frac {n-1} 2} e^{ i n \xi (\usigma) x},
\end{equation}
with $c(\usigma) = e \lp \mp \frac {n-1} 8 \mp \frac {n_{\pm} (\usigma) } {2n} \pm \frac 1 {2 n}  \sum_{l      \in L_{\pm} (\usigma)}   {l     } \rp$. 
\end{prop}
\begin{proof}
	Using the multiplication formula of the Gamma function
	\begin{equation}\label{2eq: multiplication theorem}
	\prod_{k = 0}^{n-1} \Gamma \lp s + \frac {k} n \rp = 
	(2 \pi)^{\frac {n-1} 2} n^{\frac 1 2 - n s} \Gamma (n s),
	\end{equation}
	straightforward calculations yield
	\begin{equation*}
	G (s; \usigma, \ulambda) = c_1 (\usigma) (2 \pi)^{\frac {n-1} 2} n^{\frac 1 2 - n \lp s - \frac {n-1} {2n} \rp} \Gamma  \lp n \lp s - \frac {n-1} {2n} \rp\rp e \lp \frac {n_+ (\usigma) - n_- (\usigma) } {4  } \cdot s \rp,
	\end{equation*}
	with $c_1 (\usigma) = e \lp \mp \frac {(n+1) n_{\pm } (\usigma) } {4 n} \pm \frac 1 {2 n}  \sum_{l      \in L_{\pm} (\usigma)}   {l     } \rp$. Inserting this into the contour integral in \eqref{2eq: definition of J (x; sigma)}  and making the change of variables from $s$ to $\frac 1 n \lp s + \frac {n-1} 2 \rp$, one arrives at
	\begin{equation*}
	J (x; \usigma, \ulambda) = \frac {c_1 (\usigma) c_2 (\usigma)} {\sqrt n} \lp \frac {2 \pi} x \rp^{\frac {n-1} 2 } \frac 1 {2\pi i} \int_{ n \EC - \frac {n-1} 2 }  \Gamma (s) e \lp \frac {n_+ (\usigma) - n_- (\usigma) } {4 n} \cdot s \rp  
	(n x)^{- s} ds,
	\end{equation*}
	with $c_2 (\usigma) = e \lp \mp \frac {n-1} 8 \pm \frac {(n-1) n_{\pm} (\usigma) } {4 n}  \rp$. \eqref{2eq: special example} now follows from \eqref{2eq: Mellin n=1} if the contour $\EC$ is suitably chosen.
\end{proof}

\section{The rigorous interpretation of formal integral representations}
\label{sec: Rigorous interpretations of formal integral representation}


We first introduce some new notations. Let $d = n-1$, $\ut = (t_1, ... , t_d) \in \BR_+^d$, $\unu = (\nu_1, ..., \nu_d) \in \BC^{d}$ and  $\usigma = (\varsigma_1, ..., \varsigma_d, \varsigma_{d+1}) \in \{+, - \}^{d+1}$.  
For $a > 0$ define $\BS^d_a = \left\{ \unu \in \BC^d : |\Re  \nu_l     | < a \text { for all } l      = 1, ..., d \right \}$. For $\nu \in \BC$ define
\begin{equation*}
\begin{split}
[\nu]_{\alpha} =  \textstyle \prod_{k = 0}^{\alpha-1} (\nu - k), \ (\nu)_{\alpha} =  \textstyle \prod_{k = 0}^{\alpha-1} (\nu + k) \ \text{ if } \alpha\geq 1, \hskip 10 pt [\nu]_{0} = (\nu)_0 = 1.
\end{split}
\end{equation*}
Denote by $p_{\unu} $ the power function
\begin{equation*} 
p_{\unu} (\ut) = \prod_{l      = 1}^d t_{l     }^{\nu_l      - 1},
\end{equation*}
let
\begin{equation*} 
\theta ( \ut; \usigma) = \varsigma_{d+1} t_1 ... t_d + \sum_{l      = 1}^d \varsigma_l      t_{l     }\-,
\end{equation*}
and define the formal integral 
\begin{equation}\label{3eq: definition J nu (x; sigma)}
J_{\unu} ( x ; \usigma) = \int_{\BR_+^d} p_{\unu} (\ut)e^{i x \theta ( \ut; \usigma)} d \ut.
\end{equation}
One sees that the formal integral representation of $J ( x ; \usigma, \ulambda)$ given in  (\ref{2eq: formal integral representation of J (x; sigma)}) is equal to $J_{\unu} ( x ; \usigma)$ if  $\nu_{l     } =  \lambda_l      - \lambda_{d+1}$, $l      = 1, ..., d$.


For $d = 1$, it is seen in \S \ref{sec: Classical Bessel function} that $J_{\nu}(x; \usigma)$ is conditionally convergent if and only if $|\Re \nu | < 1$ but fails to be absolutely convergent. 
When $d \geq 2$, we are in a worse scenario. The notion of convergence for multiple integrals is always in the absolute sense. Thus, the $d$-dimensional multiple integral in \eqref{3eq: definition J nu (x; sigma)} alone does not make any sense, 
since it is clearly not absolutely convergent. 


In the following, we shall address this fundamental convergence issue of the formal integral $J_{\unu} ( x ; \usigma)$, relying on its structural simplicity, so that it will be provided with mathematically rigorous meanings\footnote{It turns out that our rigorous interpretation actually coincides with the {\it Hadamard partie finie} of the formal integral.}. Moreover, it will be shown that our rigorous interpretation of $J_{\nu}(x; \usigma)$ is a smooth function of $x$ on $\BR_+$ as well as an analytic function of $\unu$ on $\BC^d$. 


\subsection{Formal partial integration operators}

The most crucial observation is that there are \textit{two} kinds of formal partial integrations. 
The first kind arises from 
$$\partial \big(e^{\varsigma_l      i x t_{l     }\- } \big) = - \varsigma_l      i  x t_{l     }^{-2} e^{\varsigma_l      i x t_{l     }\-} \partial t_{l     },$$ 
and the second kind from
 $$\partial \left(e^{\varsigma_{d+1} i x t_1 ... t_d} \right) = \varsigma_{d+1} i x t_1 ... \widehat {t_{l     }} ... t_d e^{\varsigma_{d+1} i x t_1 ... t_d} \partial t_{l     },$$ where $\widehat {t_{l     }}$ means that  $t_{l     }$ is omitted from the product.
\begin{defn}\label{2defn: formal partial integration}
Let $$\mathscr T (\BR_+) = \left\{ h \in C^\infty (\BR_+) : t^\alpha h^{(\alpha)}(t) \lll_\alpha 1 \text { for all } \alpha \in \BN \right \}.$$ For $h(\ut) \in \bigotimes^d \mathscr T (\BR_+)$, in the sense that $h(\ut) $ is a linear combination of functions of the form $ \prod_{l      = 1}^d h_l     (t_{l     })$, define the   integral
\begin{equation*} 
J_{\unu} (x; \usigma; h) = \int_{\BR_+^d} h (\ut) p_{\unu} (\ut) e^{ i x \theta (\ut; \usigma)} d \ut. 
\end{equation*} 
We call $J_{\unu} (x; \usigma; h)$ a $J$-integral of index $\unu$.
Let us introduce an auxiliary space
$$\mathscr J_{\unu} (\usigma) = \mathrm {Span }_{\BC[x\-]} \left\{ J_{\boldsymbol \nu'} (x; \usigma; h) : \boldsymbol \nu' \in \unu + \BZ^d, h \in \textstyle \bigotimes^d \mathscr T (\BR_+)\right \}. $$
Here $\BC[x\-]$ is the ring of polynomials of variable $x\-$ and complex coefficients.
Finally, we define $\EuScript P_{+,\, l     }$ and $\EuScript P_{-,\, l     }$  to be the two $\BC[x\-]$-linear operators on the space $\mathscr J_{\unu} (\usigma)$, in symbolic notion, as follows,
\begin{align*} 
\EuScript P_{+,\, l     } (J_{\unu} & (x ; \usigma; h)) = \varsigma_l      \varsigma_{d+1} J_{\unu + \ue^d + \ue_l      } \left(x; \usigma; h \right) \\
& - \varsigma_l      i (\nu_l      + 1) x \- J_{\unu + \ue_l     } \left(x; \usigma; h \right) - \varsigma_{l     } i x\- J_{\unu + \ue_l     } \left(x; \usigma; t_{l     } \partial_{l     } h \right), 
\end{align*}
\begin{align*} 
\EuScript P_{-,\, l     } & (J_{\unu} (x ; \usigma; h)) = \varsigma_l      \varsigma_{d+1} J_{\unu - \ue^d - \ue_l      } \left(x; \usigma; h \right)\\
& + \varsigma_{d+1} i (\nu_l      - 1) x \- J_{\unu - \ue^d} \left(x; \usigma; h \right) + \varsigma_{d+1} i x\- J_{\unu -\ue^d } \left(x; \usigma; t_{l     } \partial_{l     } h \right), 
\end{align*}
where $\ue_l      = (\underbrace {0,..., 0, 1}_{l     }, 0 ..., 0)$ and $\ue^d = (1, ..., 1)$, and $\partial_{l     } h$ is the abbreviated $\partial h / \partial t_{l     }$.
\end{defn}

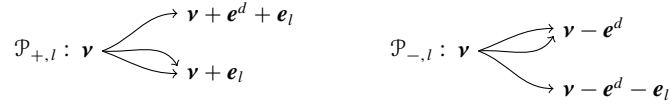
\begin{figure}
	\begin{center}
		\begin{tikzpicture}
		
		\node [left] at (0, 0) { \small $\unu$};

		\draw[->] (0,0) to [out=20,in=180] (1, 0.5);
		\draw[->] (0,0) to [out=- 20,in=180] (1, - 0.3);
		\draw[->] (0,0) to [out= -5,in=120] (1, - 0.2);
		
		\node [right] at (1, 0.5) { \small $\unu + \ue^d + \ue_l     $};
		\node [right] at (1, -0.3) { \small $\unu + \ue_l     $};
		\node [left] at (-0.3, 0) { \small $\EuScript P_{+,\, l     } :$};

		\node [left] at (5, 0) { \small $\unu$};

		\draw[->] (5,0) to [out=-20,in=180] (6, -0.5);
		\draw[->] (5,0) to [out= 5,in=-120] (6, 0.2);
		\draw[->] (5,0) to [out= 20,in=180] (6,  0.3);

		\node [right] at (6, 0.3) { \small $\unu - \ue^d$};
		\node [right] at (6, -0.5) { \small $\unu - \ue^d - \ue_l     $};
		\node [left] at (5-0.3, 0) { \small $\EuScript P_{-,\, l     } :$};
		
		\end{tikzpicture}
	\end{center}
	\caption{Index shifts} 
\end{figure}

The formulations of $\EuScript P_{+,\, l     }$ and $\EuScript P_{-,\, l     }$ are quite involved at a first glance. However, the most essential feature of these operators is simply \textit{index shifts}! 

\begin{observ} \nonumber
After the operation of $\EuScript P_{+,\, l     }$ on a $J$-integral, all the indices of the three resulting $J$-integrals are nondecreasing and the increment of the $l     $-th index is one greater than the others. The operator $\EuScript P_{-,\, l     }$ has the effect of decreasing all indices by one except possibly two for  the $l     $-th index.
\end{observ}

\begin{lem}\label{lem: I pm ell}
Let notations be as above.

{\rm(1).} Let  $h(\ut) = \prod_{l      = 1}^d h_l     (t_{l     })$. Suppose that the set $\{1, 2, ..., d\}$ splits into two subsets $L_+$ and $L_-$ such that 
\begin{itemize}
\item [-]	$h_l      $ vanishes at infinity if $l      \in L_-$, and 
\item[-] $h_l      $ vanishes in a neighbourhood of zero if $l      \in L_+$.
\end{itemize} If $\Re \nu_l     >  0$ for all  $l      \in L_-$ and 
 $\Re \nu_l      < 0$ for all  $l      \in L_+$, then the $J$-integral $J_{\unu} (x ; \usigma; h)$ absolutely converges.

{\rm(2).} Assume the same conditions in {\rm(1)}.  Moreover, suppose that $\Re \nu_l     > 1$ for all  $l      \in L_-$ and  $\Re \nu_l      < -1$ for all  $l      \in L_+$. Then, for $l      \in L_-$, all the three $J$-integrals in the definition of $\EuScript P_{+,\, l     } (J_{\unu} (x ; \usigma; h) )$ are absolutely convergent and the operation of $\EuScript P_{+,\, l     }$ on $J_{\unu} (x ; \usigma; h)$ is the actual partial integration of the first kind on the integral over $d t_{l     }$. Similarly, for  $l      \in L_+$, the operation of $\EuScript P_{-,\, l     }$ preserves absolute convergence and is the actual partial integration of the second kind on the integral over $d t_{l     }$.


{\rm(3).}  $\EuScript P_{+,\, l     }$ and $\EuScript P_{-,\, l     }$ commute with $\EuScript P_{+,\, k}$ and  $\EuScript P_{-,\, k}$ if $l      \neq k$.

{\rm(3).}  $\EuScript P_{+,\, l     }$ and $\EuScript P_{-,\, l     }$ commute with $\EuScript P_{+,\, k}$ and  $\EuScript P_{-,\, k}$ if $l      \neq k$.

{\rm(4).} Let  $\alpha \in \BN$. $\EuScript P^{\alpha}_{+,\, l     } (J_{\unu}  (x; \usigma; h) )$ is a linear combination of 
\begin{align*}  
[\nu_l  - 1]_{\alpha_3} & x^{- \alpha + \alpha_1} J_{\unu +  \alpha_1 \ue^d + \alpha \ue_l     } (x; \usigma; \textstyle  t^{\alpha_2}_l      \partial_{l     }^{\alpha_2} h ), 
\end{align*} 
and $\EuScript P^{\alpha}_{-,\, l     }  (J_{\unu} (x;  \usigma;  h)  )$ is a linear combination of
\begin{align*}  
[ \nu_l      - 1 ]_{\alpha_3} x^{- \alpha + \alpha_1} J_{\unu - \alpha \ue^d - \alpha_1 \ue_l     } (x; \usigma; \textstyle  t^{\alpha_2}_l      \partial_l     ^{\alpha_2}  h ),
\end{align*}
for $\alpha_1 + \alpha_2 + \alpha_3 \leqslant \alpha$.
The coefficients of these linear combinations may be uniformly bounded by a constant depending only on $\alpha$.
\end{lem}
\begin{proof}
(1-3) are obvious. The two statements in (4) follow from calculating 
\begin{equation*}
x^{-\alpha} t_l^{\alpha + \alpha_0} \partial ^{\alpha_0}_l   \left(   h(\ut) p_{\unu       } (\ut) e^{\varsigma_{d+1} i x    t_1 ... t_d     }\right) e^{i x \sum_{k = 1}^d \varsigma_{k} t_{k}\- }, \hskip 10 pt \alpha_0 \leqslant \alpha,
\end{equation*}
and
\begin{equation*}
x^{-\alpha} \partial ^\alpha_l      \left(   h(\ut) p_{\unu - \alpha \ue^d + \alpha \ue_l     } (\ut) e^{i x \sum_{k = 1}^d \varsigma_{k} t_{k}\- }\right) e^{\varsigma_{d+1} i x t_1 ... t_d }.
\end{equation*}
For the latter, one applies the following formula
\begin{equation*}
\frac {d^\alpha \big( e^{ a t\- } \big)} {d t^\alpha} = (- )^\alpha \sum_{\beta = 1}^\alpha \frac {\alpha! (\alpha-1)!} {(\alpha-\beta)! \beta! (\beta-1)!} a^\beta t^{- \alpha - \beta} e^{ a t\- },  \hskip 10 pt \alpha \in \BN_+, \, a \in \BC.
\end{equation*} 
\end{proof}

\subsection{\texorpdfstring{Partitioning the integral $J_{\unu}(x; \usigma)$}{Partitioning the integral $J_{\nu}(x; \varsigma)$}}

Let $I$ be a finite set that includes $\{+, -\}$ and let
$$\sum_{\varrho \in I} h_\varrho (t) \equiv 1, \hskip 10 pt t \in \BR_+,$$ be a partition of unity on $\BR_+$ such that each $h_\varrho $ is a function in $ \mathscr T (\BR_+)$, 
$h_-(t) \equiv 1$ on a neighbourhood of zero and $h_+(t) \equiv 1$ for large $t$. Put $h_{\urho} (\ut) =  \prod_{l      = 1}^d h_{\varrho_l     } (t_{l     })$ for $\urho =  (\varrho_1, ..., \varrho_d) \in I^d$. We partition the integral $J_{\unu}(x; \usigma)$ into a finite sum of $J$-integrals 
\begin{equation*} 
J_{\unu}(x; \usigma) = \sum_{\urho \in I^d } J_{\unu}(x; \usigma; \urho),
\end{equation*}
with
\begin{equation*} 
J_{\unu}(x; \usigma; \urho) = J_{\unu}(x; \usigma; h_{\urho})
= \int_{\BR_+^d} h_{\urho} (\ut) p_{\unu} (\ut) e^{ i x \theta (\ut; \usigma)} d \ut.
\end{equation*}

\subsection{\texorpdfstring{The definition of $\BJ_{\unu}(x; \usigma)$}{The definition of $J_{\nu}(x; \varsigma)$}} \label{sec: Definition of BJ nu (x; sigma)}

Let $a > 0$ and assume $\unu \in \BS_{a}$. Let $A \geq a + 2$ be an integer. For $\urho \in I^d$ denote $L_\pm (\urho) = \left\{ l      : \varrho_l      = \pm \right \}$. 

We first treat $J_{\unu}(x; \usigma; \urho)$ in the case when both $L_+(\urho)$ and $L_-(\urho) $ are nonempty. Define $\EuScript P_{+,\, \urho} = \prod_{l      \in L_-(\urho)} \EuScript P_{+,\, l     }$. This is well-defined due to commutativity (Lemma \ref{lem: I pm ell} (3)).
By Lemma \ref{lem: I pm ell} (4) we find that $\EuScript P^{2 A}_{+, \urho} (J_{\unu}(x; \usigma; \urho) )$ is a linear combination of
\begin{equation} \label{3eq: after I - rho}
\begin{split}
& \left(\prod_{l      \in L_- (\urho)} [ \nu_l     - 1 ]_{\alpha_{3, l     }} \right) x^{- 2 A |L_-(\urho)| + \sum_{l      \in L_-(\urho)} \alpha_{1, l     }} \cdot \\
& \hskip 40pt J_{\unu + (\sum_{l      \in L_-(\urho)} \alpha_{1, l     } ) \ue^d + 2 A \sum_{l      \in L_-(\urho)} \ue_l     } \big(x; \usigma; \big( \textstyle \prod_{l      \in L_-(\urho)} t^{\alpha_{2, l     }}_l      \partial_l     ^{\alpha_{2, l     }}\big) h_{\urho}\big),
\end{split}
\end{equation}
with 
$\alpha_{1, l     } + \alpha_{2, l     } + \alpha_{3, l     } \leqslant 2 A$ for each $l      \in L_-(\urho)$.
After this, we choose $l     _+ \in L_+(\urho)$ and apply  $\EuScript P_{-,\, l     _+}^{ A + \sum_{l      \in L_-(\urho)} \alpha_{1, l     }  }$ on the $J$-integral in (\ref{3eq: after I - rho}). 
By Lemma \ref{lem: I pm ell} (4) we obtain a linear combination of
\begin{equation} \label{3eq: linear combination first case}
\begin{split}
& [\nu_{l     _+} - 1]_{\alpha_3} \left(\prod_{l      \in L_- (\urho)} [\nu_l    - 1]_{\alpha_{3, l     }}  \right) x^{- A (2 |L_-(\urho)| + 1) + \alpha_1} \cdot \\
& \hskip 40pt J_{\unu - A \ue^d + 2A \sum_{l      \in L_-(\urho)} \ue_l      - \alpha_1 \ue_{l     _+} } \big(x; \usigma; \big( t^{\alpha_2}_{l     _+} 
\partial^{\alpha_2}_{l     _+} \textstyle \prod_{l      \in L_-(\urho)} t^{\alpha_{2, l     }}_l      \partial^{\alpha_{2, l     }} _{l     } \big) h_{\urho}\big),
\end{split}
\end{equation}
with 
 $\alpha_1 + \alpha_2 + \alpha_3 \leq \sum_{l      \in L_-(\urho)} \alpha_{1, l     } + A$.
It is  easy to verify that the real part of the $l     $-th index of the $J$-integral in (\ref{3eq: linear combination first case}) is positive if    $l      \in L_-(\urho)$ and 
negative if $l      \in L_+(\urho)$. 
Therefore, the $J$-integral in (\ref{3eq: linear combination first case}) is absolutely convergent according to Lemma \ref{lem: I pm ell} (1). We define  $\BJ_{\unu}(x; \usigma; \urho)$ to be the total linear combination of  all the $J$-integrals obtained after these two steps of operations.

When $L_- (\urho) \neq \O$ but $L_+ (\urho) = \O$, we define $\BJ_{\unu}(x; \usigma; \urho) = \EuScript P^A_{+, \urho} ( J_{\unu}(x; \usigma; \urho) )$. It is a linear combination of 
\begin{equation} \label{3eq: linear combination second case}
\begin{split}
& \left( \prod_{l      \in L_- (\urho)} [\nu_l     - 1]_{\alpha_{3, l     }} \right) x^{- A |L_-(\urho)| + \sum_{l      \in L_-(\urho)} \alpha_{1, l     }} \cdot \\
& \hskip 40pt J_{\unu + (\sum_{l      \in L_-(\urho)} \alpha_{1, l     } ) \ue^d + A \sum_{l      \in L_-(\urho)} \ue_l     } \big(x; \usigma; \big(\textstyle \prod_{l      \in L_-(\urho)} t^{\alpha_{2, l     }}_l      \partial^{\alpha_{2, l     }} _{l     } \big) h_{\urho}\big),
\end{split}
\end{equation}
with $\alpha_{1, l     } + \alpha_{2, l     } + \alpha_{3, l     } \leqslant A$. The $J$-integral in (\ref{3eq: linear combination second case}) is absolutely convergent.

When  $L_+ (\urho) \neq \O$ but $L_- (\urho) = \O$, we choose $l     _+ \in L_+ (\urho) $ and define $\BJ_{\unu}(x; \usigma; \urho) = \EuScript P^A_{-,\, l     _+} ( J_{\unu}(x; \usigma; \urho) )$. This is a linear combination of 
\begin{equation} \label{3eq: linear combination third case}
\begin{split}
 [\nu_{l     _+} - 1]_{\alpha_3} x^{- A + \alpha_1} J_{\unu - A \ue^d - \alpha_1 \ue_{l     _+} } \big(x; \usigma; t^{\alpha_2}_{l     _+} \partial_{l     _+}^{\alpha_2} h_{\urho} \big),
\end{split}
\end{equation}
with $\alpha_1 + \alpha_2 + \alpha_3 \leq A$. The $J$-integral in (\ref{3eq: linear combination third case}) is again absolutely convergent.

Finally,  when both  $L_- (\urho) $ and $L_+ (\urho) $ are empty, we put $\BJ_{\unu}(x; \usigma;  \urho) = J_{\unu}(x; \usigma;  \urho)$.

\begin{lem}\label{lem: BJ nu (x; sigma; rho) is well-deifned}
The definition of  $\BJ_{\unu}(x; \usigma; \urho)$ is independent on $A$ and the choice of $l     _+\in L_+ (\urho)$.
\end{lem}
\begin{proof}
We shall treat the case when both $L_+(\urho)$ and $L_-(\urho) $ are nonempty. The other cases are similar and simpler.

Starting from the $\BJ_{\unu}(x; \usigma; \urho)$ defined with $A$, we conduct the following operations in succession for all $l      \in L_-(\urho)$: $\EuScript P_{+,\, l     }$ twice and then $\EuScript P_{-,\, l     _+}$ once, twice or three times on each resulting $J$-integral so that the increment of the $l     $-th index is exactly one.  In this way, one arrives at the $\BJ_{\unu}(x; \usigma; \urho)$ defined with $A+1$. 
In view of the assumption $A \geq a + 2 $, 
absolute convergence is maintained at each step due to Lemma \ref{lem: I pm ell} (1). Moreover, in our settings, the operations $\EuScript P_{+,\, l     }$ and  $\EuScript P_{-,\, l     _+}$ are actual partial integrations (Lemma \ref{lem: I pm ell} (2)), so the value  is preserved in the process. In conclusion, $\BJ_{\unu}(x; \usigma; \urho)$  is independent on $A$.

Suppose $l     _+, k_+ \in L_+(\urho)$. Repeating the process described in the last paragraph $A$ times, but with $l     _+$ replaced by $k_+$,  the $\BJ_{\unu}(x; \usigma; \urho)$ defined with $l     _+$ turns into a sum of integrals of an expression symmetric about $l     _+$ and $k_+$. 
Interchanging $l     _+$ and $k_+$ throughout the arguments above, the $\BJ_{\unu}(x; \usigma; \urho)$ defined with $k_+$ is transformed into the same  sum of integrals. Thus we conclude that $\BJ_{\unu}(x; \usigma; \urho)$ is independent on the choice of $l     _+$.
\end{proof}

Putting these together, we define
\begin{equation*} 
\BJ_{\unu}(x; \usigma) = \sum_{\urho \in I^d } \BJ_{\unu}(x; \usigma; \urho),
\end{equation*}
and call $\BJ_{\unu} (x; \usigma)$ the \textit{rigorous interpretation  of $J_{\unu} (x; \usigma)$}. 
The definition of $\BJ_{\unu} (x; \usigma)$ is clearly independent on the partition of unity $\{ h_{\varrho} \}_{\varrho \in I }$ on $\BR_+$.

Uniform convergence of the $J$-integrals in (\ref{3eq: linear combination first case}, \ref{3eq: linear combination second case}, \ref{3eq: linear combination third case}) with respect to $\unu$ implies that $\BJ_{\unu}(x; \usigma)$ is an analytic function of $\unu$ on $\BS_a^d$ and hence on the whole $\BC^d$ since $a$ was arbitrary. Moreover, for any nonnegative integer $j $, if one chooses $A \geq a + j + 2$, differentiating  $j$ times under the integral sign for  the $J$-integrals in (\ref{3eq: linear combination first case}, \ref{3eq: linear combination second case}, \ref{3eq: linear combination third case}) is legitimate. Therefore,  $\BJ_{\unu}(x; \usigma)$ is a smooth function of $x$. 

Henceforth, with ambiguity, we shall write $\BJ_{\unu} (x; \usigma)$ and  $\BJ_{\unu} (x; \usigma; \urho)$ as $J_{\unu} (x; \usigma)$ and $J_{\unu} (x; \usigma; \urho)$ respectively.

\section{\texorpdfstring{Equality between $J_{\unu}(x; \usigma)$ and $J (x ; \usigma, \ulambda)$}{Equality between $J_{\nu}(x; \varsigma)$ and $J (x ; \varsigma, \lambda)$}} \label{sec: Relating Bessel functions} 
The goal of this section is to prove  that the Bessel function $J (x ; \usigma, \ulambda)$ is indeed equal to the rigorous interpretation of its formal integral representation  $J_{\unu}(x; \usigma)$.

\begin{prop}\label{prop: J (x ; sigma, lambda) = J nu (x; sigma)} Suppose that $\ulambda \in \BL^{d}$ and $\unu \in \BC^d$ satisfy $\nu_l      = \lambda_l      - \lambda_{d+1}$, $l      = 1, ..., d$. Then
\begin{equation*}
J (x ; \usigma, \ulambda) = J_{\unu} (x; \usigma).
\end{equation*}
\end{prop}

To prove this proposition, one first needs to know how the iterated integral  $\Upsilon(x; \usigma)$ given in \eqref{2eq: Psi (x; sigma) another integral representation} is interpreted  (compare \cite[\S 6]{Miller-Schmid-2004} and \cite[\S 5]{Miller-Schmid-2006}).

Suppose that $\Re \lambda_1 > ... > \Re \lambda_d > \Re \lambda_{d+1}$. 
Let $\upsilon \in \mathscr S (\BR_+)$ be a Schwartz function on $\BR_+$. Define
\begin{equation}\label{4eq: Psi d+1}
\Upsilon_{d+1} (x; \usigma ) = \int_{\BR_+} \upsilon (y) y^{-\lambda_{d+1}} e(\varsigma_{d+1} x y) d y, \hskip 10 pt x \in \BR _+,
\end{equation}
and for each $l      = 1, ..., d $ recursively define
\begin{equation}\label{4eq: Psi ell}
\begin{split}
 \Upsilon_{l     } \left(x; \usigma \right)  = \int_{\BR_+} \Upsilon_{l      + 1} \lp y; \usigma  \rp y^{\lambda_l      - \lambda_{l      + 1} - 1} e\lp \varsigma_{l     } x y\- \rp d y,
\hskip 10 pt x \in \BR _+.
\end{split}
\end{equation}
\begin{lem}\label{4lem: Psi ell} 
Suppose that  $\Re \lambda_1 > ... > \Re \lambda_d > \Re \lambda_{d+1}$. 
Recall the definition of $\mathscr T (\BR _+)$ given in Definition {\rm \ref{2defn: formal partial integration}}, and define the space $\mathscr T_\infty (\BR _+)$  of all functions in $ \mathscr T (\BR _+)$ that  decay rapidly at infinity, along with all their derivatives. Then $\Upsilon_{l     } \left(x; \usigma \right) \in \mathscr T_\infty (\BR _+)$ for each $ l      = 1, ..., d+1$. 





\end{lem}
\begin{proof}



In the case $l      = d+1$, $\Upsilon_{d+1} (x; \usigma )$ is the Fourier transform of a Schwartz function on $\BR$ (supported in $\BR_+$) and hence is actually a Schwartz function on $\BR$. In particular, $\Upsilon_{d+1} (x; \usigma ) \in \mathscr T_\infty (\BR _+)$. One may also prove this directly via performing partial integration and differentiation under the integral sign on the integral in  \eqref{4eq: Psi d+1}. 

Suppose that $\Upsilon_{l     +1} \left(x; \usigma \right) \in \mathscr T_\infty (\BR _+)$. The condition $\Re \lambda_{l     } > \Re \lambda_{l     +1}$ secures the convergence of the integral in \eqref{4eq: Psi ell}. 
Partial integration   has the effect of dividing $\varsigma_{l     } 2 \pi i x$ and    results in an integral of the same type but with \textit{the power of $y$ raised by one}, so repeating this yields the rapid decay of $\Upsilon_{l     } \left(x; \usigma \right)$. Moreover, differentiation under the integral sign \textit{decreases the power of $y$ by one}, so multiple differentiating  $\Upsilon_{l     } \left(x; \usigma \right)$ is legitimate after repeated partial integrations. From these, it is straightforward to prove that $ \Upsilon_{l     } (x; \usigma ) \in \mathscr T (\BR _+)$. Finally, keeping repeating partial integrations yields the rapid decay of all the derivatives of $\Upsilon_{l     } \left(x; \usigma \right)$.
\end{proof}


The change of variables from $y $ to $ x y$ in (\ref{4eq: Psi ell}) yields
\begin{equation*} 
\Upsilon_{l     } \lp x; \usigma \rp = \int_{\BR_+} \Upsilon_{l      + 1} \lp x y; \usigma  \rp x^{\lambda_l      - \lambda_{l      + 1}} y^{\lambda_l      - \lambda_{l      + 1} - 1} e\lp \varsigma_{l     } y\- \rp d y.
\end{equation*}
Some calculations then show that $ \Upsilon_{1} \lp x; \usigma  \rp$ is equal to the iterated integral
\begin{equation}\label{4eq: Psi 1 (x; sigma)}
\begin{split}
x^{\nu_1 } \int_{\BR_+^{d+1}} \upsilon \left( y \right) y^{- \lambda_{d+1} }   \left( \prod_{l      = 1}^{d} y_l     ^{ \nu_l       - 1} \right)  
  e \left( \varsigma_{d+1}  xy y_1 ... y_{d} + \sum_{l      = 1}^d \varsigma_l      y_l     \- \right) dy d y_d ... dy_1.
\end{split}
\end{equation}
Comparing \eqref{4eq: Psi 1 (x; sigma)} with (\ref{2eq: Psi (x; sigma) another integral representation}), one sees that $\Upsilon(x; \usigma) = x^{ - \lambda_1 } \Upsilon_1 (x; \usigma )$. 

The (actual) partial integration $\EuScript P_{l     }$ on the integral over $d y_{l     }$ 
is in correspondence with $\EuScript P_{+,\, l     }$, whereas the partial integration $\EuScript P_{d+1}$ on the integral over $d y$ 
has the similar effect as $\EuScript P_{-,\, l     _+}$ of decreasing the powers of all the $y_{l     }$ by one. 
These observations are crucial to our proof of Proposition \ref{prop: J (x ; sigma, lambda) = J nu (x; sigma)} as follows.

\begin{proof}[Proof of Proposition {\rm \ref{prop: J (x ; sigma, lambda) = J nu (x; sigma)}}]
	
Suppose that  $\Re \lambda_1 > ... > \Re \lambda_d > \Re \lambda_{d+1}$. 
We first partition the integral over $d y_{l     }$ in \eqref{4eq: Psi 1 (x; sigma)}, for each $l      = 1, ..., d$, into a sum of integrals according to a partition of unity $\{h^{\mathrm{o}}_{\varrho} \}_{\varrho \in I}$ of $\BR_+$. These partitions result in a partition of the integral (\ref{4eq: Psi 1 (x; sigma)}) into the sum
\begin{equation*}
\Upsilon_{1} \lp x; \usigma   \rp = \sum_{\urho \in I^d } \Upsilon_1 (x; \usigma  ; \urho),
\end{equation*}
with
\begin{equation}\label{4eq: Psi 1 (x; sigma; rho)}
\begin{split}
\Upsilon_1 (x; \usigma  ; \urho) = x^{\nu_1 } \int_{\BR_+^{d+1}} \upsilon \left( y \right) y^{- \lambda_{d+1} } & \left( \prod_{l      = 1}^{d} h^{\mathrm{o}}_{\varrho_l     } (y_l     ) y_l     ^{ \nu_l      - 1} \right) \\
e &  \left( \varsigma_{d+1}  xy y_1 ... y_{d} + \sum_{l      = 1}^d \varsigma_l      y_l     \- \right) dy d y_d ... dy_1.
\end{split}
\end{equation}
We now conduct the operations in \S \ref{sec: Definition of BJ nu (x; sigma)} with $\EuScript P_{+,\, l     }$ replaced by $\EuScript P_{l     }$ and $\EuScript P_{-,\, l     _+}$ by $\EuScript P_{d+1}$ to each integral $ \Upsilon_1 (x; \usigma   ; \urho)$ defined in (\ref{4eq: Psi 1 (x; sigma; rho)}). While preserving the value, these   partial integrations turn the iterated integral $ \Upsilon_1 (x; \usigma   ; \urho)$ into an absolutely convergent multiple integral. We are then able to move the innermost integral over $d y$ to the outermost place. The integral over  $dy_d ... dy_1$ now becomes   the {inner} integral. Making the change of variables $y_l      = t_{l     } (xy)^{-\frac 1 { d+1} }$ to the inner integral  over  $dy_d ... dy_1$,
each partial integration $\EuScript P_{l     }$ that we did turns into $\EuScript P_{+,\, l     }$. By the same arguments in the proof of Lemma \ref{lem: BJ nu (x; sigma; rho) is well-deifned} showing that $J_{\unu}(x; \usigma)$ is independent on the choice of $l     _+\in L_+(\urho)$, the operations of $\EuScript P_{d+1}$ that we conducted at the beginning may be reversed and substituted by those of $\EuScript P_{-,\, l     _+}$. It follows that the inner integral over $dy_d ... dy_1$ is equal to $x^{\lambda_1 } \upsilon (y) J_{\unu} \big( 2\pi (xy)^{\frac 1 { d+1 } }; \usigma; \urho \big)$, with $ h_{\varrho}(t) = h^{\mathrm{o}}_{\varrho}\big(  t (xy)^{-\frac 1 { d+1 }} \big)$. Summing over $\urho \in I^d$, we conclude that
\begin{equation*}
\Upsilon (x; \usigma) = x^{ - \lambda_1} \Upsilon_1 (x; \usigma  ) =  \int_{\BR_+} \upsilon (y) J_{\unu} \big( 
2\pi (xy)^{\frac 1 { d+1 }}; \usigma \big) d y.
\end{equation*}
Therefore, in view of (\ref{2eq: Psi (x; sigma) as Hankel transform}), we have
$J (x ; \usigma, \ulambda) = J_{\unu} (x; \usigma).$ This equality  holds true universally due to the principle of analytic continuation. 
\end{proof}


In view of Proposition \ref{prop: J (x ; sigma, lambda) = J nu (x; sigma)}, we shall subsequently assume that $\ulambda \in \BL^{d}$ and $\unu \in \BC^d$ satisfy the relations $\nu_l      = \lambda_l      - \lambda_{d+1}$, $l      = 1, ..., d$.

\section{$H$-Bessel functions and $K$-Bessel functions}\label{sec: Bessel functions of K-type and H-Bessel functions}

According to Proposition \ref{prop: Classical Bessel functions}, $ J_{2\lambda} (x; \pm, \pm) = J (x; \pm, \pm, \lambda, - \lambda) $ 
is a Hankel function, and $ J_{2\lambda} (x; \pm, \mp) = J (x; \pm, \mp, \lambda, - \lambda)$
is a $K$-Bessel function. There is a remarkable difference between the behaviours of   Hankel functions and the $K$-Bessel function for large argument. The Hankel functions oscillate and decay proportionally to $\frac 1 {\sqrt x}$, whereas the $K$-Bessel function exponentially decays. 
On the other hand, this phenomena also arises in higher rank for the  prototypical example shown in Proposition \ref{prop: special example}.  

In the following, we shall  show that such a categorization stands in general for the Bessel functions $J_{\unu} (x; \usigma)$ of an arbitrary index $\unu$.
For this, we shall analyze each integral $J_{\unu}(x; \usigma; \urho)$ in the rigorous interpretation of  $J_{\unu} (x; \usigma)$ using \textit{the method of stationary phase}.

First of all, the asymptotic behaviour of $J_{\unu} (x; \usigma)$  for large argument should rely on the existence of a stationary point of the phase function $ \theta (\ut ; \usigma )$ on $\BR_+^d$.
We have $$ \theta' (\ut; \usigma) = \left( \varsigma_{d+1} t_1 ... \widehat {t_{l     }} ... t_d -  \varsigma_l      t_{l     }^{-2} \right)_{l      = 1}^d.$$ 
A stationary point  of  $ \theta (\ut ; \usigma )$  exists in $\BR^d_+$ if and only if $\varsigma_1 = ... = \varsigma_d = \varsigma_{d+1}$, in which case  it is equal to $ \ut_0 = (1, ..., 1)$.

\begin{term}\label{term: Bessel functions of K-type and H-Bessel functions}
We write $H^{\pm}_{\unu}(x ) = J_{\unu} (x; \pm, ..., \pm)$, $H^{\pm} (x; \ulambda ) = J (x; \pm, ... \pm, \ulambda)$ 
and call them $H$-Bessel functions. If two of the signs $\varsigma_1, ..., \varsigma_d, \varsigma_{d+1}$ are different, then $J_{\unu} (x; \usigma)$, or $J (x; \usigma, \ulambda)$, is called a $K$-Bessel function.
\end{term}

\addtocontents{toc}{\protect\setcounter{tocdepth}{1}}

\subsection*{Preparations}
We shall retain the notations in \S \ref{sec: Rigorous interpretations of formal integral representation}.
Moreover, for our purpose we choose a partition of unity $\left\{ h_{\varrho} \right \}_{\varrho \in \{-, 0, +\}}$ on $\BR_+$ such that
$h_- $, $h_0 $ and $h_+ $ are  functions in $\mathscr T (\BR_+)$ supported on $K_- = \left(0, \frac 1 2\right]$, $K_0 = \left[\frac 1 4, 4\right]$ and $K_+ = \left[2, \infty \right)$ respectively. 
Put $K_{\urho} = \prod_{l      = 1}^d K_{\varrho_l     } $ and $h_{\urho} (\ut) =  \prod_{l      = 1}^d h_{\varrho_l     } (t_{l     })$ for $\urho \in \{-, 0 , + \}^d$.
Note that $\ut_0$ is enclosed in the central hypercube $K_{\boldsymbol 0}$.
According to this partition of unity, $J_{\unu}(x; \usigma)$ is partitioned into the sum of $3^d$ integrals $J_{\unu}(x; \usigma; \urho)$. In view of (\ref{3eq: linear combination first case}, \ref{3eq: linear combination second case}, \ref{3eq: linear combination third case}), $J_{\unu}(x; \usigma; \urho)$ is a $\BC [x\-]$-linear combination of absolutely convergent $J$-integrals of the form
\begin{equation}\label{5eq: J-integral}
J_{\boldsymbol \nu'} (x; \usigma; h) = \int_{\BR_+^d} h (\ut) p_{\boldsymbol \nu'} (\ut) e^{ i x \theta (\ut; \usigma)} d \ut.
\end{equation}
Here $h \in \bigotimes^d \mathscr T (\BR_+)$ is supported in $K_{\urho}$, and $\boldsymbol \nu' \in \unu + \BZ^d$ satisfies
\begin{equation}\label{5eq: conditions on nu'}
\Re  \nu'_{l     } - \Re  \nu_{l     } \geq A \text { if } l      \in L_-(\urho), \text { and } \Re  \nu'_{l     } - \Re  \nu_{l     } \leq - A \text{ if }l      \in L_+(\urho),
\end{equation} 
with $A > \max  \left\{ |\Re \nu_l      | \right \} + 2$.

\addtocontents{toc}{\protect\setcounter{tocdepth}{2}}

\subsection{\texorpdfstring{Estimates for $J_{\unu}(x; \usigma; \urho)$ with $\urho \neq \boldsymbol 0$}{Estimates for  $J_{\nu}(x; \varsigma; \varrho)$ with $\varrho \neq 0$}} \label{sec: Bound for rho neq 0}


Let 
\begin{equation}\label{5eq: Theta (t; sigma)}
\Theta (\ut; \usigma) = \sum_{l      = 1}^d \lp t_{l     } \partial_{l     } \theta (\ut; \usigma) \rp^2 = \sum_{l      = 1}^d \lp \varsigma_{d+1} t_1 ... t_d - \varsigma_{l     } t_{l     }\- \rp^2.
\end{equation}

 \begin{lem}\label{lem: lower bound for Theta}
Let $\urho \neq \boldsymbol 0$. We have for all $\ut \in K_{\urho}$
$$\Theta (\ut; \usigma) \geq \frac 1 {16}.$$
\end{lem}
\begin{proof}
Instead, we shall prove
\begin{equation*}
\max \Big\{\left|\varsigma_{d+1} t_1 ... t_d - \varsigma_{l     } t_{l     }\- \right| :  \ut \in \BR_+^d \smallsetminus  K_{\boldsymbol 0} \text { and } l      = 1, ..., d \Big \} \geq \frac 1 4.
\end{equation*}
Firstly, if $t_1 ... t_d < \frac 3 4$, then there exists $t_l      < 1$ and hence $\left|\varsigma_{d+1} t_1 ... t_d - \varsigma_{l     } t_{l     }\- \right| > 1 - \frac 3 4 = \frac 1 4$. Similarly, if $t_1 ... t_d > \frac 7 4$, then there exists $t_l      > 1$ and hence $\left|\varsigma_{d+1} t_1 ... t_d - \varsigma_{l     } t_{l     }\- \right| >\frac 7 4 - 1 > \frac 1 4$. Finally, suppose that $\frac 3 4 \leq t_1 ... t_d \leq \frac 7 4$, then for our choice of $\ut $ there exists $l     $ such that $t_{l     } \notin \lp \frac 1 2, 2 \rp$, and therefore we still have $\left|\varsigma_{d+1} t_1 ... t_d - \varsigma_{l     } t_{l     }\- \right| \geq \frac 1 4$.
\end{proof}

Using (\ref{5eq: Theta (t; sigma)}), we rewrite the $J$-integral $J_{\boldsymbol \nu'} (x; \usigma; h)$ in \eqref{5eq: J-integral} as below,
\begin{equation} \label{5eq: rewrite J nv' (x; sigma; h)}
\begin{split}
\sum_{l      = 1}^d \int_{\BR_+^d} h (\ut) \left( \varsigma_{d+1} p_{\boldsymbol \nu' + \ue^d + \ue_{l     }} (\ut) - \varsigma _{l     } p_{\boldsymbol \nu'} (\ut) \right) \Theta (\ut; \usigma)\- \cdot \partial_{l     } \theta (\ut; \usigma) e^{i x \theta (\ut; \usigma)} d \ut.
\end{split}
\end{equation}
We now make use of  the    \textit{third} kind of partial integrations arising from
$$\partial \big(e^{i x \theta (\ut; \usigma) } \big) = i x \cdot \partial_{l     } \theta (\ut; \usigma) e^{ i x \theta (\ut; \usigma) } \partial t_l     . $$
For the $l     $-th integral in \eqref{5eq: rewrite J nv' (x; sigma; h)}, we apply the corresponding partial integration of the third kind. In this way, \eqref{5eq: rewrite J nv' (x; sigma; h)} turns into
\begin{align*}
& - (ix)\- \sum_{l      = 1}^d \int_{\BR_+^d} t_l      \partial_l      h \lp \varsigma _{d+1} p_{\boldsymbol \nu' + \ue^d} - \varsigma _{l     } p_{\boldsymbol \nu' -  \ue_{l     }} \rp  \Theta \- e^{i x \theta} d\ut\\
& - (ix)\- \sum_{l      = 1}^d \int_{\BR_+^d} h \lp \varsigma _{d+1} (\nu'_{l     } + 1) p_{\boldsymbol \nu' + \ue^d } - \varsigma _{l     } (\nu'_l      - 1) p_{\boldsymbol \nu' - \ue_{l     } } \rp \Theta \-  e^{i x \theta} d\ut\\
& + \varsigma _{d+1} 2 d^2 (ix)\- \int_{\BR_+^d} h  p_{\boldsymbol \nu' + 3 \ue^d} \Theta^{-2}  e^{i x \theta} d \ut \\
& + 2 (ix)\-  \sum_{l      = 1}^d \int_{\BR_+^d} h \big( \varsigma _{l     } (1 - 2 d) p_{\boldsymbol \nu' + 2 \ue^d -\ue_l     }    - \varsigma _{d+1} p_{\boldsymbol \nu' + \ue^d - 2 \ue_l     } + \varsigma _{l     } p_{\boldsymbol \nu' - 3 \ue_l     } \big) \Theta^{-2}  e^{i x \theta} d \ut\\
& + 4 (ix)\- \sum_{1\leq l      < k \leq d} \varsigma _{d+1} \varsigma _{l      } \varsigma _{k} \int_{\BR_+^d} h p_{\boldsymbol \nu' + \ue^d - \ue_l      - \ue_{k}}  \Theta^{-2}  e^{i x \theta} d \ut,
\end{align*}
where $\Theta$ and $\theta$ are the shorthand notations for $\Theta (\ut; \usigma)$ and $\theta (\ut; \usigma)$. Since the shifts of indices do not exceed $3$, it follows from the condition (\ref{5eq: conditions on nu'}), combined with  Lemma \ref{lem: lower bound for Theta}, that all the integrals above absolutely converge provided $A > \fr + 3$. 

Repeating the above manipulations, we obtain the following lemma by a straightforward inductive argument.
\begin{lem}\label{lem: J(x; sigma; h) for rho neq 0}
Let $B$ be a nonnegative integer, and choose $A =  \lfloor \mathfrak r \rfloor + 3 B + 3$. Then $J_{\boldsymbol \nu'} (x; \usigma; h)$ is equal to a linear combination of $\left(\frac 1 2 {(d^2 - d)}  + 7 d + 1\right)^B$ many absolutely convergent integrals of the following form
\begin{equation*}
(ix)^{-B} P(\boldsymbol \nu') \int_{\BR_+^d} \ut^{\ualpha}\partial^{\ualpha} h (\ut) p_{\boldsymbol \nu''}(\ut) \Theta(\ut; \usigma)^{-B-B_2} e^{i x \theta(\ut; \usigma)} d\ut,
\end{equation*}
where  $|\ualpha| + B_1 + B_2 = B$ {\rm($\ualpha \in \BN^d $)}, $P $ is a polynomial  of degree $B_1$ and integer coefficients of size $O_{B, d} (1)$, and $\boldsymbol \nu'' \in \boldsymbol \nu' + \BZ^d$ satisfies $|\nu''_{l     }  - \nu'_{l     }| \leq  B + 2 B_2$ for all $l      = 1, ..., d$. Recall that in the multi-index notation $|\ualpha| = \sum_{l      = 1}^d \alpha_l     $, $\ut^{\ualpha} = \prod_{l      = 1}^d t_{l     }^{\alpha_l     }$ and $\partial^{\ualpha} = \prod_{l      = 1}^d \partial_{l     }^{\alpha_l     } $.
\end{lem}
Define $\mathfrak c = \max  \left\{ | \nu_l      | \right \} + 1$ and  $\mathfrak r = \max  \left\{ |\Re \nu_l      | \right \}$.  Suppose that  $x \geq \mathfrak c$.  Applying Lemma \ref{lem: J(x; sigma; h) for rho neq 0} and \ref{lem: lower bound for Theta} to the $J$-integrals in (\ref{3eq: linear combination first case}, \ref{3eq: linear combination second case}, \ref{3eq: linear combination third case}), one obtains the estimate
\begin{equation*} 
J_{\unu}(x; \usigma; \urho) \lll_{\mathfrak r, M, d} \lp \frac {\mathfrak c} x \rp^M ,
\end{equation*} 
for   any given nonnegative integer $M $. 
Slight modifications of the above arguments yield a similar estimate for the derivative
\begin{equation}\label{5eq: bound for J nu (j) (x; sigma; rho) rho neq 0}
J^{(j)}_{\unu}(x; \usigma; \urho) \lll_{\mathfrak r, M, j, d} \lp \frac {\mathfrak c} x \rp^M .
\end{equation} 

\begin{rem}
Our proof of \eqref{5eq: bound for J nu (j) (x; sigma; rho) rho neq 0} is similar to that of \cite[Theorem 7.7.1]{Hormander}. Indeed, $\Theta (\ut; \usigma) $ plays the same role as $|f'|^2 + \Im f$ in the proof of \cite[Theorem 7.7.1]{Hormander}, where $f$ is the phase function there. The non-compactness of $K_{\urho}$ however prohibits the application of \cite[Theorem 7.7.1]{Hormander} to the $J$-integral in \eqref{5eq: J-integral} in our case.

\end{rem}

\subsection{Rapid decay of $K$-Bessel functions} Suppose that   there exists $ k \in \{ 1, ..., d\}$ such that  $ \varsigma_{k} \neq \varsigma_{d+1}$. Then for any $\ut \in K_{\boldsymbol 0}$  
$$ 
\left|\varsigma _{d+1} t_1 ... t_d - \varsigma _{k} t_{k}^{-1} \right| > t_{k}\- \geq \frac 1 4.$$
Similar to the arguments in \S \ref{sec: Bound for rho neq 0}, repeating the $k$-th partial integration of the third kind yields
the same bound (\ref{5eq: bound for J nu (j) (x; sigma; rho) rho neq 0}) in the case $\urho = \boldsymbol 0$.

\begin{rem}
For this, we may also directly apply \cite[Theorem 7.7.1]{Hormander}.
\end{rem}

\begin{thm}\label{thm: Bessel functions of K-type}
Let $\mathfrak c = \max  \left\{ | \nu_l      | \right \} + 1$ and  $\mathfrak r = \max  \left\{ |\Re \nu_l      | \right \}$. Let $j$ and $M$ be nonnegative integers. Suppose that one of the signs $\varsigma_1, ..., \varsigma_d$ is different from $ \varsigma_{d+1}$. 
Then  
\begin{equation*} 
 J_{\unu}^{(j)} (x;  \usigma) \lll_{\mathfrak r, M, j, d} \lp \frac {\mathfrak c} x \rp^M
\end{equation*}
for any $x \geq \fc$. In particular, $J_{\unu} (x;  \usigma)$ is a Schwartz function at infinity, namely, all derivatives $ J_{\unu}^{(j)} (x;  \usigma)$ rapidly decay at infinity.
\end{thm}

\subsection{Asymptotic expansions of $H$-Bessel functions}\label{sec: asymptotic expansions of H Bessel functions}

In the following, we shall adopt the convention $(\pm i)^{a} = e^{ \pm \frac 1 2 i \pi a}$, $a \in \BC$.

We first introduce the function $W^{\pm}_{\unu} (x )$, which is closely related to the   Whittaker  function of imaginary argument if $d=1$ (see \cite[\S 17.5, 17.6]{Whittaker-Watson}), defined by
\begin{equation*} 
W^{\pm}_{\unu} (x ) = (d+1)^{\frac 1 2} (\pm 2 \pi i)^{- \frac d 2} e^{\mp i (d+1) x} H^{\pm}_{\unu} (x).
\end{equation*}
Write
$H^\pm_{\unu} (x; \urho) = J_{\unu} (x; \pm,..., \pm; \urho)$ and define
\begin{equation*} 
W^{\pm}_{\unu} (x;  \urho) = (d+1)^{\frac 1 2} (\pm 2 \pi i)^{- \frac d 2} e^{\mp i (d+1) x} H^{\pm}_{\unu} (x; \urho).
\end{equation*}
For $\urho \neq \boldsymbol 0$, the bound (\ref{5eq: bound for J nu (j) (x; sigma; rho) rho neq 0}) for $H^\pm_{\unu} (x; \urho)$ is also valid for $ W^{\pm}_{\unu} (x;  \urho)$. Therefore, we are left with analyzing $W^{\pm}_{\unu} (x; \boldsymbol 0)$. We have
\begin{equation}\label{5eq: W nu (j) (x; pm; 0) integral}
\begin{split}
W^{\pm, (j)}_{\unu} (x; \boldsymbol 0) =\ & (d+1)^{\frac 1 2} (\pm 2 \pi i)^{- \frac d 2} (\pm i)^j \\
& \int_{K_{\boldsymbol 0}} \lp \theta (\ut) - d - 1 \rp^j h_{\boldsymbol 0} (\ut) p_{\unu} (\ut) e^{\pm i x \lp \theta (\ut) - d - 1 \rp} d\ut,
\end{split}
\end{equation}
with \begin{equation}\label{5eq: theta (t)}
\theta (\ut) = \theta (\ut; +, ..., +) = t_1 ... t_d + \sum_{l      = 1}^d t_l     \-.
\end{equation}
\begin{prop} \label{prop: Stationary phase}
\cite[Theorem 7.7.5]{Hormander}.
Let $K \subset \BR^d$ be a compact set, $X$ an open neighbourhood of $K$ and $M$ a nonnegative integer. If $u(\ut) \in C^{2M}_0 (K)$, $f(\ut) \in C^{3M + 1} (X)$ and $\Im f \geq 0$ in $X$, $\Im f (\ut_0) = 0$, $ f' (\ut_0) = 0$, $\det f'' (\ut_0) \neq 0$ and $ f'  \neq 0$ in $K \smallsetminus \{\ut_0 \}$, then for $x > 0$
\begin{equation*}
\begin{split}
&\left| \int_{K} \hskip -3 pt  u(\ut) e^{i x f(\ut)} d\ut \hskip -1 pt  -  \hskip -1 pt e^{i x f(\ut_0)} \hskip -1 pt \left((2\pi i)^{- d } \det f'' (\ut_0) \right)^{ - \frac 1 2} \hskip -3 pt \sum_{ m = 0 }^{M-1} x^{- m - \frac d 2} \EuScript L_m u \right|   \lll    x^{-M} \hskip -3 pt \sum_{|\ualpha| \leq 2M} \sup\left| D^{\ualpha} u \right|.
\end{split}
\end{equation*}
Here the implied constant depends only on $M$, $f$, $K$ and  $d$.
With
$$g  (\ut) = f(\ut) - f(\ut_0) - \frac 1 2 \left\langle f'' (\ut_0) (\ut - \ut_0), \ut - \ut_0 \right\rangle $$
which vanishes of third order at $\ut_0$, we have
\begin{equation*}
\EuScript L_m u =  i^{- m}\sum_{r = 0}^{2 m} \frac 1 { 2^{m+r} (m+r) !r! } \left \langle  f'' (\ut_0)\- D, D \right \rangle^{m+r} \lp g^r  u \rp (\ut_0).\footnote{According to H\"ormander, $D = - i (\partial_1, ..., \partial_d)$. }
\end{equation*}
This is a differential operator of order $2 m$ acting on $u$ at $\ut_0$. The coefficients are rational homogeneous functions of degree $- m$ in
$ f'' (\ut_0)$, ..., $ f^{(2m+2)} (\ut_0)$ with denominator $(\det f'' (\ut_0))^{3m}$. In every term the total number of derivatives of $u$ and of $f''$ is at most $2 m$.
\end{prop}

We now apply Proposition \ref{prop: Stationary phase} to the integral in \eqref{5eq: W nu (j) (x; pm; 0) integral}. For this, we let
\begin{align*} 
& K = K_{\boldsymbol 0} = \textstyle  \left[ \frac 1 4, 4\right ]^{d}, \hskip 10 pt X = \textstyle  \left( \frac 1 5, 5\right )^{d},\\
& f (\ut) =  \pm \left( \theta (\ut) - d - 1 \right), \hskip 10 pt f' (\ut) = \pm \left( t_1  ...\widehat{t_{l     }} ... t_{d} - t_{l     }^{-2} \right)_{l      = 1}^{d}, \hskip 10 pt \ut_0 = (1, ..., 1),\\
& f'' (\ut_0) = \pm \begin{pmatrix}
2 & 1 & \cdots & 1\\
1 & 2 & \cdots & 1\\
\vdots & \vdots & \ddots &  \vdots\\
1 & 1 & \cdots & 2
\end{pmatrix}, \hskip 10 pt \det  f'' (\ut_0) = (\pm)^d ( d+1 ), \hskip 10 pt g (\ut) = \pm G(\ut),\\
& f'' (\ut_0)\- = \pm \frac 1 {d+1} \begin{pmatrix}
d & -1 & \cdots & -1\\
-1 & d & \cdots & -1\\
\vdots & \vdots & \ddots &  \vdots\\
-1 & -1 & \cdots & d
\end{pmatrix},\\
& u(\ut) = (d+1)^{\frac 1 2} (\pm 2 \pi i)^{- \frac d 2} (\pm i)^j \left( \theta (\ut) - d - 1 \right)^j p_{\unu} (\ut) h_{\boldsymbol 0} (\ut),
\end{align*}
with
\begin{equation}\label{5eq: G(t)}
\begin{split}
G(\ut) =   t_1 ... t_{d} + \sum_{l      = 1}^{d} \left( - t_{l     } ^2 + (d+1) t_{l     } + t_{l     }\- \right)   - \sum_{ 1 \leq l      < k \leq d } t_{l     } t_{k} - \frac {(d+1)(d+2)} 2 .
\end{split}
\end{equation}
Proposition \ref{prop: Stationary phase} yields the following asymptotic expansion of $W^{\pm, (j)}_{\unu} (x; \boldsymbol 0)$,
\begin{equation*} 
W^{\pm, (j)}_{\unu} (x; \boldsymbol 0) = \sum_{m = 0 }^{M-1} ( \pm i )^{j - m} B_{m, j} (\unu) x^{- m - \frac d 2} + O_{\fr, M, j, d} \left( \frc^{2 M} x^{-M}\right), \hskip 10 pt x > 0,
\end{equation*}
with 
\begin{equation}\label{5eq: B mj}
B_{m, j} (\unu) =  \sum_{r=0}^{2m} \frac {(-)^{m+r} \EuScript L^{m+r} \left( G^r (\theta - d - 1)^j p_{\unu} \right) (\ut_0)} { (2 (d+1))^{m+r} (m+r)! r! } ,
\end{equation}
where $\EuScript L $ is the second-order differential operator given by
\begin{equation}\label{5eq: differential operator D}
\EuScript L = d \sum_{l      = 1}^{d} \partial_{l     }^2 - 2 \sum_{ 1 \leq l      < k \leq d } \partial_{l     } \partial_{k} .
\end{equation} 

\begin{lem}\label{5lem: coefficient}  We have $B_{m, j} (\unu)  = 0$ if $  m < j$.
Otherwise,  $B_{m, j} (\unu) \in \BQ[\unu]$ is a symmetric polynomial  of degree $2m-2j$. 
In particular,
$B_{m, j} (\unu) \lll_{m, j, d} \frc^{2 m - 2 j}$ for $m \geq j$.
\end{lem}

\begin{proof}
The symmetry of  $ B_{m, j} (\unu) $ is clear from definition. 
Since $ \theta - d - 1 $ vanishes of second order at $\ut_0$, $2 j$ many derivatives are required to remove the zero of $ \left( \theta - d - 1 \right)^j $ at $\ut_0$. From this, along with  the descriptions of the differential operator $\EuScript L_m$ in Proposition \ref{prop: Stationary phase}, one proves the lemma. 
\end{proof}

Furthermore, in view of the bound (\ref{5eq: bound for J nu (j) (x; sigma; rho) rho neq 0}), the total contribution to $W_{\unu}^{\pm, (j)}(x )$ from  all the $W_{\unu}^{\pm, (j)}(x;  \urho)$ with $\urho \neq \boldsymbol 0$ is of size $O_{\mathfrak r, M, j, d} \lp   {\mathfrak c}^{ M} x^{- M} \rp$ and hence may be absorbed into the error term in the asymptotic expansion of $W^{\pm, (j)}_{\unu} (x; \boldsymbol 0)$. 

In conclusion, the following proposition is established. 

\begin{prop}\label{prop: asymptotics for W nu (j) (x; pm)} Let $M$, $j$ be nonnegative integers such that $M \geq j$.
Then for $x \geq \frc$ 
\begin{equation*} 
W_{\unu}^{\pm, (j)}(x ) = \sum_{ m = j }^{M-1} ( \pm i )^{j - m} B_{m, j} (\unu) x^{- m - \frac d 2} + O_{\fr, M, j, d} \left( \frc^{2 M } x^{-M}\right).
\end{equation*}
\end{prop}


\begin{cor}\label{cor: W(j)}
Let $N$, $j$ be nonnegative integers such that $N \geq j$, and let $\epsilon > 0$.

{ \rm (1).}  We have
$ W_{\unu}^{\pm, (j)}(x ) \lll_{\fr, j, d} \fc^{2j} x^{-j}$ for $x \geq \frc $.

{ \rm (2).} If $x \geq \frc^{2 + \epsilon}$, then
\begin{equation*} 
W_{\unu}^{\pm, (j)}(x ) = \sum_{ m = j }^{N-1} ( \pm i )^{j - m} B_{m, j} (\unu) x^{- m - \frac d 2} + O_{\fr, N, j, \epsilon, d} \left( \frc^{2 N } x^{-N - \frac d 2}\right).
\end{equation*}
\end{cor}
\begin{proof}
On letting $M = j$, Proposition \ref{prop: asymptotics for W nu (j) (x; pm)} implies (1). On  choosing $M$ sufficiently large so that $(2 + \epsilon) \lp M-N + \frac d 2\rp \geq 2 (M-N) $, Proposition \ref{prop: asymptotics for W nu (j) (x; pm)} and Lemma \ref{5lem: coefficient} yield
\begin{align*} 
& W_{\unu}^{\pm, (j)}(x ) - \sum_{ m = j }^{N-1} ( \pm i )^{j - m} B_{m, j} (\unu) x^{- m - \frac d 2} \\
= \ & \sum_{ m = N }^{M-1} ( \pm i )^{j - m} B_{m, j} (\unu) x^{- m - \frac d 2} + O_{\fr, j, M, d} \left( \frc^{2 M } x^{-M}\right) = O_{\fr, j, N, \epsilon, d} \left( \frc^{2 N } x^{-N - \frac d 2}\right).
\end{align*}
\end{proof}

Finally,   the asymptotic expansion of $H^{\pm} (x; \ulambda)$($ = H^{\pm}_{\unu} (x)$) is formulated as below. 

\begin{thm}\label{thm: asymptotic expansion}
Let $\mathfrak C =  \max  \left\{ | \lambda_l      | \right \} + 1$ and $\mathfrak R = \max  \left\{ |\Re \lambda_l      | \right \}$. Let $M$ be a nonnegative integer.

{\rm (1).} Define $W^{\pm} (x; \ulambda) = \sqrt n (\pm 2 \pi i)^{- \frac { n-1} 2} e^{\mp i n x} H^{\pm}  (x; \ulambda)$. Let $M \geq j \geq 0$. Then \begin{equation*} 
W^{\pm, (j)} (x; \ulambda) = \sum_{m = j }^{M-1} ( \pm i )^{j - m} B_{m, j} (\ulambda) x^{- m - \frac { n-1} 2} + O_{\mathfrak R, M, j , n} \left( \fC^{2 M } x^{-M}\right)
\end{equation*} 
for all $x \geq \mathfrak C $.
Here $B_{m, j} (\ulambda) \in \BQ[\ulambda]$ is a symmetric polynomial in $\ulambda$ of degree $2m$, with $B_{0, 0} (\ulambda) = 1$. The coefficients of $B_{m, j} (\ulambda)$ depends only on $m$, $j$ and $d$.

{\rm (2).} Let $B_{m } (\ulambda) = B_{m, 0} (\ulambda)$. Then for $x \geq \mathfrak C $  
\begin{equation*}
\begin{split}
H^{\pm}  (x; \ulambda) = n^{- \frac 1 2}   (\pm 2 \pi i)^{ \frac {n-1}  2} e^{ \pm i n x} x^{ - \frac {n-1} 2}  \hskip - 3 pt \lp  \sum_{m=0}^{M-1} ( \pm i )^{ - m} B_{m} (\ulambda) x^{- m} + O_{\mathfrak R, M, d} \left( \mathfrak C^{2 M} x^{-M + \frac {n-1} 2} \hskip -1 pt \right) \hskip -2 pt \rp \hskip -2 pt.
\end{split}
\end{equation*}
\end{thm}

\begin{proof}
This theorem is a direct consequence of Proposition \ref{prop: asymptotics for W nu (j) (x; pm)} and Lemma \ref{5lem: coefficient}.
It is only left to verify the symmetry of $B_{m, j}(\ulambda) = B_{m, j} (\unu)$ with respect to $\ulambda$. Indeed, in view of (\ref{2eq: definition of G(s;  sigma, lambda)}, \ref{2eq: definition of J (x; sigma)}), $H^{\pm} (x; \ulambda)  $ is symmetric with respect to $\ulambda$, so $B_{m, j}(\ulambda) $ must be represented by a symmetric polynomial in $\ulambda$ modulo $\sum_{l     =1}^{d+1} \lambda_l     $.
\end{proof}

\begin{cor}\label{cor: H pm}
Let $M$ be a nonnegative integer, and let $\epsilon > 0$.
Then for $x \geq  \mathfrak C^{2 + \epsilon}$  
\begin{equation*}
\begin{split}
H^{\pm}  (x; \ulambda) =   n^{- \frac 1 2} (\pm 2 \pi i)^{ \frac {n-1} 2} e^{ \pm i n x} x^{ - \frac {n-1} 2} 
  \lp \sum_{m=0}^{M-1} ( \pm i )^{ - m} B_{m} (\ulambda) x^{- m} + O_{\mathfrak R, M, \epsilon, n} \left( \mathfrak C^{2 M} x^{-M  } \right) \rp.
\end{split}
\end{equation*}
\end{cor}

\subsection{Concluding remarks}

\subsubsection{On the analytic continuations of $H$-Bessel functions}\label{sec: Analytic continuations of the H-Bessel functions}
Our observation is that the phase function $\theta $ defined by (\ref{5eq: theta (t)}) is {\it always positive} on $\BR^{d}_+$. It follows that if one replaces $x$  by $z = x e^{i\omega}$, with $x > 0$ and $0 \leq \pm \omega \leq \pi$, then the various $J$-integrals in the rigorous interpretation of $H^{\pm}_{\unu} (z)$ remain  absolutely convergent, uniformly with respect to $z$, since $\left| e^{\pm i z \theta(\ut)} \right| = e^{\mp x \sin \omega \, \theta(\ut)} \leq 1.$ Therefore, the resulting integral $H^{\pm}_{\unu} (z)$ gives rise to an analytic continuation of $H^{\pm}_{\unu} (x)$ onto the half-plane $\BH^{\pm} = \left\{ z \in \BC \smallsetminus \{0\} : 0 \leq \pm \arg z \leq \pi \right\} $. In view  of Proposition \ref{prop: J (x ; sigma, lambda) = J nu (x; sigma)}, one may define $H^{\pm} (z; \ulambda) = H^{\pm}_{\unu} (z)$ and regard it as the analytic continuation of  $H^{\pm} (x; \ulambda)$ from $\BR _+$ onto $\BH^{\pm}$. Furthermore, with slight modifications of the arguments above, where the phase function $f $ is now chosen to be $\pm e^{i \omega} (\theta - d - 1)$ in the application of Proposition \ref{prop: Stationary phase},  the domain of validity for the asymptotic expansions in Theorem \ref{thm: asymptotic expansion} may be extended from $\BR_+$ onto $  \BH^{\pm}$. For example, we have 
\begin{equation}\label{5eq: asymptotic expansion 1}
\begin{split}
H^{\pm}  (z; \ulambda) = n^{- \frac 1 2}  &(\pm 2 \pi i)^{ \frac { n-1} 2} e^{ \pm i n z} z^{ - \frac { n-1} 2} \\
& \lp \sum_{m=0}^{M-1} ( \pm i )^{ - m} B_{m} (\ulambda) z^{- m} + O_{\mathfrak R, M, n} \left( \mathfrak C^{2 M} |z|^{-M + \frac { n-1} 2} \right) \rp,
\end{split}
\end{equation}
for all $z \in \BH^\pm$ such that $ |z| \geq \fC$.

Obviously, the above method of obtaining the analytic continuation of $H^{\pm}_{\unu} $ does not apply to $K$-Bessel functions.

\subsubsection{On the asymptotic of the Bessel kernel $J_{(\ulambda, \udelta)}$}

As in \eqref{2eq: Bessel kernel}, $ J_{(\ulambda, \udelta)} (\pm x) $ is a combination of $J \big(2 \pi x^{\frac 1 n}; \usigma, \ulambda \big) $, and hence its asymptotic   follows immediately from Theorem \ref{thm: Bessel functions of K-type} and \ref{thm: asymptotic expansion}. For convenience of reference, we record the asyptotic  of $J_{(\ulambda, \udelta)} (\pm x)$ in the following theorem.

\begin{thm}\label{thm: asymptotic Bessel kernel, 1}
	Let $(\ulambda, \udelta) \in \BL^{n-1} \times (\BZ/2 \BZ)^{n}$.
	Let $M  \geqslant 0 $.  
	Then, for $x > 0$,  we may write  
	\begin{align*}
	& J_{(\ulambda, \udelta)} \left(x^n \right)   =  \sum_{\pm} \frac  { { (\pm)^{ \sum  \delta_l    }   e \lp   \pm   \lp n x  +   \frac {n-1} 8 \rp  \rp }} {n^{\frac 1 2} x^{  \frac {n-1} 2} }   W_{ \ulambda }^{\pm} (x) + E^+_{(\ulambda, \udelta)} (x), \\
	& J_{(\ulambda, \udelta)} \left( - x^n \right)   = E^-_{(\ulambda, \udelta)} (x),  
	\end{align*}
	if $n$ is even, and
	\begin{align*}
	J_{(\ulambda, \udelta)} \left(\pm x^n \right) & =  \frac  { { (\pm)^{\sum  \delta_l     }    e \lp   \pm   \lp n x  +   \frac {n-1} 8  \rp \rp }} {n^{\frac 1 2} x^{  \frac {n-1} 2} }   W_{ \ulambda }^{\pm} (x) + E^\pm_{(\ulambda, \udelta)} (x),   
	\end{align*}
	if $n$ is odd, such that
	\begin{align*}
	W_{ \ulambda }^{\pm } (x) =    \sum_{m=0}^{M-1}  B^{\pm}_{m } (\ulambda) x^{-   m  }  
	+ O_{\mathfrak R,  M ,  n} \left( \fC^{2 M } x^{-  M +  \frac {n-1} 2}\right), 
	\end{align*} 
	and
	\begin{equation*}
	E^{\pm }_{(\ulambda, \udelta)} \left( x \right) =  O_{\mathfrak R, M,  n} \lp    {\mathfrak C}^{M} x^{-   M } \rp,
	\end{equation*}
	for $x \geqslant \mathfrak C $. With the notations in Theorem {\rm \ref{thm: asymptotic expansion}}, we have $W_{ \ulambda }^{\pm } (x) = (2 \pi x)^{\frac {n-1} 2 } W^{\pm} (2 \pi x; \ulambda)$ and $B^{\pm}_{m } (\ulambda) = (\pm 2 \pi i)^{ -m} B_{m } (\ulambda)$.

\end{thm}

\subsubsection{On the implied constants of estimates} All the implied constants that occur in this section are of exponential dependence on the real parts of the indices. 
If one considers  the $d$-th symmetric lift of a holomorphic Hecke cusp form of weight $k$, the estimates are particularly awful in the $k$ aspect. 


In \S \ref{sec: Recurrence relations and differential equations of the Bessel functions} and \S \ref{sec: Bessel equations}, we shall further explore the theory of Bessel functions from the perspective of differential equations.  As a consequence, we may prove that  if the argument is sufficiently large then all the estimates in this section can be improved so that the dependence on the index can be completely eliminated.

\subsubsection{On the coefficients in the asymptotics} \label{sec: remark on the coefficients in the asymptotics}
One feature of the method of stationary phase is the  explicit formula of the coefficients in the asymptotic expansion in terms of  certain partial differential operators.  In the present case of $H^{\pm}  (x; \ulambda) = H^{\pm}_{\unu} (x)$,
\eqref{5eq: B mj} provides an explicit formula of $B_{m} (\ulambda) = B_{m, 0} (\unu)$. 
To compute $\EuScript L^{m+r} \left( G^r p_{\unu} \right) (\ut_0)$ appearing in \eqref{5eq: B mj},  we observe that the function $G$ defined in (\ref{5eq: G(t)}) does not only vanish of third order at $\ut_0$. Actually, $\partial^{\ualpha} G (\ut_0) $ vanishes except for $ \ualpha = (0, ..., 0, \alpha , 0 ..., 0)$, with $\alpha \geq 3$. In the exceptional case we have $\partial^{\ualpha} G (\ut_0) = (- )^{\alpha} \alpha !$. However, the resulting expression is  considerably complicated. 

When $d = 1$, we have $\EuScript L = (d/dt)^2$. For $2 m \geq r \geq 1$,
\begin{align*}
 & \lp d/ dt \rp^{2m+2r} \lp G^r p_{\nu}\rp (1) \\
 = &\, (2m+2r)! \sum_{\alpha = 0}^{2m-r} \left| \left\{ (\alpha_1, ..., \alpha_r) : \sum_{q = 1}^r \alpha_q = 2m+2r- \alpha, \alpha_q \geq 3 \right \} \right| \frac { (- )^{\alpha}[ \nu - 1 ]_\alpha} {\alpha !} \\
 = &\, (2m+2r)! \sum_{\alpha = 0}^{2m-r} {2m-\alpha-1 \choose r-1} \frac {  (1 - \nu)_{\alpha} } {\alpha !}.
\end{align*}
Therefore \eqref{5eq: B mj} yields 
\begin{equation*}
\begin{split}
B_{m, 0} (\nu) = \lp -\frac {1} {4}\rp^m \lp \frac { (1 - \nu)_{2m}} {m!} + \sum_{r=1}^{2m} \frac {(-)^r (2m+2r)!} {4^r (m+r)! r!} \sum_{\alpha = 0}^{2m-r} {2m-\alpha-1 \choose r-1} \frac { (1 - \nu)_{\alpha}} { \alpha !} \rp.
\end{split}
\end{equation*}
However, this expression of  $ B_{m, 0} (\nu)$ is more involved than what is known in the literature. 
Indeed, we have the  asymptotic expansions of $H^{(1)}_{\nu} $ and  $H^{(2)}_{\nu} $ (\cite[7.2 (1, 2)]{Watson})
\begin{align*}
H_{\nu}^{(1, 2)} (x)  \sim \lp \frac 2 {\pi x} \rp^{\frac 1 2} e^{\pm i \lp x - \frac 1 2 \nu \pi - \frac 1 4\pi \rp} \lp \sum_{m=0}^\infty \frac { (\pm)^m \lp \frac 1 2 - \nu \rp_m  \lp \frac 1 2 + \nu \rp_m} { m! (2 i x)^m} \rp,  
\end{align*}
which are deducible from Hankel's integral representations (\cite[6.12 (3, 4)]{Watson}). 
In view of Proposition \ref{prop: Classical Bessel functions} and Theorem \ref{thm: asymptotic expansion}, we have 
\begin{equation*}
B_{m, 0} (\nu) =  \frac {\lp \frac 1 2 - \nu \rp_m  \lp \frac 1 2 + \nu \rp_m } {4^m m!}.
\end{equation*}
Therefore, we deduce the following  combinatoric identity 
\begin{equation}
\begin{split}
  \frac { (- )^m  \lp \frac 1 2 - \nu \rp_m  \lp \frac 1 2 + \nu \rp_m } {m!}  
& =   \frac { (1 - \nu)_{2m}} {m!} \\
& + \sum_{r=1}^{2m} \frac {(- )^r  (2m+2r)!} {4^r (m+r)! r!} \sum_{\alpha = 0}^{2m-r}{2m-\alpha-1 \choose r-1} \frac {(1 - \nu)_{\alpha}} {\alpha !}.
\end{split}
\end{equation}
It seems however hard to find an elementary proof of this identity.

\section{Recurrence formulae and the differential equations for Bessel functions}
\label{sec: Recurrence relations and differential equations of the Bessel functions}

Making use of certain recurrence formulae for  $J_{\unu} (x; \usigma)$, 
we shall derive the differential equation satisfied by $J (x; \usigma, \ulambda)$.  

\subsection{Recurrence formulae}

Applying the formal partial integrations of either the first or the second kind and the differentiation under the integral sign on the formal integral expression of $J_{\unu} (x; \usigma)$ in (\ref{3eq: definition J nu (x; sigma)}), one obtains the recurrence formulae
\begin{equation}\label{6eq: recurrence from partial integration}
  {\nu_l     } ({i x})\-  J_{\unu} (x; \usigma) = \varsigma_l      J_{\unu  - \ue_l     } (x; \usigma) - \varsigma_{d + 1} J_{\unu + \ue^d} (x; \usigma)
\end{equation}
for $l      = 1, ..., d$, and
\begin{equation}\label{6eq: reccurence from differentiation}
J'_{\unu} (x; \usigma) = \varsigma_{d + 1} i J_{\unu + \ue^d} (x; \usigma) + i \sum_{l      = 1}^{d} \varsigma_l      J_{\unu - \ue_l     } (x; \usigma).
\end{equation}
It is easy to verify (\ref{6eq: recurrence from partial integration}) and (\ref{6eq: reccurence from differentiation}) using the rigorous interpretation of $J_{\unu} (x; \usigma)$ established in \S \ref{sec: Definition of BJ nu (x; sigma)}.
Moreover, using (\ref{6eq: recurrence from partial integration}), one may reformulate  (\ref{6eq: reccurence from differentiation})  as below,
\begin{equation}
J'_{\unu} (x; \usigma) \label{6eq: reccurence second form}
= \varsigma_{d + 1} i (d + 1) J_{\unu + \ue^d} (x; \usigma) + \frac {\sum_{l      = 1}^{d} \nu_l     } { x} J_{\unu} (x; \usigma).
\end{equation}

\subsection{The differential equations}


\begin{lem}\label{lem: J (j) nu (x; sigma) derivatives}
Define $\ue^{l     } = (\underbrace {1,..., 1}_{l      }, 0 ..., 0)$,  $l      = 1, ..., d$, and denote $\ue^0 = \ue^{d+1} = (0, ..., 0)$ for convenience. Let $\nu_{d+1} = 0$.

{ \rm (1).} For $l      = 1, ..., d+1$ we have
\begin{equation}\label{6eq: J' nu+e ell (x; sigma)}
J'_{\unu + \ue^{l     }} (x; \usigma) = \varsigma_{l     } i (d + 1) J_{\unu + \ue^{l     -1}} (x; \usigma) - \frac {\varLambda_{d-l     +1} (\unu) + d - l      +1} { x} J_{\unu + \ue^{l     }} (x; \usigma),
\end{equation} 
with
\begin{equation*} 
 \varLambda_{m } (\unu) = - \sum_{k = 1}^{d}\nu_{k} + (d+1) \nu_{d-m+1}, \hskip 10 pt   m = 0, ..., d.
\end{equation*}

{ \rm (2).} For $0 \leq j \leq k \leq d+1$ define
\begin{equation*}
U_{k, j} (\unu) = 
\begin{cases}
1, & \text { if } j=k,\\
- \lp \varLambda_{j} (\unu) + k - 1 \rp U_{k - 1, j} (\unu) + U_{k-1, j-1} (\unu), & \text { if } 0 \leq j \leq k-1,
\end{cases}
\end{equation*}
with the notation $U_{k, -1} (\unu) = 0$, and
\begin{equation*}
S_0 (\usigma) = +, \hskip 10 pt S_j (\usigma) = \prod_{m = 0}^{j - 1}\varsigma_{d-m+1}\ \text{ for } j = 1, ..., d+1.
\end{equation*}
Then
\begin{equation} \label{6eq: J (j) nu (x; sigma) derivatives}
J^{(k)}_{\unu}  (x; \usigma) = \sum_{j = 0}^k  S_j (\usigma) (i (d+1))^j U_{k, j} (\unu) x^{j - k} J_{\unu + \ue^{d - j + 1}}  (x; \usigma).
\end{equation}
\end{lem}

\begin{proof}
By (\ref{6eq: reccurence second form}) and (\ref{6eq: recurrence from partial integration}),   
\begin{equation*}
\begin{split}
& J'_{\unu + \ue^{l     }} (x; \usigma) = \varsigma_{d + 1} i (d + 1) J_{\unu + \ue^{l     } + \ue^d} (x; \usigma) + \frac {\sum_{k = 1}^{d}\nu_{k} + l     } { x} J_{\unu + \ue^{l     }} (x; \usigma) \\
= &\ i (d + 1) \left(- \frac {\nu_{l     } + 1} {i x} J_{\unu + \ue^{l     }} (x; \usigma) + \varsigma_{l     } J_{\unu + \ue^{l      - 1}} (x; \usigma) \right) + \frac {\sum_{k = 1}^{d}\nu_{k} + l     } { x} J_{\unu + \ue^{l     }} (x; \usigma)\\
= &\ \varsigma_{l     } i (d + 1)  J_{\unu + \ue^{l     -1}} (x; \usigma) + \frac {\sum_{k = 1}^{d}\nu_{k} - (d+1) \nu_l      + l      - d - 1} { x} J_{\unu + \ue^{l     }} (x; \usigma).
\end{split}
\end{equation*}
This proves \eqref{6eq: J' nu+e ell (x; sigma)}.

(\ref{6eq: J (j) nu (x; sigma) derivatives}) is trivial when $k=0$. Suppose that $k \geq 1$ and that (\ref{6eq: J (j) nu (x; sigma) derivatives}) is already proven for $k-1$. The inductive hypothesis and  (\ref{6eq: J' nu+e ell (x; sigma)}) imply
\begin{align*}
J^{(k)}_{\unu}  (x; \usigma) = &  \sum_{j=0}^{k-1}  S_j(\usigma) (i(d+1))^j U_{k-1, j} (\unu) x^{j-k+1} \\
& \lp (j-k+1) x^{-1} J_{\unu + \ue^{d-j+1}} (x; \usigma) \right. \\
& \left. \quad + \varsigma _{d-j+1} i (d+1) J_{\unu + \ue^{d-j}}  (x; \usigma) - (\varLambda_j(\unu) + j) x\- J_{\unu + \ue^{d-j+1}} (x; \usigma)  \rp\\
= & - \sum_{j=0}^{k-1} S_j(\usigma) (i(d+1))^j U_{k-1, j} (\unu) (\varLambda_j(\unu) + k-1) x^{j-k} J_{\unu + \ue^{d-j+1}} (x; \usigma)\\
& + \sum_{j=1}^{k } S_{j-1}(\usigma) \varsigma _{d-j+2} (i(d+1))^j U_{k-1, j-1} (\unu) x^{j-k} J_{\unu + \ue^{d-j+1}} (x; \usigma).
\end{align*}
Then (\ref{6eq: J (j) nu (x; sigma) derivatives}) follows from the definitions of $ U_{k, j} (\unu)$ and  $ S_j (\usigma) $.
\end{proof}

Lemma \ref{lem: J (j) nu (x; sigma) derivatives} (2) may be recapitulated as
\begin{equation}\label{6eq: differentials, matrix form}
X_{\unu} (x; \usigma)
= 
D(x)\- U ( \unu) D(x) S(\usigma) Y_{\unu}  (x; \usigma),
\end{equation}
where $X_{\unu} (x; \usigma) = \big( J^{(k)}_{\unu} (x; \usigma) \big)_{k=0}^{ d+1 }$ and $Y_{\unu}  (x; \usigma) = \big(  J_{\unu + \ue^{d - j + 1}}  (x; \usigma) \big)_{j=0}^{d+1}$ are column vectors of functions, 
$S(\usigma) = \mathrm{diag} \left(S_j(\usigma)(i(d+1))^j \right)_{j=0}^{d+1}$ and $D(x) = \mathrm{diag} \left(x^j \right)_{j=0}^{d+1}$ are diagonal matrices, and $U ( \unu) $ is the lower triangular unipotent $(d+2) \times (d+2)$ matrix whose $(k+1, j+1)$-th entry is equal to $U_{k, j} (\unu)$. 
The inverse matrix $U ( \unu)\-$ is again a lower triangular unipotent matrix. Let  $ V_{k, j} (\unu) $ denote the $(k+1, j+1)$-th entry of  $U ( \unu)\-$. It is evident that $ V_{k, j} (\unu) $ is a polynomial in $\unu$ of  degree $k-j$ and integral coefficients.

Observe that $J_{\unu + \ue^{d+1}} (x; \usigma) = J_{\unu + \ue^0} (x; \usigma) = J_{\unu} (x; \usigma)$. Therefore,  (\ref{6eq: differentials, matrix form}) implies that $J_{\unu} (x; \usigma)$ satisfies the following linear differential equation of order $d+1$
\begin{equation} \label{6eq: differential equation of J}
\begin{split}
\sum_{j = 1}^{d+1} V_{d+1, j} (\unu)   x^{j - d - 1} w^{(j)}  
  + \left( V_{d+1, 0} (\unu) x^{- d - 1} - S_{d+1} (\usigma) (i(d+1))^{d+1} \right) w = 0.
\end{split}
\end{equation}

\subsection{Calculations of the coefficients in the differential equations}\label{sec: Calculation of the coefficients}

\begin{defn}\label{6defn: coefficients of Bessel equation}
Let $\uLambda = \{\varLambda_m \}_{m=0}^\infty$ be a sequence of complex numbers. 

{ \rm (1).} For $k, j \geq -1$ inductively define a double sequence of polynomials $ U_{k, j} (\uLambda) $ in $\uLambda$ by the initial conditions
\begin{equation*} 
U_{-1, -1}(\uLambda) = 1, \hskip 10 pt U_{k, - 1} (\uLambda) =  U_{-1, j} (\uLambda) = 0 \ \text{ if } k , j \geq 0,
\end{equation*} 
and the recurrence relation
\begin{equation}\label{6eq: definition of U k j (mu)}
U_{k, j} (\uLambda) = - \lp \varLambda_j + k - 1\rp U_{k-1, j} (\uLambda)  + U_{k-1, j-1} (\uLambda), \hskip 10 pt k, j \geq 0.
\end{equation}

{ \rm (2).} For $j, m \geq -1$ with $(j, m) \neq (-1, -1)$ define a double sequence of integers $A_{j, m}$ by the initial conditions
\begin{equation*}
A_{-1, 0} = 1, \hskip 10 pt A_{-1, m} = A_{j, -1} = 0 \  \text{ if } m \geq 1, j \geq 0,
\end{equation*}
and the recurrence relation
\begin{equation}\label{6eq: definition of a k m}
A_{j, m}  = j A_{j, m-1} + A_{j-1, m}, \hskip 10 pt j,  m \geq 0.
\end{equation}

{ \rm (3).} For $ k, m \geq 0 $ we define $\sigma_{k, m} (\uLambda)$ to be the elementary symmetric polynomial in $\varLambda_0, ..., \varLambda_{k}$ of degree $m$, with the convention that
$\sigma_{k, m} (\uLambda) = 0 $ if $m \geq k+2$. Moreover, we denote
\begin{equation*}
\sigma_{-1, 0} (\uLambda) = 1, \hskip 10 pt \sigma_{k, -1} (\uLambda) = \sigma_{-1, m} (\uLambda) = 0 \ \text{ if } k \geq -1, m \geq 1.
\end{equation*} Observe that, with the above notations as initial conditions, $\sigma_{k, m} (\uLambda)$ may also be inductively defined by the recurrence relation
\begin{equation}\label{6eq: recurrence relation of elementary symmetric polynomials}
\sigma_{k, m}(\uLambda) = \varLambda_k \sigma_{k-1, m-1} (\uLambda) + \sigma_{k-1, m} (\uLambda),  \hskip 10 pt k,  m \geq 0.
\end{equation}

{ \rm (4).} For $k \geq 0, j \geq -1$ define 
\begin{equation}\label{6eq: definition of V k j (mu)}
V_{k, j} (\uLambda) = 
\begin{cases}
0,  & \text { if } j > k,\\
\ds \sum_{m=0}^{k-j} A_{j, k-j-m} \sigma_{k-1, m} (\uLambda),  & \text{ if } k \geq j.
\end{cases}
\end{equation}
\end{defn}

\begin{lem} \label{lem: simple facts on U, a, V} Let notations be as above.

{ \rm (1).} $U_{k, j}(\uLambda)$ is a polynomial in $\varLambda_0, ..., \varLambda_j$. $U_{k, j} (\uLambda) = 0$ if $j > k $, and $U_{k, k} (\uLambda) = 1$. $U_{k, 0}(\uLambda) = [ - \varLambda_0]_k$ for $k \geq 0$.

{ \rm (2).} $A_{j, 0} = 1$, and $A_{j, 1} = \frac 1 2 {j(j+1)} $.

{ \rm (3).} $V_{k, j} (\uLambda) $ is a symmetric polynomial in $\varLambda_0, ..., \varLambda_{k-1}$. $V_{k, k} (\uLambda) = 1$. $V_{k, -1} (\uLambda) = 0$ and $V_{k, k-1} (\uLambda) = \sigma_{k-1, 1} (\uLambda) + \frac 12 {k(k-1)} $ for $k \geq 0$.

{ \rm (4).} $V_{k, j} (\uLambda) $ satisfies the following recurrence relation
\begin{equation}\label{6eq: Recurrence relation of V k j}
V_{k, j} (\uLambda) =  ( \varLambda_{k-1} + j) V_{k-1, j} (\uLambda) + V_{k-1, j-1} (\uLambda), \hskip 10 pt k \geq 1, j \geq 0.
\end{equation}
\end{lem}

\begin{proof}
(1-3) are  evident from the definitions.

(4). (\ref{6eq: Recurrence relation of V k j}) is obvious if $j \geq k$. If $k > j$, then the recurrence relations (\ref{6eq: recurrence relation of elementary symmetric polynomials}, \ref{6eq: definition of a k m}) for $\sigma_{k, m} (\uLambda)$ and $A_{j, m}$, in conjunction with the definition (\ref{6eq: definition of V k j (mu)}) of $V_{k, j} (\uLambda)$,  yield
\begin{align*}
V_{k, j} (\uLambda) = & \sum_{m=0}^{k - j } A_{j, k - j - m } \sigma_{k - 1, m} (\uLambda)\\
= &\, \varLambda_{k - 1} \sum_{m = 1}^{k - j }  A_{j, k - j - m } \sigma_{k-2, m-1} (\uLambda) + \sum_{m = 0}^{k - j }   A_{j, k - j - m } \sigma_{k-2, m} (\uLambda)\\
= &\, \varLambda_{k - 1} \sum_{m = 0}^{k - j - 1}  A_{j, k - j - m - 1} \sigma_{k-2, m} (\uLambda) \\
& + j \sum_{m = 0}^{k - j - 1} A_{j, k - j - m - 1} \sigma_{k-2, m} (\uLambda) + \sum_{m = 0}^{k - j } A_{j-1, k - j - m } \sigma_{k-2, m} (\uLambda)\\
= &\,  ( \varLambda_{k-1} + j) V_{k-1, j} (\uLambda) + V_{k-1, j-1} (\uLambda).
\end{align*}
\end{proof}

\begin{lem}\label{lem: Orthogonality}
 For $k \geq 0$ and $j \geq - 1$ such that $k \geq j$, we have
\begin{equation}\label{6eq: Orthogonality}
\sum_{l      = j}^k U_{k, l     } (\uLambda) V_{l     , j} (\uLambda) = \delta_{k, j},
\end{equation}
where $ \delta_{k, j}$ denotes Kronecker's delta symbol. Consequently, 
\begin{equation}\label{6eq: Orthogonality, 1}
\sum_{l      = j}^k V_{k, l     } (\uLambda) U_{l     , j} (\uLambda) = \delta_{k, j}.
\end{equation}
\end{lem}
\begin{proof}
(\ref{6eq: Orthogonality}) is obvious if either $k = j$ or $j = -1$. In the proof we may therefore assume that $k - 1 \geq j \geq 0$ and that (\ref{6eq: Orthogonality}) is already proven for smaller values of $k - j$ as well as for smaller values of $j$ and the same $k-j$.

By the recurrence relations (\ref{6eq: definition of U k j (mu)}, \ref{6eq: Recurrence relation of V k j}) for $U_{k, j} (\uLambda)$ and  $V_{ k, j} (\uLambda)$ and the induction hypothesis,
\begin{align*}
& \sum_{l      = j}^k U_{k, l     } (\uLambda) V_{l     , j} (\uLambda) \\
=  & - \sum_{l     =j}^{k-1} ( k-1 +  \varLambda_{l     }) U_{k-1, l     } (\uLambda) V_{l     , j} (\uLambda) + \sum_{l     =j}^{k } U_{k-1, l     -1} (\uLambda) V_{l     , j} (\uLambda)\\
 = & - ( k-1) \delta_{k-1, j} - \sum_{l     =j}^{k-1} \varLambda_{l     } U_{k-1, l     }  (\uLambda) V_{l     , j} (\uLambda) + \sum_{l     =j +1}^{k } \varLambda_{l     -1} U_{k-1, l     -1} (\uLambda) V_{l     -1, j} (\uLambda) \\
& + j \sum_{l     =j + 1}^{k } U_{k-1, l     -1} (\uLambda) V_{l     -1, j} (\uLambda) +  \sum_{l     =j }^{k } U_{k-1, l     -1} (\uLambda) V_{l     -1, j-1} (\uLambda) \\
= & -( k-1) \delta_{k-1, j} + 0 + j \delta_{k-1, j} + \delta_{k-1, j-1} = 0.
\end{align*}
This completes the proof of (\ref{6eq: Orthogonality}).
\end{proof}

Finally, we have the following explicit formulae for $A_{j, m}$.
\begin{lem}\label{lem: formula of a jm}
 We have $A_{0,0} = 1$, $A_{0, m} = 0$ if $m \geq 1$, and
\begin{equation}\label{6eq: formula for a j m}
A_{j, m} = \sum_{r=1}^j \frac {(- )^{j-r} r^{m+j}}{r! (j-r)!} \  \text { if }  j \geq 1, m \geq 0.
\end{equation}
\end{lem}
\begin{proof}
It is easily seen that $A_{0,0} = 1$ and $A_{0, m} = 0$ if $m \geq 1$.

It is straightforward to verify that the sequence given by (\ref{6eq: formula for a j m}) satisfies the recurrence relation (\ref{6eq: definition of a k m}), so it is left to show that \eqref{6eq: formula for a j m} holds true for $m = 0$. Initially, $A_{j, 0} = 1$, and hence one must verify
\begin{equation*}
\sum_{r=1}^j \frac {(- )^{j-r}
r^{j}}{r! (j-r)!} = 1.
 \end{equation*}
 This however follows from considering all the identities obtained by differentiating the following binomial identity up to $j$ times and then evaluating at $x = 1$,
\begin{equation*}
(x - 1)^j - (-1)^j = j! \sum_{r=1}^j \frac {(-1)^{j-r}} {r! (j-r)!} x^r.
\end{equation*}
\end{proof}

\subsection {Conclusion}

We first observe that, when $0\leq j \leq k \leq d+1$, both $U_{k, j} (\uLambda)$ and $V_{k, j} (\uLambda)$ are polynomials in $\varLambda_0, ..., \varLambda_{d}$  according to Lemma \ref{lem: simple facts on U, a, V} (1, 3). 
If one puts $\varLambda_{m} = \varLambda_m (\unu)$ for $m = 0, ..., d$, then   $U_{k, j} (\unu) = U_{k, j} (\uLambda)$. It follows from Lemma \ref{lem: Orthogonality} that  $V_{k, j} (\unu) = V_{k, j} (\uLambda)$. Moreover, the relations $\nu_{l     } = \lambda_{l     } - \lambda_{d+1}$, $l      = 1, ..., d$, along with the assumption $\sum_{l      = 1}^{d+1} \lambda_l      = 0$, yields
\begin{equation*} 
\varLambda_{m} (\unu) = (d+1)  \lambda_{d- m+1 }.
\end{equation*}
Now we can reformulate  \eqref{6eq: differential equation of J} in the following theorem.
\begin{thm}\label{thm: Bessel equations}
The Bessel function $J (x; \usigma, \ulambda)$ satisfies the following linear differential equation of order $d+1$
\begin{equation} \label{6eq: differential equation of J 2}
\sum_{j = 1}^{d+1} V_{d+1, j} (\ulambda) x^{j} w^{(j)} + \left( V_{d+1, 0} (\ulambda) - S_{d+1} (\usigma) (i(d+1))^{d+1}  x^{ d + 1} \right) w  = 0,
\end{equation}
where
\begin{equation*} 
S_{d+1} (\usigma) = \prod_{l      = 1}^{d+1} \varsigma_{l     }, \hskip 10 pt V_{d+1, j} (\ulambda) = \sum_{m=0}^{d-j+1} A_{j, d-j-m+1} (d+1)^m \sigma_{m} (\ulambda),
\end{equation*}
$\sigma_m (\ulambda)$ denotes the elementary symmetric polynomial in $\ulambda$ of degree $m$, with $\sigma_1 (\ulambda) = 0$, and $A_{j, m}$ is recurrently defined in Definition {\rm \ref{6defn: coefficients of Bessel equation} (3)} and explicitly given in Lemma {\rm \ref{lem: formula of a jm}}.  We shall call the equation \eqref{6eq: differential equation of J 2} a Bessel equation of index $\ulambda$, or simply a Bessel equation if the index $\ulambda$ is given.
\end{thm}


For a given index  $\ulambda$, (\ref{6eq: differential equation of J 2}) only provides two Bessel equations.  The sign $S_{d+1} (\usigma)$ determines which one of the two   Bessel equations   a Bessel function  $ J (x; \usigma, \ulambda)$ satisfies.
\begin{defn}
We call $S_{d+1} (\usigma) = \prod_{l      = 1}^{d+1} \varsigma_{l     }$ the sign of the Bessel function $ J (x; \usigma, \ulambda)$ as well as   the Bessel equation satisfied by $ J (x; \usigma, \ulambda)$. 
\end{defn}

Finally, we collect some simple facts on $V_{d+1, j} (\ulambda)$ in the following lemma, which will play   important roles later in the study of   Bessel equations. See \eqref{6eq: Orthogonality, 1} in Lemma \ref{lem: Orthogonality} and Lemma \ref{lem: simple facts on U, a, V} (3).

\begin{lem}\label{lem: simple facts on V (lambda)} We have

{ \rm (1).}
$\sum_{j = 0}^{d+1} V_{d+1, j} (\ulambda) [ - (d+1) \lambda_{d+1}]_{j} = 0.$

{ \rm (2).} $V_{d+1, d} (\ulambda) = \frac 12 {d (d+1)} $.

\end{lem}

\begin{rem}
If we define
\begin{equation*}
\boldsymbol J (x; \usigma, \ulambda) = J \lp(d+1)\-x; \usigma, (d+1)\- \ulambda \rp,
\end{equation*} 
then this normalized Bessel function satisfies a differential equation with coefficients free of powers of $(d+1)$, that is,
\begin{equation*} 
\sum_{j = 1}^{d+1} \boldsymbol V_{d+1, j} (\ulambda) x^{j} w^{(j)} + \left(\boldsymbol V_{d+1, 0} (\ulambda) - S_{d+1} (\usigma) i^{d+1}  x^{ d + 1} \right) w  = 0,
\end{equation*}
with
\begin{equation*} 
\boldsymbol V_{d+1, j} (\ulambda) = \sum_{m=0}^{d-j+1} A_{j, d-j-m+1} \sigma_{m} (\ulambda).
\end{equation*}
In particular, if $d=1$, $\ulambda = (\lambda, -\lambda)$, then the two normalized Bessel equations are
\begin{equation*} 
x^2 \frac {d^2 w} {dx^2} + x \frac {d w} {dx } + \left( - \lambda^2 \pm  x^{2} \right) w = 0.
\end{equation*}
These are exactly the Bessel equation and the modified Bessel equation of index $\lambda$. 
\end{rem}


\section{Bessel equations} \label{sec: Bessel equations}

The theory of linear ordinary differential equations with analytic coefficients\footnote{\cite[Chapter 4, 5]{Coddington-Levinson} and \cite[Chapter II-V]{Wasow} are the main references that we follow, and the reader is referred to these books for terminologies and definitions.} will be employed in this section to study Bessel equations. 

Subsequently, we shall use $z$ instead of $x$ to indicate complex variable.
For $\varsigma \in \{+, - \}$ and $\ulambda \in \BL^{n-1}$, we introduce the Bessel differential operator
\begin{equation}
\nabla_{\varsigma, \ulambda} = \sum_{j = 1}^{n} V_{n, j} (\ulambda) z^{j} \frac {d^j} {d z^j} + V_{n, 0} (\ulambda) - \varsigma (i n)^{n}  z^{n} .
\end{equation}
The Bessel equation of index $\ulambda$ and sign $\varsigma$  may be written as
\begin{equation} \label{7eq: differential equation, variable z}
\nabla_{\varsigma, \ulambda} (w)  = 0.
\end{equation}
Its corresponding system of differential equations is given by
\begin{equation}\label{7eq: differential equation, matrix form}
w' = B (z; \varsigma, \ulambda) w,
\end{equation}
with
\begin{equation*} 
\begin{split}
B   (z; \varsigma, \ulambda)  = 
\begin{pmatrix}
0 & 1 & 0 & \cdots & 0\\
0 & 0 & 1 & \cdots & 0\\
\vdots & \vdots & \vdots & \ddots & \vdots \\
0 & 0 &  \cdots & \cdots & 1 \\
- V_{n, 0} (\ulambda) z^{-n} + \varsigma (in)^n & -  V_{n, 1} (\ulambda) z^{-n+1} & \cdots & \cdots & -  V_{n, n-1} (\ulambda) z^{-1}
\end{pmatrix}.
\end{split}
\end{equation*}

We shall study Bessel equations on the Riemann surface $\BU$ associated with $\log z$, that is, the universal cover of $\BC \smallsetminus \{ 0 \}$. Each element in $\BU$ is represented by a pair $(x, \omega)$ with modulus $x \in \BR _+$ and argument $\omega \in \BR$, and will be denoted by $z = x e^{i\omega} = e^{\log x + i \omega}$ with some ambiguity. Conventionally, define $z^\lambda = e^{\lambda \log z }$ for $z \in \BU, \lambda \in \BC$, $\overline z = e^{- \log z}$, and moreover let $1 = e^{ 0}$, $-1 = e^{\pi i}$ and $\pm i = e^{\pm \frac 1 2 \pi i}$.

First of all, since Bessel equations are nonsingular on $\BU$, all the solutions of  Bessel equations are analytic on $\BU$.

Each Bessel equation has only two singularities at $z = 0$ and $z = \infty$. According to the classification of singularities, $0$ is a \textit{regular singularity}, so the Frobenius method gives rise to solutions of Bessel equations developed in series of ascending powers of $z$, or possibly logarithmic sums of this kind of series, whereas $\infty$ is an \textit{irregular singularity of rank one}, and therefore one may find certain formal solutions that are the asymptotic expansions of some actual solutions of Bessel equations.
Accordingly, there are two kinds of Bessel functions arising as solutions of Bessel equations. Their study is not only useful in understanding the Bessel functions $J (x ; \usigma, \ulambda) $ and the Bessel kernels $J_{(\ulambda, \udelta)} \lp x \rp$ on the real numbers (see \S \ref{sec: Analytic continuation of J (x; usigma; ulambda)}, \ref{sec: H-Bessel functions and K-Bessel functions, II}) but also significant in investigating the Bessel kernels on the complex numbers (see \cite[\S 7, 8]{Qi2}).

Finally, a simple but important observation is as follows. 

\begin{lem}\label{lem: phi (e pi i m/n z) is also a solution}
Let $\varsigma \in \{+, - \}$ and $a$ be an integer. If $\varphi (z)$ is a solution of the Bessel equation of sign $\varsigma$, then
$\varphi \big(e^{\pi i \frac a n} z\big)$ satisfies the  Bessel equation of sign $(-)^a \varsigma$.
\end{lem}

Variants of Lemma \ref{lem: phi (e pi i m/n z) is also a solution}, Lemma \ref{lem: J ell (z; sigma; lambda) relations},  \ref{lem: J(z ; usigma; lambda) relations to H +} and \ref{lem: J(z; lambda; xi) relations to J(xi z; lambda; 1)}, will play important roles later in \S \ref{sec: H-Bessel functions and K-Bessel functions revisited} when we study the connection formulae for various kinds of Bessel functions. 

\subsection{Bessel functions of the first kind}\label{sec: Bessel functions of the first kind}

The indicial equation associated with $\nabla_{\varsigma, \ulambda}$ is given as below,
\begin{equation*}
\sum_{j = 0}^n [\rho]_{j} V_{n, j} (\ulambda) = 0.
\end{equation*}
Let $P_{\ulambda} (\rho) $ denote the polynomial on the left of this equation. Lemma \ref{lem: simple facts on V (lambda)} (1) along with the symmetry of  $V_{n, j} (\ulambda)$ yields the following identity,
\begin{equation*}
\sum_{j = 0}^n [- n \lambda_{l     }]_{j} V_{n, j} (\ulambda) = 0,
\end{equation*}
for each $l      = 1, ..., n$. Therefore,
\begin{equation*} 
P_{\ulambda} (\rho) = \prod_{l      = 1}^n (\rho + n \lambda_{l     }).
\end{equation*}

Consider the formal series
\begin{equation*}
\sum_{m=0}^\infty c_m z^{\rho + m},
\end{equation*}
where the index $\rho$ and the coefficients $c_m$, with $c_0 \neq 0$, are to be determined.
It is easy to see that
\begin{equation*}
\nabla_{\varsigma, \ulambda} \sum_{m=0}^\infty c_m z^{\rho + m} = \sum_{m=0}^\infty c_m P_{\ulambda} (\rho + m)  z^{\rho + m} - \varsigma (in)^n \sum_{m=0}^\infty c_m z^{\rho + m + n}.
\end{equation*}
If the following equations are satisfied
\begin{equation}\label{7eq: series solution, system of equations}
\begin{split}
& c_m P_{\ulambda} (\rho + m) =0, \hskip 10 pt n > m  \geq 1,\\
 & c_m  P_{\ulambda} (\rho + m) - \varsigma (in)^n c_{m-n} = 0,  \hskip 10 pt  m \geq n,
\end{split}
\end{equation}
then 
\begin{equation*}
\nabla_{\varsigma, \ulambda} \sum_{m=0}^\infty c_m z^{\rho + m} = c_0  P_{\ulambda} (\rho ) z^{\rho}.
\end{equation*}
Given $l      \in\{ 1, ..., n \}$. Choose $\rho = - n \lambda_{l     }$ and let $c_0 = \prod_{k = 1}^n \Gamma \lp  { \lambda_{k} - \lambda_{l     }} + 1 \rp \- $. Suppose, for the moment, that no two components of $n \ulambda$ differ by an integer. Then  $P_{\ulambda} (- n \lambda_l      + m) \neq 0$   for any $m \geq 1$ and  $c_0\neq 0$, and hence the system of equations (\ref{7eq: series solution, system of equations}) is uniquely solvable. It follows that
\begin{equation}\label{7eq: series solution J ell (z; sigma)}
\sum_{m=0}^\infty \frac { (\varsigma i^n)^m  z^{ n (- \lambda_{l      } + m)} } { \prod_{k = 1}^n \Gamma \lp  { \lambda_{ k } - \lambda_{l     }}  + m + 1 \rp} 
\end{equation}
is a formal solution of the differential equation (\ref{7eq: differential equation, variable z}).

Now suppose that $\ulambda \in \BL^{n-1}$ is unrestricted. The series in (\ref{7eq: series solution J ell (z; sigma)}) is absolutely convergent, compactly convergent with respect to $\ulambda $, and hence gives rise to an analytic function of $z$ on 
the Riemann surface $\BU$, as well as an analytic function of $\ulambda$. 
We denote by $J_{l     } (z; \varsigma, \ulambda)$ the analytic function given by the series (\ref{7eq: series solution J ell (z; sigma)}) and call it a \textit{Bessel function of the first kind}. It is evident that  $J_{l     } (z; \varsigma, \ulambda)$ is an actual solution of (\ref{7eq: differential equation, variable z}).

\begin{defn}\label{defn: D, generic lambda}
Let $\BD^{n-1}$ denote the set of $\ulambda \in \BL^{n-1}$ such that no two components of $\ulambda$ differ by an integer. We call an index $\ulambda $ { generic} if $\ulambda \in \BD^{n-1}$.
\end{defn}

When $\ulambda \in \BD^{n-1}$,  all the $J_{l     } (z; \varsigma, \ulambda)$ constitute a fundamental set of solutions, since the leading term in the expression (\ref{7eq: series solution J ell (z; sigma)}) of $J_{l     } (z; \varsigma, \ulambda)$ does not vanish.
However, this is no longer the case if $ \ulambda \notin  \BD^{n-1}$. Indeed, if $\lambda_{l      } - \lambda_{k}$ is an integer, $k \neq l     $, then $J_{l      } (z; \varsigma, \ulambda) = (\varsigma i^n)^{\lambda_{l      } - \lambda_{k}} J_{k} (z; \varsigma, \ulambda)$. There are other solutions arising as  certain logarithmic sums of series of ascending powers of $z$. Roughly speaking, powers of $\log z$ may occur in some solutions. 
For more details the reader may consult \cite[\S 4.8]{Coddington-Levinson}.

\begin{lem}\label{lem: J ell (z; sigma; lambda) relations}
Let $a$ be an integer. We have
\begin{equation*}
J_{l     } \big(e^{\pi i \frac a n} z; \varsigma, \ulambda \big) = e^{- \pi i a \lambda_l     }J_{l     } (z; (-)^a \varsigma, \ulambda).
\end{equation*}
\end{lem}

\begin{rem}\label{rem: n=2 J-Bessel function}
If $n = 2$, then we have the following formulae according to \cite[3.1 (8), 3.7 (2)]{Watson},
\begin{equation*}
\begin{split}
J_1 (z; +, \lambda, - \lambda) = J_{- 2 \lambda} (2 z), \hskip 10 pt &  J_2 (z; +, \lambda, - \lambda) = J_{2 \lambda} (2 z),\\
J_1 (z; -, \lambda, - \lambda) = I_{- 2 \lambda} (2 z), \hskip 10 pt &  J_2 (z; -, \lambda, - \lambda) = I_{2 \lambda} (2 z).
\end{split}
\end{equation*}
\end{rem}


\begin{rem}\label{rem: GHF}
	Recall the definition of the generalized hypergeometric functions given by the series {\rm \cite[\S 4.4]{Watson}}
	\begin{equation*}
	{_pF_q} (\alpha_1, ..., \alpha_p ; \rho_1, ..., \rho_q ; z) = \sum_{m=0}^\infty \frac {(\alpha_1)_m ... (\alpha_p)_m} {m! (\rho_1)_m ... (\rho_q)_m} z^m.
	\end{equation*}
	It is evident that each Bessel function $J_{l     } (z; \varsigma, \ulambda)$ is closely related to a certain generalized hypergeometric function $ {_0F_{n-1}} $ as follows
	\begin{equation*} 
	J_{l     } (z; \varsigma, \ulambda) = \lp \prod_{k \neq l}  \frac {z^{ \lambda_k - \lambda_l }} {\Gamma \lp \lambda_k - \lambda_l + 1 \rp} \rp \cdot
	{_0F_{n-1}} \lp \left\{  \lambda_k - \lambda_l + 1 \right\}_{k \neq l} ; \varsigma i^n z^n \rp.
	\end{equation*}
\end{rem}

\subsection{\texorpdfstring{The analytic continuation of $J (x; \usigma, \ulambda)$}{The analytic continuation of $J (x; \varsigma, \lambda)$}}\label{sec: Analytic continuation of J (x; usigma; ulambda)}
For any given $\ulambda \in \BL^{n-1}$, since $J (x; \usigma, \ulambda)$ satisfies the Bessel equation of sign $S_n (\usigma)$, it admits a unique analytic continuation  $J (z; \usigma, \ulambda)$ onto $\BU$. Recall the definition
\begin{equation*} 
J (x ; \usigma, \ulambda) = \frac 1  {2 \pi i}\int_{\EuScript C} G(s; \usigma, \ulambda) x^{-n s} d s, \hskip 10 pt x \in \BR _+,
\end{equation*}
where $G (s; \usigma, \ulambda) = \prod_{k = 1}^{n} \Gamma (s - \lambda_{k} ) e \left(\frac 1 4 {\varsigma_{k} (s - \lambda_{k} )}   \right)$ and $\EuScript C$ is a suitable contour. 

Let $\varsigma = S_n (\usigma)$. For the moment, let us assume that $ \ulambda $ is generic.
For $l      = 1, ..., n$ and  $m = 0, 1, 2, ...$, $G (s; \usigma, \ulambda)$ has a simple pole at $\lambda_l      - m$ with residue
\begin{align*}
& (- )^m\frac {1} {m !}  e \lp \frac {\sum_{k = 1}^n \varsigma_{k} ( \lambda_{l     } - \lambda_{k} - m)} 4 \rp  \prod_{k \neq l      } \Gamma (\lambda_{l     } - \lambda_{k} - m) 
=   \\
& \pi^{n-1} e \lp - \frac {\sum_{k = 1}^n \varsigma_{k}\lambda_{k}} 4 \rp e \lp \hskip -2 pt \frac {\sum_{k = 1}^n \varsigma_{k}\lambda_{l     }} 4 \hskip -2 pt \rp   \lp \prod_{k \neq l      } \frac 1 {\sin \lp \pi (\lambda_{l     } - \lambda_{k} ) \hskip -1 pt \rp} \hskip -2 pt \rp   \hskip -2 pt \frac { (\varsigma i^n )^{m} } {\prod_{k=1}^n  \Gamma (\lambda_{k} - \lambda_{l      } + m + 1)} .
\end{align*}
Here we have used Euler's reflection formula for the Gamma function. 
Applying Cauchy's residue theorem, $J (x ; \usigma, \ulambda)$ is developed into an absolutely convergent series  on shifting the contour $\EuScript C$ far left, and, in view of (\ref{7eq: series solution J ell (z; sigma)}), we obtain 
\begin{equation} \label{7eq: J(x; sigma; lambda) as a sum of J ell (x; sigma; lambda)}
\begin{split}
J (z ; \usigma, \ulambda) =  \pi^{n-1} E (\usigma, \ulambda) \sum_{l      = 1}^n E_{l     } (\usigma, \ulambda)   S_{l     } ( \ulambda) J_{l     } (z; \varsigma, \ulambda), \hskip 10 pt z \in \BU,
\end{split}
\end{equation}
with 
\begin{align*}
& E (\usigma, \ulambda) = e \lp - \frac {\sum_{k = 1}^n \varsigma_{k}\lambda_{k}} 4 \rp, \hskip 3 pt E_{l     } (\usigma, \ulambda) = e \lp \frac {\sum_{k = 1}^n \varsigma_{k}\lambda_{l     }} 4 \rp, \hskip 3 pt  S_{l     } ( \ulambda) =     \prod_{k \neq l      } \frac 1 {\sin \lp \pi (\lambda_{l     } - \lambda_{k} ) \rp }  .
\end{align*}
Because of the possible vanishing of $\sin \lp \pi (\lambda_{l     } - \lambda_{k} ) \rp$, the definition of $S_{l     } ( \ulambda) $ may fail  to make sense if $\ulambda $ is not generic.  In order to properly interpret \eqref{7eq: J(x; sigma; lambda) as a sum of J ell (x; sigma; lambda)} in the non-generic case, one has to pass to the limit, that is, 
\begin{equation} \label{7eq: J(x; sigma; lambda) as a sum of J ell (x; sigma; lambda), limit}
J   (z ; \usigma, \ulambda) = \pi^{n-1} E (\usigma, \ulambda) \cdot  \lim_{\sstyle \boldsymbol \lambda' \ra \ulambda \atop \sstyle  \boldsymbol \lambda' \, \in \, \BD^{n-1}} \sum_{l      = 1}^n E_{l     } (\usigma, \ulambda')  S_l      (\boldsymbol \lambda')   J_{l     } (z; \varsigma, \boldsymbol \lambda' ). 
\end{equation}

We recollect the definitions of $L_{\pm} (\usigma)$ and $n_{\pm} (\usigma)$  introduced in Proposition \ref{prop: special example}.
\begin{defn}\label{defn: signature}
Let $\usigma \in \{+, - \}^n$. We define $L_\pm (\usigma) = \{ l      : \varsigma_l      = \pm \}$ and $n_\pm (\usigma) = \left| L_\pm (\usigma) \right|$.
The pair of integers $(n_+ (\usigma), n_- (\usigma) )$ is called the signature of $\usigma$, as well as the signature of the Bessel function $J(z ; \usigma, \ulambda)$.
\end{defn}

 With Definition \ref{defn: signature}, we reformulate (\ref{7eq: J(x; sigma; lambda) as a sum of J ell (x; sigma; lambda)}, \ref{7eq: J(x; sigma; lambda) as a sum of J ell (x; sigma; lambda), limit}) in the following lemma.

\begin{lem}\label{lem: J(x; sigma; lambda) as a sum of J ell (x; sigma; lambda)}
We have
\begin{equation*} 
\begin{split}
J (z ;  \usigma , \ulambda) = \pi^{n-1} E (\usigma, \ulambda)  \sum_{l      = 1}^n  E_{l     } (\usigma, \ulambda) S_{l     } ( \ulambda) J_{l     } \big(z; (-)^{n_- (\usigma)}, \ulambda \big),
\end{split}
\end{equation*}
with $E (\usigma, \ulambda) = e   \lp  - \frac 1 4 \sum_{k\in L_+ (\usigma)} + \frac 1 4 \sum_{k\in L_- (\usigma)} \lambda_{k} \rp$, $E_{l     } (\usigma, \ulambda) = e \lp \frac 1 4 { (n_+ (\usigma) - n_- (\usigma)) }  \lambda_{l     } \rp$ and $S_{l     } ( \ulambda) =    1/ \prod_{k \neq l      } {\sin \lp \pi (\lambda_{l     } - \lambda_{k} ) \rp }  $. When $\ulambda $ is not generic, the right hand side is to be replaced by its limit.
\end{lem}

\begin{rem}\label{7rem: connections to the first kind}
	In view of Proposition {\rm \ref{prop: Classical Bessel functions}} and Remark {\rm \ref{rem: n=2 J-Bessel function}}, Lemma {\rm \ref{lem: J(x; sigma; lambda) as a sum of J ell (x; sigma; lambda)}} is equivalent to the connection formulae in {\rm (\ref{2eq: connection formulae}, \ref{1eq: connection formula K})} {\rm ({\it see \cite[3.61(5, 6), 3.7 (6)]{Watson}})}.
\end{rem}

\begin{rem}
	In the  case when  $\ulambda = \frac 1 n \lp  \frac {n-1} 2, ..., - \frac {n-1} 2 \rp$, the formula in  Lemma {\rm \ref{lem: J(x; sigma; lambda) as a sum of J ell (x; sigma; lambda)}} amounts to splitting the Taylor series expansion of $e^{in \xi (\usigma) x }$ in \eqref{2eq: special example} 
	according to the congruence classes of indices modulo $n$. To see this, one requires the multiplicative formula of the Gamma function \eqref{2eq: multiplication theorem} as well as the trigonometric identity
	\begin{equation*}
		\prod_{k=1}^{n-1} \sin \lp \frac {k \pi} n  \rp = \frac n {2^{n-1}}.
	\end{equation*}
\end{rem}

Using Lemma \ref{lem: J ell (z; sigma; lambda) relations} and \ref{lem: J(x; sigma; lambda) as a sum of J ell (x; sigma; lambda)}, one proves the following lemma, which implies that the Bessel function $J(z ; \usigma, \ulambda)$ is determined by its signature up to a constant multiple. 

\begin{lem}\label{lem: J(z ; usigma; lambda) relations to H +}
Define $H^{\pm} (z; \ulambda) = J(z; \pm, ..., \pm, \ulambda)$. Then
\begin{equation*}
J(z ; \usigma, \ulambda) = e \lp \pm \frac {\sum_{l      \in L_ \mp (\usigma)} \lambda_{l      }} 2 \rp H^\pm \Big(e^{\pm \pi i \frac { n_\mp (\usigma)} n} z; \ulambda \Big).
\end{equation*}
\end{lem}

\begin{rem}\label{rem: Barnes integral}
We have the following Barnes type  integral representation,
\begin{equation} \label{7eq: Barnes contour integral 1}
J (z ; \usigma, \ulambda) = \frac 1  {2 \pi i}\int_{\EuScript C'} G(s; \usigma, \ulambda) z^{-n s} d s, \hskip 10 pt z \in \BU,
\end{equation}
where $\EuScript C'$ is a contour that starts from and returns to $- \infty$ after encircling the poles of the integrand counter-clockwise. Compare \cite[\S 6.5]{Watson}. Lemma {\rm \ref{lem: J(z ; usigma; lambda) relations to H +} } may also be seen from this integral representation.

When $- \frac {n_-(\usigma)} n \pi  < \arg z < \frac {n_+(\usigma)} n \pi $, the contour $\EuScript C'$ may be opened out to the vertical line $(\sigma)$, with $\sigma > \max  \{ \Re  \lambda_l      \}$. Thus
\begin{equation} \label{7eq: Barnes contour integral 2}
J (z ; \usigma, \ulambda) = \frac 1  {2 \pi i}\int_{(\sigma)} G(s; \usigma, \ulambda) z^{-n s} d s, \hskip 10 pt - \frac {n_-(\usigma)} n \pi  < \arg z < \frac {n_+(\usigma)} n \pi.
\end{equation}
On the boundary rays $\arg z = \pm \frac {n_\pm(\usigma)} n \pi$, the contour $ (\sigma) $ should be shifted to $\EuScript C$ defined as in {\rm \S \ref{sec: Definition of Bessel functions}}, in order to secure convergence.

The contour integrals in {\rm (\ref{7eq: Barnes contour integral 1}, \ref{7eq: Barnes contour integral 2})} absolutely converge, compactly in both $z$ and $\ulambda$. To see these, one uses Stirling's formula to examine the behaviour  of the integrand $G(s; \usigma, \ulambda) z^{-n s}$ on integral contours, where for {\rm(\ref{7eq: Barnes contour integral 1})} a transformation of  $G(s; \usigma, \ulambda)$ by Euler's reflection formula is required.
\end{rem}

\subsection{Asymptotics for Bessel equations and Bessel functions of the second kind}

Subsequently,  we   proceed to investigate the asymptotics at infinity for Bessel equations. 


\begin{defn}
	For  $\varsigma \in \{ +, -\}$ and a positive integer $N$, we let $\BX_{N} (\varsigma) $ denote the set of $N$-th roots of $\varsigma 1$.\footnote{Under certain circumstances, it is suitable to view an element $\xi$ of $\BX_{N} (\varsigma)$  as a point in $\BU$ instead of $\BC \smallsetminus \{0\}$. 
		This however should be clear from the context.}
\end{defn}

Before   delving into our general study, let us first consider the prototypical example given in Proposition \ref{prop: special example}. 

\begin{prop}
	For any $\xi \in \BX_{2n} (+)$, the function $ z^{- \frac {n-1} 2} e^{ i n \xi z}$ is a solution of the Bessel equation of index $ \frac 1 n \lp  \frac {n-1} 2, ..., - \frac {n-1} 2 \rp$ and sign $\xi^{\, n}$. 
\end{prop}

\begin{proof}
	When $ \Im \xi \geq 0$, this can be seen from Proposition \ref{prop: special example} and Theorem \ref{thm: Bessel equations}. For arbitrary $ \xi $, one makes use of Lemma \ref{lem: phi (e pi i m/n z) is also a solution}.
\end{proof}

\subsubsection{Formal solutions of Bessel equations at infinity}\label{sec: formal solutions at infinity}
Following \cite[Chapter 5]{Coddington-Levinson}, we shall consider the system of differential equations (\ref{7eq: differential equation, matrix form}).
We have 
\begin{equation*}
B (\infty; \varsigma, \ulambda) = \begin{pmatrix}
0 & 1 & 0 & \cdots & 0\\
0 & 0 & 1 & \cdots & 0\\
\vdots & \vdots & \vdots & \ddots & \vdots \\
0 & 0 & 0 & \cdots & 1 \\
\varsigma (in)^n & 0 & \cdots & \cdots & 0
\end{pmatrix}.
\end{equation*}
If one let  $\BX_n (\varsigma) = \left \{\xi_1 ,..., \xi_n \right \}$, then the eigenvalues of $B (\infty; \varsigma, \ulambda)$ are $in\xi_1 ,..., in\xi_n$.
The conjugation by the following matrix diagonalizes $B (\infty; \varsigma, \ulambda)$,
\begin{align*}
& T = \frac 1 n \begin{pmatrix}
1 & (in\xi_1)^{-1} &  \cdots & (in\xi_1)^{-n+1}\\
1 & (in\xi_2)^{-1} &  \cdots & (in\xi_2)^{-n+1}\\
\cdots & \cdots &  \cdots & \cdots \\
1 & (in\xi_n)^{-1} &  \cdots & (in\xi_n)^{-n+1}\\
\end{pmatrix},\\
& T\- = \begin{pmatrix}
1 & 1 & \cdots  & 1\\
in\xi_1 & in\xi_2   & \cdots & in\xi_n\\
\cdots & \cdots & \cdots & \cdots \\
(in\xi_1)^{n-1} & (in\xi_2)^{n-1} & \cdots & (in\xi_n)^{n-1}\\
\end{pmatrix}.
\end{align*}
Thus, the substitution $u = T w$ turns the system of differential equations (\ref{7eq: differential equation, matrix form}) into
\begin{equation}\label{7eq: differential equation, matrix form 2}
u' = A(z) u,
\end{equation}
where $A(z) = T B (z; \varsigma, \ulambda) T\-$ is a matrix of polynomials in $z\-$ of degree $n$,
\begin{equation*}
A(z) = \sum_{j=0}^n  z^{-j} A_j,
\end{equation*}
with 
\begin{equation}\label{7eq: A0, Aj}
\begin{split}
A_0 &= \varDelta = \mathrm {diag} \lp i n \xi_{l     } \rp_{l      = 1}^n, \\
A_j &= - i^{\, -j+1} n^{-j} V_{n, n-j} (\ulambda) \lp \xi_{k} \xi_{l     }^{\, -j} \rp_{k, l      = 1}^n, \ j = 1, ..., n.
\end{split}
\end{equation}
It is convenient to put $A_j = 0$ if $j > n$. The dependence on $ \varsigma, \ulambda$ and the ordering of the eigenvalues have been suppressed in our notations in the interest of brevity.

Suppose that $\widehat \Phi$ is a formal solution matrix for (\ref{7eq: differential equation, matrix form 2}) of the form
\begin{equation*} 
\widehat \Phi (z) = P(z) z^R e^{Q z},
\end{equation*}
where $P $ is a formal power series in $z\-$,
\begin{equation*}
P(z) = \sum_{m=0}^\infty z^{-m} P_m,
\end{equation*}
and $R$, $Q$ are constant diagonal matrices. Since
\begin{equation*}
\widehat \Phi ' = P' z^R e^{Qz} + z\- P R z^R e^{Qz} + P z^R Q e^{Qz} = \lp P' + z\- P R + P Q \rp  z^R e^{Qz},
\end{equation*}
the differential equation (\ref{7eq: differential equation, matrix form 2}) yields
\begin{equation*}
\sum_{m=0}^\infty z^{-m-1} P_m (R - m I) + \sum_{m=0}^\infty z^{-m } P_m Q = \lp  \sum_{j=0}^\infty z^{-j } A_j \rp \lp \sum_{m=0}^\infty z^{- m } P_m \rp,
\end{equation*}
where $I$ denotes the identity matrix.
Comparing the coefficients of various powers of $z\-$, it follows that $\widehat \Phi$ is a formal solution matrix for (\ref{7eq: differential equation, matrix form 2}) if and only if $R$, $Q$ and $P_m$ satisfy the following equations
\begin{equation}\label{7eq: relations of R Q Pm}
\begin{split}
P_0 Q - \varDelta P_0 &= 0\\
P_{m+1} Q - \varDelta P_{m+1} & = \sum_{j = 1}^{m+1} A_j P_{m-j+1} + P_m (m I - R ), \hskip 10 pt  m \geq 0.
\end{split}
\end{equation}
A solution of the first equation in (\ref{7eq: relations of R Q Pm}) is given by
\begin{equation}\label{7eq: Q=A0, P0=I}
Q = \varDelta, \hskip 10 pt P_0 = I.
\end{equation}
Using (\ref{7eq: Q=A0, P0=I}), the second equation in (\ref{7eq: relations of R Q Pm}) for $m=0$ becomes
\begin{equation}\label{7eq: relation m=0}
P_1 \varDelta - \varDelta P_1 =  A_1 - R.
\end{equation}
Since $\varDelta$ is diagonal, the diagonal entries of the left side of (\ref{7eq: relation m=0}) are zero, and hence the diagonal entries of $R$ must be identical with those of $A_1$. In view of (\ref{7eq: A0, Aj}) and Lemma \ref{lem: simple facts on V (lambda)} (2), we have $$A_1 = - \frac 1 n V_{n, n-1} (\ulambda) \cdot \lp \xi_{k} \xi_{l     }^{-1} \rp_{k, l      = 1}^n = - \frac {n-1} 2 \lp \xi_{k} \xi_{l     }^{-1} \rp_{k, l      = 1}^n,$$ and therefore 
\begin{equation}\label{7eq: R = - (n-2)/2}
R = - \frac {n-1} 2 I.
\end{equation}
Let $ p _{1,\, k l     }$ denote the $(k, l     )$-th entry of $P_1$. It follows from (\ref{7eq: A0, Aj}, \ref{7eq: relation m=0}) that
\begin{equation}\label{7eq: off-diagonal P1}
in (\xi_l      - \xi_k) p _{1,\, k l     } = - \frac {n-1} 2 \xi_k \xi_l     \-, \hskip 10 pt k \neq l     .
\end{equation}
The off-diagonal entries of $P_1$ are uniquely determined by \eqref{7eq: off-diagonal P1}. Therefore, a solution of (\ref{7eq: relation m=0}) is 
\begin{equation}\label{7eq: P1 = D1 + tilde P1}
P_1 = D_1 + P^o_1,
\end{equation}
where $D_1$ is any diagonal matrix and $ P^o_1$ is the matrix with diagonal entries zero and $(k, l     )$-th entry $p _{1,\, k l     }$, $ k \neq l     $. To determine $D_1$, one resorts to the second equation in (\ref{7eq: relations of R Q Pm}) for $m=1$, which, in view of (\ref{7eq: Q=A0, P0=I}, \ref{7eq: R = - (n-2)/2}, \ref{7eq: P1 = D1 + tilde P1}), may be written as
\begin{equation*}
P_2\varDelta - \varDelta P_2 -  \lp A_1 + \frac {n-1} 2 \rp D_1 - \frac {n+1} 2 P^o_1 = A_1 P^o_1 + A_2 + D_1.
\end{equation*}
The matrix on the left side has  zero diagonal entries. It follows that $D_1$ must be equal to the diagonal part of $- A_1 P^o_1 - A_2.$

In general, using (\ref{7eq: Q=A0, P0=I}, \ref{7eq: R = - (n-2)/2}), the second equation in (\ref{7eq: relations of R Q Pm}) may be written as
\begin{equation}\label{7eq: relations of Pm}
P_{m+1} \varDelta - \varDelta P_{m+1} = \sum_{j = 1}^{m+1} A_j P_{m-j+1} + \lp m + \frac {n-1} 2 \rp P_m, \hskip 10 pt  m \geq 0.
\end{equation}
Applying (\ref{7eq: relations of Pm}), an induction on $m$ implies that
\begin{equation*} 
P_{m } = D_m + P^o_m, \hskip 10 pt m\geq 1, 
\end{equation*}
where $D_m$ and $P^o_m$ are inductively defined as follows. Put $D_0 = I$. Let $m D_m$ be the diagonal part of
\begin{equation*}
- \sum_{j=2}^{m+1} A_{j } D_{m-j+1 } - \sum_{j=1}^{m} A_j P^o_{m-j+1}, 
\end{equation*} 
and let $P^o_{m+1}$ be the matrix with diagonal entries zero such that $P^o_{m+1} \varDelta - \varDelta P^o_{m+1}$ is the off-diagonal part of
\begin{equation*}
\begin{split}
\sum_{j=1}^{m+1} A_{j } D_{m-j+1 } + \sum_{j=1}^{m} A_j P^o_{m-j+1} 
 + \lp m + \frac {n-1} 2 \rp  P^o_m.   
\end{split}
\end{equation*}
In this way, an inductive construction of the formal solution matrix of (\ref{7eq: differential equation, matrix form 2}) is completed for the given initial choices $Q = \varDelta$, $P_0 = I$.

With the observations that $A_j$ is of degree $j$ in $\ulambda$ for $j \geq 2$  and that $A_1$ is constant, we may show the following lemma using an inductive argument.

\begin{lem} \label{lem: entries of Pm} The entries of $P_m$ are symmetric polynomial in $\ulambda$. If $m \geq 1$, then the off-diagonal entries of $P_m$ have degree  at most $2m-2$, whereas the degree of each diagonal entry is exactly $2m$.
\end{lem}


The first row of $T\- \widehat \Phi$ constitute a fundamental system of formal solutions of the Bessel equation  (\ref{7eq: differential equation, variable z}). Some calculations yield the following proposition, where for  derivatives of order higher than $n-1$ the differential equation (\ref{7eq: differential equation, variable z}) is applied.

\begin{prop}\label{prop: formal solution}
Let $\varsigma \in \{+, - \}$ and $\xi \in \BX_n (\varsigma)$. There exists a unique sequence of symmetric polynomials $B_m (\ulambda; \xi)$ in $\ulambda$ of degree $2 m$ 
and coefficients depending only on $m$, $\xi$ and $n$, normalized so that $B_0 (\ulambda; \xi) = 1$, such that
\begin{equation}\label{7eq: formal solution, asymptotic}
e^{i n \xi z} z^{- \frac { n-1} 2} \sum_{m=0}^\infty B_m (\ulambda; \xi) z^{-m}
\end{equation}
is a formal solution of the Bessel equation of sign $\varsigma$ {\rm(\ref{7eq: differential equation, variable z})}. We shall denote the formal series in {\rm (\ref{7eq: formal solution, asymptotic})} by $\widehat J (z; \ulambda; \xi)$. Moreover, the $j$-th formal derivative $\widehat J^{(j)} (z; \ulambda; \xi)$ is also of the form as {\rm (\ref{7eq: formal solution, asymptotic})}, but with coefficients depending on $j$ as well. 
\end{prop}

\begin{rem}
The above arguments are essentially adapted from the proof of \cite[Chapter 5, Theorem 2.1]{Coddington-Levinson}. This construction of the formal solution and Lemma {\rm \ref{lem: entries of Pm}} will be required later in \S { \rm \ref{sec: Error Bounds for the asymptotic expansions} } for the error analysis. 

However, This method is \textit{not} the best for the actual computation of the coefficients $B_m (\ulambda; \xi)$. We may derive the recurrent relations for $B_m (\ulambda; \xi)$ by a more direct but less suggestive approach as follows.

The substitution $w = e^{i n \xi z} z^{- \frac { n-1} 2} u $ transforms the Bessel equation {\rm (\ref{7eq: differential equation, variable z})} into
\begin{equation*}
\sum_{j = 0}^{n} W_{j} (z; \ulambda) u^{(j)} = 0,
\end{equation*}
where $W_j  (z; \ulambda)$ is a polynomial in $z\-$ of degree $n-j$,
\begin{equation*}
W_j (z; \ulambda) = \sum_{k = 0}^{n-j} W_{j, k} (\ulambda) z^{-k},
\end{equation*}
with
\begin{align*}
W_{0, 0} (\ulambda) &=  (in \xi)^{n} - \varsigma (in)^n = 0,\\
W_{j, k} (\ulambda) &= \frac { (in \xi)^{n-j-k} } {j! (n-j-k)!} \sum_{r=0}^k \frac {(n-r)!} { (k-r)!} \left[- \frac {n-1} 2\right]_{k-r} V_{n, n-r} (\ulambda), \hskip 10 pt (j,k) \neq (0,0).
\end{align*}
We have
\begin{equation*}
W_{0, 1} (\ulambda) = (in \xi)^{n-1} \lp n \lp - \frac {n-1} 2 \rp V_{n, n}  (\ulambda) + V_{n, n-1} (\ulambda) \rp = 0,
\end{equation*}
but $W_{1, 0} (\ulambda) = n( in \xi)^{n-1}$ is nonzero. Some calculations show that $ B_m (\ulambda; \xi)$ satisfy the following recurrence relations
\begin{align*}
(m-1) &  W_{1, 0} (\ulambda) B_{m-1} (\ulambda; \xi) =   \sum_{k=2}^{\min\{m , n\} } B_{m- k} (\ulambda; \xi) \sum_{j=0}^{k} W_{j, k-j} (\ulambda) [k-m]_j  , \hskip 10 pt m \geq 2.
\end{align*}

If $n = 2$, for a fourth root of unity $\xi = \pm 1, \pm i$ one may calculate in this way to obtain
\begin{equation*}
B_m (\lambda, - \lambda; \xi) = \frac {  \lp \frac 1 2 - 2\lambda \rp_m  \lp \frac 1 2 + 2\lambda\rp_m} { (4i\xi)^m m!}.
\end{equation*}
\end{rem}

\subsubsection{Bessel functions of the second kind}\label{sec: asymptotic expansions for the Bessel equations}

{\it Bessel functions of the second kind} are solutions of Bessel equations defined according to their asymptotic expansions at infinity. 
We shall apply several results in the asymptotic theory of ordinary differential equations  from \cite[Chapter IV]{Wasow}. 

Firstly, \cite[Theorem 12.3]{Wasow} implies the following lemma.

\begin{lem} [Existence of solutions] \label{lem: asymptotic, existence of a solution}
Let $\varsigma \in \{+, - \}$, $\xi \in \BX_n (\varsigma)$, and $\BS \subset \BU$ be an open sector with vertex at the origin and a positive central angle not exceeding $\pi$. Then there exists a solution of the Bessel equation of sign $\varsigma$ {\rm (\ref{7eq: differential equation, variable z})} that has the asymptotic expansion $\widehat J (z; \ulambda; \xi )$ defined in {\rm (\ref{7eq: formal solution, asymptotic})} on $\BS$. Moreover, each derivative of this solution has the formal derivative of  $\widehat J (z; \ulambda; \xi )$ of the same order as its asymptotic expansion.
\end{lem}


For two distinct $\xi, \xi' \in \BX_n(\varsigma)$, the ray emitted from the origin on which
$$\Re \lp (i \xi - i\xi') z \rp = - \Im \lp (\xi - \xi') z \rp = 0 $$ is called a \textit{separation ray}. 

We first consider the case $n=2$. It is clear that the separation rays constitute either the real or the imaginary axis and thus separate $ \BC \smallsetminus \{0 \}$ into two half-planes. 
 Accordingly, we define  $\BS_{\pm 1} = \left\{ z : \pm \Im z > 0 \right \}$ and $\BS_{\pm i} = \left\{ z : \pm \Re z > 0 \right \}$.

In the case $n \geq 3$, there are $2n$ distinct separation rays in $ \BC \smallsetminus \{0 \}$ given by the equations 
$$\arg z = \arg ( i \xi' ),  \hskip 10 pt \xi' \in \BX_{2n} (+).$$
These separation rays divide $\BC \smallsetminus \{0\}$ into $2 n$ many open sectors  \begin{equation}\label{7def: S pm xi}
\BS^{\pm}_\xi = \left\{ z : 0 < \pm \left( \arg z - \arg ( i \xoverline \xi ) \right) < \frac {\pi} n \right \}, \hskip 10 pt \xi \in \BX_{ n} (\varsigma).
\end{equation}
In both sectors $\BS^{+}_{ \xi }$ and $\BS^{-}_{ \xi }$ we have
\begin{equation}\label{7eq: condition for S xi}
\Re (i \xi  z) < \Re (i \xi' z) \ \text{ for all } \xi' \in \BX_n(\varsigma),\ \xi' \neq \xi.
\end{equation}
Let $\BS_{ \xi }$ be the sector on which \eqref{7eq: condition for S xi} is satisfied. It is evident that
\begin{equation}\label{7def: S xi}
\BS_\xi = \left\{ z : \left| \arg z - \arg ( i \xoverline \xi ) \right| < \frac {\pi} n \right \}.
\end{equation}


\begin{lem}\label{lem: asymptotic, uniqueness}
Let $\varsigma \in \{+, - \}$ and $\xi \in \BX_n( \varsigma)$.

{\rm (1. Existence of asymptotics).} If $n \geq 3$,  on the sector $\BS^{\pm}_{\xi }$,  all the solutions of the Bessel equation of sign $\varsigma$  have asymptotic representation a multiple of $\widehat J (z; \ulambda; \xi')$ for some $\xi' \in \BX_n( \varsigma)$. If $ n = 2$, the same assertion is true with  $\BS^{\pm}_{\xi }$ replaced by  $\BS_{\xi }$.

{\rm (2. Uniqueness of the solution).}
There is a unique solution of the Bessel equation of sign $\varsigma$ that possesses $\widehat J (z; \ulambda; \xi )$ as its asymptotic expansion on $\BS_{\xi }$ or any of its open subsector, and we shall denote this solution by $J (z; \ulambda; \xi )$. Moreover, $J^{(j)} (z; \ulambda; \xi ) \sim \widehat J^{(j)} (z; \ulambda; \xi )$ on $\BS_{\xi }$ for any $j \geq 0$.
\end{lem}

\begin{proof}
(1) follows directly from \cite[Theorem 15.1]{Wasow}.

For $n = 2$, since \eqref{7eq: condition for S xi} holds for the sector $\BS_{\xi }$, (2) is true according to \cite[Corollary to Theorem 15.3]{Wasow}. Similarly, if $n \geq 3$, (2) is true with  $\BS_{\xi }$ replaced by  $\BS^{\pm}_{\xi }$. Thus there exists a unique solution of the Bessel equation  of sign $\varsigma$ possessing $\widehat J (z; \ulambda; \xi )$ as its asymptotic expansion on $\BS^{\pm}_{\xi }$ or any of its open subsector. For the moment, we denote this solution   by  $J^{\pm} (z; \ulambda; \xi )$.
On the other hand, because $\BS_{\xi }$ has central angle $\frac 2 n \pi < \pi$, there exists a solution $J (z; \ulambda; \xi )$ with asymptotic $\widehat J (z; \ulambda; \xi )$ on a given open subsector $\BS \subset \BS_{\xi }$ due to Lemma \ref{lem: asymptotic, existence of a solution}. Observe that at least one of $\BS \cap \BS^+_{\xi }$ and $\BS \cap \BS^-_{\xi }$ is a nonempty open sector, say $\BS \cap \BS^+_{\xi } \neq \O $, then the uniqueness of $J (z; \ulambda; \xi )$ follows from that of $J^{+} (z; \ulambda; \xi )$ along with the principle of analytic continuation.
\end{proof}

\begin{prop}\label{prop: J (z; lambda; xi) on a large sector}
Let $\varsigma \in \{+, - \}$, $\xi \in \BX_n( \varsigma)$, $ \vartheta $ be a small positive constant, say $0 < \vartheta < \frac 1 2 \pi $, and define
\begin{equation}\label{7eq: S' xi (delta)}
\BS'_{\xi } (\vartheta) = \left\{ z : \left| \arg z - \arg ( i \xoverline \xi) \right|  < \pi + \frac {\pi} n - \vartheta \right \}.
\end{equation} Then $J (z; \ulambda; \xi )$ is the unique solution of the Bessel equation  of sign $\varsigma$ that has the asymptotic expansion $\widehat J (z; \ulambda; \xi )$ on $\BS'_{\xi } (\vartheta)$. Moreover, $ J^{(j)} (z; \ulambda; \xi ) \sim \widehat J^{(j)} (z; \ulambda; \xi )$ on $ \BS'_{\xi }(\vartheta)$ for any  $j \geqslant 0$.
\end{prop}

\begin{proof}
Following from Lemma \ref{lem: asymptotic, existence of a solution}, there exists a solution of the Bessel equation  of sign $\varsigma$ that has the asymptotic expansion $\widehat J (z; \ulambda; \xi )$ on the open sector
\begin{equation*}
\begin{split}
\BS^{\pm}_{\xi } (\vartheta) = \left\{ z : \frac {\pi} n - \vartheta < \pm \lp \arg z - \arg ( i \xoverline \xi ) \rp < \pi + \frac {\pi} n - \vartheta \right \}.
\end{split}
\end{equation*}
On the nonempty open sector $\BS_{\xi } \cap \BS^{\pm}_{\xi } (\vartheta)$ this solution must be identical with  $J (z; \ulambda; \xi )$ by Lemma \ref{lem: asymptotic, uniqueness} (2) and hence is equal to  $J (z; \ulambda; \xi )$ on $\BS_{\xi } \cup \BS^{\pm}_{\xi } (\vartheta)$ due to the principle of analytic continuation. Therefore, the region of validity of the asymptotic  $ J (z; \ulambda; \xi ) \sim \widehat J (z; \ulambda; \xi )$ may be widened from $\BS_{\xi }$ onto $\BS'_{\xi } (\vartheta) = \BS_{\xi } \cup \BS^{+}_{\xi } (\vartheta) \cup \BS^{-}_{\xi } (\vartheta)$. 
In the same way, Lemma \ref{lem: asymptotic, existence of a solution} and \ref{lem: asymptotic, uniqueness} (2) also imply  that $ J^{(j)} (z; \ulambda; \xi ) \sim \widehat J^{(j)} (z; \ulambda; \xi )$ on $  \BS'_{\xi }(\vartheta)$. 
\end{proof}

\begin{cor}\label{prop: J (z; lambda; xi) form a fundamental set of solutions}
Let $\varsigma \in \{+, - \}$.
All the $J (z; \ulambda; \xi )$, with $\xi \in \BX_n(\varsigma)$, form a fundamental set of solutions of the Bessel equation of sign $\varsigma$.
\end{cor}

\begin{rem}\label{rem: n=2, asymptotics}
If $n = 2$, by \cite[3.7 (8), 3.71 (18), 7.2 (1, 2), 7.23 (1, 2)]{Watson} we have the following formula of $J(z; \lambda, - \lambda; \xi)$, with $\xi = \pm 1, \pm i$, and the corresponding sector on which its asymptotic expansion is valid
\begin{equation*}
\begin{split}
& J(z; \lambda, - \lambda; 1) = \sqrt {\pi i} e^{\pi i \lambda} H^{(1)}_{2 \lambda} (2 z), \hskip 10 pt 
\BS'_{1} (\vartheta) = \left\{ z : - \pi + \vartheta < \arg z < 2 \pi - \vartheta \right\} ;\\
& J(z; \lambda, - \lambda; - 1) = \sqrt {- \pi i} e^{- \pi i \lambda} H^{(2)}_{2 \lambda} (2 z), \hskip 10 pt 
\BS'_{-1} (\vartheta) = \left\{z : - 2 \pi + \vartheta < \arg z < \pi - \vartheta \right\};\\
& J(z; \lambda, - \lambda; i) = \frac 2  {\sqrt \pi}K_{2 \lambda} (2 z), \hskip 10 pt 
\BS'_{i} (\vartheta) = \left\{ z :  |\arg z| < \frac 3 2 \pi - \vartheta \right\}; \\
& J(z; \lambda, - \lambda; - i) = 2 \sqrt \pi I_{2 \lambda} (2 z) - \frac {2 i}  {\sqrt \pi} e^{2\pi i \lambda} K_{2 \lambda} ( 2 z), \\
& \hskip 130 pt
\BS'_{- i} (\vartheta) = \left\{ z : - \frac 1 2 \pi + \vartheta < \arg z < \frac 5 2 \pi - \vartheta \right\} .
\end{split}
\end{equation*}
\end{rem}


\begin{lem}\label{lem: J(z; lambda; xi) relations to J(xi z; lambda; 1)}
Let $\xi\in \BX_{2n} (+)$. We have $$J( z; \ulambda; \xi) = (\pm \xi)^{ \frac {n-1} 2} J(\pm \xi z; \ulambda; \pm 1),$$
and $B_m (\ulambda; \xi) = (\pm \xi)^{- m} B_m (\ulambda; \pm 1)$. 

\end{lem}

\begin{proof}
By Lemma \ref{lem: phi (e pi i m/n z) is also a solution}, $(\pm \xi)^{ \frac {n-1} 2} J(\pm \xi z; \ulambda; \pm 1)$ is a solution of one of the two Bessel equations of index $\ulambda$. In view of Proposition \ref{prop: formal solution} and Lemma \ref{lem: asymptotic, uniqueness} (2), it  possesses $ \widehat J (z; \ulambda; \xi )$ as its asymptotic expansion on $\BS _{\xi } $ and hence must be identical with $ J( z; \ulambda; \xi)$. 
\end{proof}

\begin{term}\label{term: Bessel functions of the second kind}
For $\xi \in \BX_{2n} (+)$, $J (z; \ulambda; \xi)$ is called a Bessel function of the second kind.
\end{term}

\begin{rem}
	The results in this section do not provide any information on the asymptotics near zero of Bessel functions of the second kind, and therefore their connections with Bessel functions of the first kind can not be clarified here. We shall nevertheless find the connection formulae between the two kinds of Bessel functions later in \S {\rm \ref{sec: H-Bessel functions and K-Bessel functions revisited}}, appealing to the asymptotic expansion of the $H$-Bessel function $H^{\pm}  (z; \ulambda)$ on the half-plane $\BH^{\pm}$ that we showed earlier in \S {\rm \ref{sec: Bessel functions of K-type and H-Bessel functions}}. 
\end{rem}

\subsection{Error analysis for asymptotic expansions}\label{sec: Error Bounds for the asymptotic expansions}

The error bound for the asymptotic expansion of $J (z; \ulambda; \xi ) $ with dependence on  $\ulambda$ is always desirable for potential applications in analytic number theory. However, the author does not find any general results on the error analysis for differential equations of order higher than two.
We shall nevertheless combine and generalize the ideas from \cite[\S 5.4]{Coddington-Levinson} and \cite[\S 7.2]{Olver} to obtain an almost optimal error estimate for the asymptotic expansion of the Bessel function $J (z; \ulambda; \xi )$.
Observe that both of their methods have drawbacks for generalizations. \cite{Olver} hardly uses the viewpoint from differential systems as only the second-order case is treated, whereas \cite[\S 5.4]{Coddington-Levinson} is restricted to the positive real axis for more clarified expositions.


\subsubsection{Preparations}

We retain the notations from \S \ref{sec: formal solutions at infinity}. For  a positive integer $M $ denote by $P_{(M)}$ the polynomial in $z\-$,
\begin{equation*}
P_{(M)} (z) = \sum_{m=0}^M z^{-m} P_m,
\end{equation*} 
and by $\widehat \Phi_{(M)}$ the truncation of $\widehat \Phi$,
\begin{equation*}
\widehat \Phi_{(M)} (z) = P_{(M)} (z) z^{-\frac { n-1} 2} e^{\varDelta z}.
\end{equation*}
By Lemma \ref{lem: entries of Pm}, we have $\left|z^{-m} P_m \right| \lll_{m, n} \fC^{2m} |z|^{-m}$, so $P_{(M)}\-$ exists as an analytic function for $|z| > c_1 \fC^2$, where $c_1$ is some constant depending only on $M$ and $n$. Moreover,
\begin{equation}
\label{7eq: bound of P (M) P (M) inverse}
\left| P_{(M)} (z) \right|, \ \left| P_{(M)}\- (z) \right| = O_{M, n} (1), \hskip 10 pt |z| > c_1 \fC^2.
\end{equation} 
Let $A_{(M)}$ and $E_{(M)}$ be defined by 
\begin{equation*}
A_{(M)} = \widehat \Phi_{(M)}' \widehat \Phi\-_{(M)}, \hskip 10 pt E_{(M)} = A - A_{(M)}.
\end{equation*}
$A_{(M)}$ and $E_{(M)}$ are clearly analytic for $|z| > c_1 \fC^2$.
Since $$E_{(M)} P_{(M)} = A P_{(M)} - \lp P'_{(M)} - \frac {n-1} 2 z\- P_{(M)} + P_{(M)} \varDelta \rp,$$ 
it follows from the construction of $\widehat \Phi$ in \S \ref{sec: formal solutions at infinity} that $E_{(M)} P_{(M)}$ is a polynomial in $z\-$ of the form $\sum_{m=M+1}^{M+n} z^{-m} E_{m}$
so that
\begin{equation*}
\begin{split}
E_{M+1} & = P^o_{M+1} \varDelta - \varDelta P^o_{M+1},\\
E_m & = \sum_{j = m-M}^{\min\{m, n\}} A_j P_{m-j}, \hskip 10 pt M+1 < m \leq M+n.
\end{split}
\end{equation*} 
Therefore, in view of Lemma \ref{lem: entries of Pm}, $| E_{M+1}| \lll_{M, n} \fC^{2M} $ and $| E_{m}| \lll_{m, n} \fC^{m + M} $ for $M+1 < m \leq M+n$. It follows that $| E_{(M)}(z) P_{(M)}(z) | \lll_{M, n} \fC^{2M} z^{-M-1}$ for $|z| > c_1 \fC^2$, and this, combined with (\ref{7eq: bound of P (M) P (M) inverse}), yields
\begin{equation}\label{7eq: bound of E (M)}
| E_{(M)} (z) | = O_{M, n} \lp \fC^{2M} |z|^{-M-1} \rp.
\end{equation}


By the definition of $A_{(M)}$, for $|z| > c_1 \fC^2$,  $\widehat \Phi_{(M)}$ is a fundamental matrix of the system
\begin{equation}\label{7eq: differential system truncated}
u' = A_{(M)} u.
\end{equation}
We shall regard the differential system (\ref{7eq: differential equation, matrix form 2}), that is,
\begin{equation} \label{7eq: differential system, nonhomogeneous}
u' = A u = A_{(M)} u + E_{(M)} u,
\end{equation}
as a nonhomogeneous system with (\ref{7eq: differential system truncated}) as the corresponding homogeneous system.

\subsubsection{Construction of a solution}

Given $l      \in \{ 1, ..., n \}$,   let $$ \widehat \varphi_{(M), l     } (z) = p_{(M),l     } (z) z^{- \frac { n-1} 2} e^{in \xi_{l     } z}$$ be the $l     $-th column vector of the matrix $\widehat \Phi_{(M)}$, where $ p_{(M),l     }$ is the $l     $-th  column vector of $P_{(M)}$. Using a version of the variation-of-constants formula and the method of successive approximations, we shall construct a solution $\varphi_{(M), l     } $ of (\ref{7eq: differential equation, matrix form 2}), for $z$ in some suitable domain, satisfying
\begin{equation} \label{7eq: error bounds 1}
\left| \varphi_{(M), l     } (z) \right| = O_{M, n} \lp |z|^{- \frac { n-1} 2} e^{\Re \lp i n \xi_{l     } z \rp} \rp,
\end{equation}
and
\begin{equation} \label{7eq: error bounds 2}
\left| \varphi_{(M), l     } (z) - \widehat \varphi_{(M), l     } (z) \right| = O_{M, n} \lp \fC^{2 M} |z|^{-M- \frac { n-1} 2} e^{\Re \lp i n \xi_{l     } z \rp} \rp,
\end{equation}
with the implied constant in \eqref{7eq: error bounds 2} also depending on the chosen domain.

{\it Step} 1. {\it Constructing the domain and the contours for the integral equation.} For $C \geq c_1 \fC^2$ and $0 < \vartheta < \frac 12 \pi $, define the domain $\BD(C; \vartheta) \subset \BU$ by 
\begin{equation*}
\BD(C; \vartheta) = \left\{ z : \left| \arg z \right| \leq  \pi, |z| > C \right \} \cup \left\{ z : \pi < \left| \arg z \right| < \frac {3} 2 \pi - \vartheta,\ \Re\,  z < - C \right \}.
\end{equation*}
For $k \neq l     $ let $\omega(l     , k) = \arg (i \xoverline \xi_{l     } - i \xoverline \xi_{k}) = \arg (i \xoverline \xi_{l     }) + \arg ( 1 - \xi_{k} \xoverline \xi_{l      })$, and define
\begin{equation*}
\BD_{\xi_l     } (C; \vartheta) = \bigcap_{k \neq l     }  e^{i \omega(l     , k)} \cdot \BD(C; \vartheta).
\end{equation*}
With the observation that
$$
\left\{ \arg ( 1 - \xi_{k} \xoverline \xi_{l      }) : k \neq l      \right \} = \left \{ \lp \frac 1 2 - \frac a n\rp \pi :  a =1, ..., n-1 \right \},
$$
it is straightforward to show that $\BD_{\xi_l     } (C; \vartheta) = i \xoverline \xi_l      \BD' (C; \vartheta)$, where $\BD' (C; \vartheta)$ is defined to be the union of the sector 
\begin{equation*}
\left\{ z : \left| \arg z \right| \leq \frac \pi 2 + \frac {\pi} n,\ |z| > C \right \}
\end{equation*}
and the following two domains
\begin{equation*}
\begin{split}
& \left\{ z : \frac \pi 2 + \frac {\pi} n < \arg z < \pi + \frac {\pi} n - \vartheta,\ \Im \big(e^{- \frac 1 n \pi i } z\big) > C  \right \},\\
& \left\{ z : - \pi  - \frac {\pi} n + \vartheta < \arg z < - \frac \pi 2 - \frac {\pi} n,\ \Im\big(e^{\frac 1 n \pi i } z\big) < - C  \right \}.
\end{split}
\end{equation*}


\begin{figure}
\begin{center}
\begin{tikzpicture}
\draw [dotted, fill=lightgray](0,0) circle [radius=0.75];

\draw [->] (-2,0) -- (3,0);
\draw [->] (0, -1.75) -- (0,2.5);
\draw [thick](1.6,0) arc [radius=1.6, start angle=0, end angle= 140];
\draw [-, thick] (1.6, 0) -- (2.97, 0);

\draw [->] (2.5,0) -- (2.3,0);
\draw [->](1.6,0) arc [radius=1.6, start angle=0, end angle= 80];

\draw[fill=white] (0.75,0) circle [radius=0.04];
\node [below right] at (0.75,0) { \footnotesize $C$};
\draw[fill=white] (-1.2257, 1.0285) circle [radius=0.04];
\node [below] at (-1.2257, 1.0285) { \footnotesize $z$};
\draw[fill=white] (1.6, 0) circle [radius=0.04];
\node [above right] at (1.6, 0) { \footnotesize $|z|$};
\node [above right] at (1.06, 1.06) { \footnotesize $\EuScript {C}(z)$};
\node [below] at (0, -1.85) { \small $|\arg z| \leq \pi$};

\draw [dotted, fill=lightgray](7.75,0) arc [radius=0.75, start angle=0, end angle= 180];
\draw [fill=lightgray, lightgray] (6.25,0) to (6.25,-1) to (7, -1) to (7, 0) to (6.25,0);
\draw [fill=lightgray, lightgray] (6.25,0) to (6.25,-0.75) to (5.25, -1.75) to (7, -1.75) to (7, 0) to (6.25,0);
\draw [dotted](6.25,0) -- (6.25,-0.75);
\draw [dotted](7 ,0) -- (5.25 ,-1.75);

\draw [->] (4.5,0) -- (10,0);
\draw [->] (7, -1.75) -- (7,2.5);
\draw [thick](8.6,0) arc [radius=1.6, start angle=0, end angle= 180];

\draw [-, thick] (5.4, 0) -- (5.4, -0.75);
\draw [-, thick] (8.6, 0) -- (9.97, 0);
\draw [->] (9.5,0) -- (9.3,0);
\draw [->](8.6,0) arc [radius=1.6, start angle=0, end angle= 80];
\draw [->] (5.4, 0) -- (5.4, -0.5);

\draw [-](7,-0.25) arc [radius=0.25, start angle=270, end angle= 225];
\node [below] at (6.85,-0.15) {\scriptsize $\vartheta$};

\draw[fill=white] (7.75,0) circle [radius=0.04];
\node [below] at (7.75,0) {\footnotesize $C$};
\draw[fill=white] (5.4, -0.75) circle [radius=0.04];
\node [ left] at (5.4, -0.75) {\footnotesize $z$};
\draw[fill=white] (5.4, 0) circle [radius=0.04];
\node [above left] at (5.4, 0) {\footnotesize $\Re  z$};
\draw[fill=white] (8.6, 0) circle [radius=0.04];
\node [below] at (8.6, 0) {\footnotesize $- \Re  z$};
\node [above right] at (8.06, 1.06) {\footnotesize $\EuScript {C}(z)$};
\node [below] at (7, -1.85) {\small $\pi < \arg z < \frac 3 2\pi - \vartheta$};
\end{tikzpicture}
\end{center}
\caption{$\EuScript C (z) \subset \BD (C; \vartheta)$}\label{fig: C (z)}
\end{figure}

\begin{figure}
\begin{center}
\begin{tikzpicture}

\draw [dotted, fill=lightgray](0.75,0) arc [radius=0.75, start angle=0, end angle= 120];
\draw [dotted, fill=lightgray](0.75,0) arc [radius=0.75, start angle=0, end angle= -120];

\draw [fill=lightgray, lightgray] (0,0) to (-0.75/2,1.7320508*0.75/2) to (0.75, 0);
\draw [fill=lightgray, lightgray] (0,0) to (-0.75/2,-1.7320508*0.75/2) to (0.75, 0);

\draw [-](- 1.4,1.7320508*1.4) -- (0,0);
\draw [-](- 0.8 , -1.7320508*0.8 ) -- (0,0);

\draw [->] (-1.4,0) -- (3,0);
\draw [->] (0, -1.7320508*0.8) -- (0,2.5);

\draw [-](0, 0.25) arc [radius=0.25, start angle=90, end angle= 120];
\node [above] at (- 0.1, 0.2) {\footnotesize $\frac \pi n$};

\draw [thick](1.6,0) arc [radius=1.6, start angle=0, end angle= 75];
\draw [-, thick] (1.6, 0) -- (2.97, 0);

\draw [->] (2.5,0) -- (2.3,0);
\draw [->](1.6,0) arc [radius=1.6, start angle=0, end angle= 55];

\draw[fill=white] (1.6, 0) circle [radius=0.04];
\draw[fill=white] (0.2588190451*1.6, 0.96592582628*1.6) circle [radius=0.04];
\node [above] at (0.2588190451*1.6, 0.96592582628*1.6) {\footnotesize $ z$};

\node [above right] at (1.36, 0.66) { \footnotesize $\EuScript {C}'(z)$};

\node [below] at (0, -1.5) { \small $|\arg z| \leq   \frac 1 2 \pi + \frac 1 n   \pi  $};

\draw [dotted, fill=lightgray](7.75,0) arc [radius=0.75, start angle=0, end angle= 120];
\draw [fill=lightgray, lightgray] (7,0) to (7 - 1.7320508*1, -1) to (7 -0.75/2- 1.7320508*1,1.7320508*0.75/2-1) to (7 -0.75/2,1.7320508*0.75/2);
\draw [fill=lightgray, lightgray] (7,0) to (7 -0.75/2,1.7320508*0.75/2) to (7.75, 0);
\draw [dotted](7 -0.75/2,1.7320508*0.75/2) -- (7 -0.75/2- 1.7320508*1,1.7320508*0.75/2-1);

\draw [fill=lightgray, lightgray] (7,0) to (7 -0.75/2,1.7320508*0.75/2) to (7 - 1.7320508*1, -1) to (7 -0.96592582628*2.82842712475, 0.258819*2.82842712475);
\draw [dotted](7 -0.96592582628*2.82842712475, 0.258819*2.82842712475) -- (7,0);
\draw [dotted](7, 0) -- (7 - 1.7320508*1, -1);

\draw [fill=lightgray, lightgray] (7 - 1.7320508*1, -1) to (7 -0.96592582628*2.82842712475, 0.258819*2.82842712475) to (7 -0.96592582628*2.82842712475, -1);

\draw [-](7 - 1.4,1.7320508*1.4) -- (7,0);

\draw [->] (7 -0.96592582628*2.82842712475,0) -- (7+3,0);
\draw [->] (7 , -1) -- (7,2.5);

\draw [-](7 - 1.7320508*0.125,-0.125) arc [radius=0.25, start angle=210, end angle= 165];
\draw [fill=lightgray, lightgray](7 - 0.3,-0.05) arc [radius=0.15, start angle=0, end angle= 360];
\node [left] at (7 - 1.7320508*0.125-0.015, -0.125+0.08) {\scriptsize $\vartheta$};

\draw [-](7, 0.25) arc [radius=0.25, start angle=90, end angle= 120];
\node [above] at (7 - 0.1, 0.2) {\footnotesize $\frac \pi n$};

\draw [thick](7+ 1.6,0) arc [radius=1.6, start angle=0, end angle= 120];
\draw [-, thick] (7+ 1.6, 0) -- (7+ 2.97, 0);
\draw [->] (7 -0.8-1.7320508/6, 1.7320508*0.8-1/6) -- (7 -0.8-1.7320508/5, 1.7320508*0.8-0.2);
\draw [-, thick] (7 -0.8, 1.7320508*0.8) -- (7 -0.8-1.7320508/2, 1.7320508*0.8-1/2);

\draw [->] (7+ 2.5,0) -- (7+ 2.3,0);
\draw [->](7+ 1.6,0) arc [radius=1.6, start angle=0, end angle= 80];

\draw[fill=white] (7 -0.8, 1.38564) circle [radius=0.04];
\draw[fill=white] (7+ 1.6, 0) circle [radius=0.04];
\draw[fill=white] (7 -0.8-1.7320508/2, 1.38564-1/2) circle [radius=0.04];
\node [above left] at (7 -0.8-1.7320508/2 + 0.2, 1.38564-1/2) {\footnotesize $ z$};

\node [above right] at (8.06, 1.06) {\footnotesize $\EuScript {C}'(z)$};

\node [below] at (7, -1.5) {\small $\frac 1 2 \pi   + \frac 1 n {\pi}   < \arg z < \pi + \frac 1 n {\pi}   - \vartheta$};

\end{tikzpicture}
\end{center}
\caption{$ \EuScript C' (z) \subset \BD' (C; \vartheta)$}\label{fig: D xi (C)}
\end{figure}

For $z \in \BD (C; \vartheta)$ we define a contour $\EuScript C (z) \subset \BD (C; \vartheta)$ that starts from $\infty$ and ends at $z$; see Figure \ref{fig: C (z)}. For $z \in \BD (C; \vartheta)$ with $ |\arg z| \leq \pi$, the contour $\EuScript C (z)$ consists of the part of the positive axis where the magnitude exceeds $ |z|$ and an arc of the circle centered at the origin of radius $|z|$, with angle not exceeding $\pi$ and endpoint $z$. For $z \in \BD (C; \vartheta)$ with $\pi < |\arg z| < \frac {3 } 2 \pi - \vartheta$, the definition of the contour $\EuScript C (z)$ is modified so that the circular arc has radius $- \Re  z$ instead of $|z|$ and ends at $\Re  z$ on the negative real axis, and that $\EuScript C (z)$ also consists of a vertical line segment joining $\Re  z$ and $z$. 
The most crucial property that $\EuScript C (z)$ satisfies is the \textit{nonincreasing} of $\Re  \zeta$ along $\EuScript C (z)$.

We also define a contour $\EuScript C' (z)$ for $z \in \BD' (C; \vartheta) $ of a similar shape as $ \EuScript C \left( z \right)$ illustrated in Figure \ref{fig: D xi (C)}.

{\it Step} 2. {\it Solving the integral equation via successive approximations.}
We first split $\widehat \Phi_{(M)}\-$ into $n$ parts
\begin{equation*}
\widehat \Phi_{(M)}\- = \sum_{k = 1}^n \Psi^{(k)}_{(M)},
\end{equation*}
where the $j$-th row of $\Psi^{(k)}_{(M)}$ is identical with the $k$-th row of $\widehat \Phi_{(M)}\-$, or identically zero, according as $j = k$ or not.

The integral equation to be considered is the following
\begin{equation}\label{7eq: differential equation to integral equation}
u(z) = \widehat \varphi_{(M), l     } (z) + \sum_{k \neq l     } \int_{\infty e^{i \omega(l     , k)}}^z K_{k} (z, \zeta) u(\zeta) d \zeta + \int_{\infty i \overline \xi_l     }^z K_{l     } (z, \zeta) u(\zeta) d \zeta,
\end{equation}
where
\begin{equation*}
K_{k} (z, \zeta) = \widehat \Phi_{(M)} (z) \Psi_{(M)}^{(k)} (\zeta) E_{(M)}(\zeta), \hskip 10 pt z, \zeta \in \BD_{\xi_l     } (C; \vartheta), k = 1, ..., n,
\end{equation*}
the integral in the sum is integrated on the contour $ e^{i \omega(l     , k)} \EuScript C \big(e^{ - i \omega(l     , k)} z \big)$, whereas the last integral is on the contour $ i \xoverline \xi_l      \EuScript C' \left(- i  \xi_l      z \right)$. Clearly, all these contours lie in $\BD_{\xi_l     } (C; \vartheta)$. Most importantly, we note that $\Re ((i \xi_{l     } - i \xi_{k}) \zeta )$ is a negative multiple of $\Re \big( e^{- i \omega(l     , k)} \zeta \big)$ and hence is \textit{nondecreasing} along the contour $ e^{i \omega(l     , k)} \EuScript C \big(e^{ - i \omega(l     , k)} z\big)$.

By direct verification, it follows that if $u (z) = \varphi (z)$ satisfies \eqref{7eq: differential equation to integral equation}, with the integrals convergent, then $\varphi$ satisfies \eqref{7eq: differential system, nonhomogeneous}.

In order to solve (\ref{7eq: differential equation to integral equation}), define the successive approximations
\begin{equation}\label{7eq: successive approximation}
\begin{split}
& \varphi^0 (z) \equiv 0,\\
\varphi^{\alpha+1} (z) = \widehat \varphi_{(M), l     } (z) + & \sum_{k \neq l     } \int_{\infty e^{i \omega(l     , k)}}^z K_{k} (z, \zeta) \varphi^\alpha(\zeta) d \zeta + \int_{\infty i \overline \xi_l     }^z K_{l     } (z, \zeta) \varphi^\alpha(\zeta) d \zeta.
\end{split}
\end{equation}

The $(j, r)$-th entry of the matrix $\widehat \Phi_{(M)} (z) \Psi_{(M)}^{(k)} (\zeta)$ is given by
\begin{equation*}
\lp \widehat \Phi_{(M)} (z) \Psi_{(M)}^{(k)} (\zeta) \rp_{j r} = \lp P_{(M)} (z) \rp_{j k} \big( P\-_{(M)} (\zeta) \big)_{k r} \lp\frac z {\zeta} \rp^{-\frac { n-1} 2} e^{in \xi_{k} (z - \zeta)}.
\end{equation*}
It follows from (\ref{7eq: bound of P (M) P (M) inverse}, \ref{7eq: bound of E (M)}) that
\begin{equation}\label{7eq: bound of K ell' (z, zeta)}
| K_{k} (z, \zeta) | \leq c_2 \fC^{2M} |z|^{- \frac { n-1} 2} |\zeta|^{-M-1+ \frac { n-1} 2} e^{\Re (in \xi_{k} (z - \zeta))},
\end{equation}
for some constant $c_2$ depending only on $M$ and $n$. Furthermore, we may appropriately choose $c_2$ such that
\begin{equation}\label{7eq: bound of integral zeta -M-1}
\int_{\infty i \overline \xi_l     }^z \left| \zeta \right|^{-M-1} \left|d \zeta \right|, \ \int_{\infty e^{i \omega(l     , k)}}^z \left| \zeta \right|^{-M-1} \left|d \zeta \right| \leq c_2 C^{-M}, \hskip 10 pt  k \neq l     .
\end{equation}

According to \eqref{7eq: successive approximation}, $\varphi^1 (z) = \widehat \varphi_{(M), l     } (z) = p_{(M),l     } (z) z^{- \frac { n-1} 2} e^{in \xi_{l     } z}$, so
\begin{equation*}
\left|\varphi^1 (z) - \varphi^0 (z) \right| = \left|\varphi^1 (z) \right| \leq c_2 |z|^{- \frac { n-1} 2} e^{\Re (in \xi_l      z)}, \hskip 10 pt z \in \BD_{\xi_l     } (C; \vartheta).
\end{equation*}
We shall show by induction that for all $z \in \BD_{\xi_l     } (C; \vartheta)$
\begin{equation}\label{7eq: successive approximation induction hypothesis}
 \left|\varphi^\alpha (z) - \varphi^{\alpha -1} (z) \right| \leq c_2 \lp n c_2^2  \fC^{2M} C^{-M} \rp^{\alpha -1} |z|^{- \frac { n-1} 2} e^{\Re (in \xi_l      z)}.
\end{equation}
Let $z \in \BD_{\xi_l     } (C; \vartheta)$. Assume that (\ref{7eq: successive approximation induction hypothesis}) holds.
From (\ref{7eq: successive approximation}) we have
\begin{equation*}
\left|\varphi^{\alpha+1} (z) - \varphi^\alpha (z) \right|  
\leq \sum_{k \neq l     } R_{k} + R_{l     },
\end{equation*}
with 
\begin{align*}
R_{k} & = \int_{\infty e^{i \omega(l     , k)}}^z | K_{k} (z, \zeta)| \left|\varphi^\alpha (\zeta) - \varphi^{\alpha-1} (\zeta) \right| \left|d \zeta\right|,\\
R_{l     } & = \int_{\infty i \overline \xi_l     }^z | K_{l     } (z, \zeta)| \left|\varphi^\alpha (\zeta) - \varphi^{\alpha-1} (\zeta) \right| \left|d \zeta\right|.
\end{align*}
It follows from (\ref{7eq: bound of K ell' (z, zeta)}, \ref{7eq: successive approximation induction hypothesis}) that $R_{k} $ has bound
\begin{equation*}
c_2^2 \fC^{2M} \lp n c_2^2  \fC^{2M} C^{-M} \rp^{\alpha-1}  |z|^{- \frac { n-1} 2} e^{\Re (in \xi_l      z)} \int_{\infty e^{i \omega(l     , k)}}^z |\zeta|^{-M-1} e^{\Re (in (\xi_{l     } - \xi_{k}) (\zeta - z))} \left|d \zeta\right|.
\end{equation*}
Since $\Re ((i \xi_{l     } - i \xi_{k}) \zeta)$ is nondecreasing on the integral contour,
\begin{equation*}
R_{k} \leq c_2^2 \fC^{2M} \lp n c_2^2 \fC^{2M} C^{-M} \rp^{\alpha-1} |z|^{- \frac { n-1} 2} e^{\Re (in \xi_l      z)} \int_{\infty e^{i \omega(l     , k)}}^z |\zeta|^{-M-1} \left|d \zeta\right|,
\end{equation*}
and (\ref{7eq: bound of integral zeta -M-1}) further yields
\begin{equation*}
 R_{k} \leq c_2 n^{\alpha -1} \lp c_2^2  \fC^{2M} C^{-M} \rp^{\alpha} |z|^{- \frac { n-1} 2} e^{\Re (in \xi_l      z)}.
\end{equation*}
Similar arguments  show that $R_{l     }$ has the same bound as $R_{k}$. Thus (\ref{7eq: successive approximation induction hypothesis}) is true with $\alpha$ replaced by $\alpha+1$.

Set the constant $C = c \fC^2$ such that $c^M \geq 2 n c_2^2$. Then $ n c_2^2 \fC^{2M} C^{-M} \leq \frac 1 2$, and therefore the series $\sum_{\alpha = 1}^\infty (\varphi^\alpha(z) - \varphi^{\alpha-1}(z))$ absolutely and compactly converges. The limit function $\varphi_{(M), l     } (z)$ satisfies (\ref{7eq: error bounds 1}) for all $z \in \BD_{\xi_l     } (C; \vartheta)$. More precisely,
\begin{equation}\label{7eq: precise error bound, 1}
\left| \varphi_{(M), l     } (z) \right| \leq 2 c_2 |z|^{- \frac { n-1} 2} e^{\Re \lp i n \xi_{l     } z \rp}, \hskip 10 pt z \in \BD_{\xi_l     } (C; \vartheta).
\end{equation} 
Using a standard argument for successive approximations, it follows that $\varphi_{(M), l     }$ satisfies the integral equation \eqref{7eq: differential equation to integral equation} and hence the differential system \eqref{7eq: differential system, nonhomogeneous}.


The proof of the error bound (\ref{7eq: error bounds 2}) is similar. Since $\varphi_{(M), l     } (z)$ is a solution of the integral equation \eqref{7eq: differential equation to integral equation}, we have
\begin{equation*}
\left| \varphi_{(M), l     } (z) - \widehat \varphi_{(M), l     } (z) \right| \leq \sum_{k \neq l     } S_{k} + S_l     ,
\end{equation*}
where
\begin{align*}
S_{k}  = \int_{\infty e^{i \omega(l     , k)}}^z | K_{k} (z, \zeta)| \left|\varphi_{(M), l     } (\zeta) \right| \left|d \zeta\right|, \hskip 5 pt
S_{l     } = \int_{\infty i \overline \xi_l     }^z | K_{l     } (z, \zeta)| \left|\varphi_{(M), l     } (\zeta) \right| \left|d \zeta\right|.
\end{align*}
With the observation that $|\zeta| \geq \sin \vartheta \cdot |z|$ for $z \in \BD_{\xi_l     } (C; \vartheta)$ and $\zeta$ on the integral contours given above, we may replace \eqref{7eq: bound of integral zeta -M-1} by the following
\begin{equation}\label{7eq: bound of integral zeta -M-1, 2}
\int_{\infty i \overline \xi_l     }^z \left| \zeta \right|^{-M-1} \left|d \zeta \right|, \ \int_{\infty e^{i \omega(l     , k)}}^z \left| \zeta \right|^{-M-1} \left|d \zeta \right| \leq c_2 |z|^{-M}, \hskip 10 pt  k \neq l     ,
\end{equation}
with  $c_2$ now also depending on $\vartheta$.

The bounds (\ref{7eq: bound of K ell' (z, zeta)}, \ref{7eq: precise error bound, 1}) of
$ K_{k} (z, \zeta) $ and $\varphi_{(M), l     } (z)$ along with \eqref{7eq: bound of integral zeta -M-1, 2} yield
\begin{align*}
S_{k} & \leq 2 c_2^2 \fC^{2M} |z|^{- \frac { n-1} 2} e^{\Re (in \xi_l      z)} \int_{\infty e^{i \omega(l     , k)}}^z |\zeta|^{-M-1} e^{\Re (in (\xi_{l     } - \xi_{k}) (\zeta - z))} \left|d \zeta\right|\\
& \leq 2 c_2^2 \fC^{2M} |z|^{- \frac { n-1} 2} e^{\Re (in \xi_l      z)} \int_{\infty e^{i \omega(l     , k)}}^z |\zeta|^{-M-1} \left|d \zeta\right|\\
& \leq 2 c_2^3 \fC^{2M} |z|^{- M - \frac { n-1} 2} e^{\Re (in \xi_l      z)}.
\end{align*}
Again, the second inequality follows from the fact that $\Re ((i \xi_{l     } - i \xi_{k}) \zeta)$ is nondecreasing on the integral contour.
Similarly, $S_{l     }$ has the same bound as $S_{k}$. Thus  \eqref{7eq: error bounds 2} is proven and can be made precise as below
\begin{equation}\label{7eq: precise error bound, 2}
\left| \varphi_{(M), l     } (z) - \widehat \varphi_{(M), l     } (z) \right| \leq 2 n c_2^3 \fC^{2 M} |z|^{-M- \frac { n-1} 2} e^{\Re \lp i n \xi_{l     } z \rp}, \hskip 10 pt z \in \BD_{\xi_l     }(C; \vartheta).
\end{equation}

\subsubsection{Conclusion}

Restricting to the sector $\BS^{\pm}_{\xi_{l     }} \cap \{z : |z| > C \} \subset \BD_{\xi_l     } (C; \vartheta)$, with 
$\BS^{\pm}_{\xi_{l     }}$   replaced by $\BS_{\xi_{l     }}$ if $n = 2$, each $\varphi_{(M), l     }$ has an asymptotic representation a multiple of $\widehat \varphi_{k}$ for some $k$ according to Lemma \ref{lem: asymptotic, uniqueness} (1). Since $\Re (i \xi_{l     } z) < \Re (i \xi_{j} z)$ for all $ j \neq l     $, the bound (\ref{7eq: error bounds 1}) forces $k = l     $. Therefore, for any positive integer $M  $, $\varphi_{(M), l     }$ is identical with the unique solution $\varphi_{l     }$ of the differential system (\ref{7eq: differential equation, matrix form 2}) with asymptotic expansion $\widehat \varphi_{l     }$ on $\BS^{\pm}_{\xi_l     }$ (see Lemma \ref{lem: asymptotic, uniqueness}). Replacing $\varphi_{(M), l     }$ by $\varphi_{ l     }$ and absorbing the $M$-th term of $\widehat \varphi_{(M), l     }$ into the error bound, 
we may reformulate (\ref{7eq: precise error bound, 2}) as the following error bound for $\varphi_{l     }$
\begin{equation}\label{7eq: error bound 3}
\left| \varphi_{ l     } (z) - \widehat \varphi_{(M-1), l     } (z) \right| = O_{M, \vartheta, n} \lp \fC^{2 M} |z|^{-M- \frac { n-1} 2} e^{\Re \lp i n \xi_{l     } z \rp} \rp, \hskip 10 pt z \in \BD_{\xi_l     } (C; \vartheta).
\end{equation}
Moreover, in view of the definition of the sector $\BS'_{\xi_l     }(\vartheta)$ given in (\ref{7eq: S' xi (delta)}), we have
\begin{equation}\label{7eq: S'(delta) subset of D(C)}
\BS'_{\xi_l     }(\vartheta) \cap \left\{z : |z| >   \frac C {\sin \vartheta} \right \} \subset \BD_{\xi_l     } (C; \vartheta).
\end{equation}
Thus, the following theorem is finally established by (\ref{7eq: error bound 3}) and  (\ref{7eq: S'(delta) subset of D(C)}).

\begin{thm}\label{thm: error bound}
Let $\varsigma \in \{+, - \}$, $\xi \in \BX_n ( \varsigma)$, $0 < \vartheta < \frac 12 \pi $, $\BS'_{\xi} (\vartheta)$ be the sector defined as in {\rm (\ref{7eq: S' xi (delta)})}, and $M $ be a positive integer. Then there exists a constant $c$, depending only on $M$, $\vartheta$ and $n$, such that
\begin{equation}\label{7eq: asymptotic of J(z; lambda, xi)}
J (z; \ulambda; \xi) = e^{i n \xi z} z^{- \frac { n-1} 2} \lp \sum_{m=0}^{M-1} B_m (\ulambda; \xi) z^{-m} + O_{M, \vartheta, n} \lp \fC^{2 M} |z|^{- M } \rp \rp
\end{equation}
for all $z \in \BS'_{\xi} (\vartheta)$ such that $ |z| > c \fC^2 $. Similar asymptotic is valid for all the derivatives of $J (z; \ulambda; \xi)$, where the implied constant of the error estimate is allowed to depend on the order of the derivative.
\end{thm}

Finally, we remark that, since $B_m (\ulambda; \xi) z^{-m} $ is of size $O _{m, n} \lp \fC^{2 m} |z|^{- m } \rp$, the error bound in \eqref{7eq: asymptotic of J(z; lambda, xi)} is optimal, given that $\vartheta$ is fixed.

\section{Connections between various types of Bessel functions} \label{sec: H-Bessel functions and K-Bessel functions revisited}

Recall from \S \ref{sec: Analytic continuations of the H-Bessel functions} that the asymptotic expansion in Theorem \ref{thm: asymptotic expansion} remains valid for the $H$-Bessel function $H^{\pm} (z; \ulambda)$ on the half-plane $\BH^{\pm} = \left\{ z : 0 \leq \pm \arg z \leq \pi \right\}$ (see \eqref{5eq: asymptotic expansion 1}).
With the observations that $H^{\pm} (z; \ulambda)$ satisfies the Bessel equation of sign $( \pm)^n$, that the asymptotic expansions of $\sqrt n (\pm 2 \pi i)^{- \frac { n-1} 2} H^{\pm} (z; \ulambda)$ and $J (z; \ulambda; \pm 1)$ have exactly the same form and the same leading term due to Theorem \ref{thm: asymptotic expansion} and Proposition \ref{prop: formal solution}, and that $\BS_{\pm 1} = \left\{ z : \lp \frac 1 2    - \frac 1 n \rp \pi   < \pm \arg z < \lp \frac 1 2     + \frac 1 n \rp \pi  \right \} \subset \BH^{\pm} $, 
Lemma \ref{lem: asymptotic, uniqueness} (2) implies the following theorem.

\begin{thm}\label{thm: bridge between H + (z; lambda) and J( z; lambda; 1)}
We have
$$ H^{\pm} (z; \ulambda) = n^{-\frac 1 2} (\pm 2 \pi i)^{ \frac { n-1} 2} J (z; \ulambda; \pm 1),$$
and $B_m(\ulambda; \pm 1) = (\pm i)^{-m} B_m (\ulambda)$.
\end{thm}

\begin{rem}
The reader should observe that $\BS_{\pm 1} \cap \BR_+ = \O $, so Theorem {\rm \ref{thm: bridge between H + (z; lambda) and J( z; lambda; 1)}}  can not be obtained by the asymptotic expansion of $ H^{\pm} (x; \ulambda) $ on $\BR_+$ derived from Stirling's   formula in Appendix {\rm \ref{appendix: asymptotic}} {\rm(}see Remark {\rm\ref{rem: appendix only for R+}}{\rm)}.

\end{rem}

\begin{rem}
$B_m(\ulambda; \pm 1)$ can only be obtained from certain recurrence relations in \S {\rm \ref{sec: formal solutions at infinity}} from the differential equation aspect. On the other hand,  using the stationary phase method, \eqref{5eq: B mj} in \S {\rm \ref{sec: asymptotic expansions of H Bessel functions}} yields an explicit formula of $B_m (\ulambda)$. Thus, Theorem {\rm \ref{thm: bridge between H + (z; lambda) and J( z; lambda; 1)}} indicates that the recurrence relations for $B_m(\ulambda; \pm 1)$ are actually solvable!
\end{rem}

As consequences of Theorem \ref{thm: bridge between H + (z; lambda) and J( z; lambda; 1)}, we can establish the connections between various Bessel functions, that is, $J  (z; \usigma, \ulambda)$, $J_l      ( z; \varsigma, \ulambda)$ and $J(z; \ulambda; \xi )$. Recall that $J  (z; \usigma, \ulambda)$ has already been expressed in terms of   $J_l      ( z; \varsigma, \ulambda)$ in Lemma \ref{lem: J(x; sigma; lambda) as a sum of J ell (x; sigma; lambda)}.

\subsection{\texorpdfstring{Relations between $J (z; \usigma, \ulambda)$ and $J(z; \ulambda; \xi )$}{Relations  between $J (z; \varsigma, \lambda)$ and $J(z; \lambda; \xi )$}} \label{sec: J (z; sigma; lambda) and Bessel functions of the second kind}

$J (z; \usigma, \ulambda)$ is equal to a multiple of $H^{\pm} \Big( e^{\pm \pi i \frac { n_{\mp}(\usigma)} n} z; \ulambda \Big)$ in view of Lemma \ref{lem: J(z ; usigma; lambda) relations to H +}, whereas $J(z; \ulambda; \xi)$ is a multiple of  $J(\pm \xi z; \ulambda; \pm 1)$ due to Lemma \ref{lem: J(z; lambda; xi) relations to J(xi z; lambda; 1)}. Furthermore, the equality, up to  constant, between $H^{\pm} (z; \ulambda)$ and  $J( z; \ulambda; \pm 1)$ has just been established in Theorem \ref{thm: bridge between H + (z; lambda) and J( z; lambda; 1)}. We then arrive at the following corollary.

\begin{cor}\label{cor: connetions, 1}
Let $L_\pm (\usigma) = \{ l      : \varsigma_l      = \pm \}$ and $n_\pm (\usigma) = \left| L_\pm (\usigma) \right|$ be as in  Definition {\rm \ref{defn: signature}}. Let $c(\usigma, \ulambda) = e \lp \mp \frac {n-1} 8 \pm \frac {(n-1) n_{\pm}(\usigma)} {4 n} \mp \frac 1 2 {\sum_{l      \in L_{\pm}(\usigma)} \lambda_{l      }}   \rp$ and  $\xi (\usigma) = \mp e^{\mp \pi i \frac { n_{\pm} (\usigma)} n}$. Then 
\begin{equation*} 
J (z; \usigma, \ulambda) = \frac { ( 2\pi  )^{\frac { n-1} 2} c(\usigma, \ulambda)  } {\sqrt n}    J(z; \ulambda;  \xi (\usigma) ).
\end{equation*}
Here, it is understood that $\arg \xi (\usigma) = \frac {n_- (\usigma)} n \pi  = \pi -  \frac {n_+ (\usigma)} n \pi $.
\end{cor}


Corollary \ref{cor: connetions, 1} shows that $J (z; \usigma, \ulambda)$ should really be categorized in the class of Bessel functions of the second kind. 
Moreover, the asymptotic behaviours of the Bessel functions $J (z; \usigma, \ulambda)$ are classified by their signatures $(n_+ (\usigma), n_- (\usigma) )$. Therefore, $J (z; \usigma, \ulambda)$ is \textit{uniquely} determined by its  signature up to a constant multiple.

\subsection{Relations connecting the two kinds of Bessel functions}


From Lemma \ref{lem: J(z; lambda; xi) relations to J(xi z; lambda; 1)} and Theorem \ref{thm: bridge between H + (z; lambda) and J( z; lambda; 1)}, one sees that $J(z; \ulambda ; \xi)$ is a constant multiple of $H^+ (\xi z; \ulambda)$.
On the other hand, $H^+ (z; \ulambda)$ can be expressed in terms of Bessel functions of the first kind in view of Lemma \ref{lem: J(x; sigma; lambda) as a sum of J ell (x; sigma; lambda)}.
Finally, using Lemma \ref{lem: J ell (z; sigma; lambda) relations}, the following corollary is readily established.
\begin{cor}\label{cor: J(z; lambda; xi) and the J-Bessel functions}
Let $\varsigma \in \{+,-\}$. If $\xi \in \BX_n ( \varsigma )$, then
\begin{equation*} 
\begin{split}
J(z; \ulambda ; \xi) =  \sqrt n \lp - \frac {\pi i \xi } {2 } \rp^{\frac {n-1} 2}  
 \sum_{l      = 1}^{n}  \big( i  \xoverline \xi \big)^{n \lambda_l     } S_l      (\ulambda) J_l      ( z; \varsigma, \ulambda),
\end{split}
\end{equation*}
with $S_l      (\ulambda) = \prod_{k \neq l     } 1/ \sin \lp \pi (\lambda_l      - \lambda_{k} ) \rp  $. According to our convention, we have $\lp   - i \xi  \rp^{  \frac {n-1} 2 } = e^{ \frac {n-1} 2 \lp - \frac 1 2 \pi i + i \arg \xi \rp }  $ and $\big( i  \xoverline \xi \big)^{n \lambda_l     } = e^{ \frac 1 2 \pi i n \lambda_l      - i n \lambda_l      \arg \xi }$. When $\ulambda $ is not generic, the right hand side should be replaced by its limit.
\end{cor}

We now fix an integer $a$ and let $\xi _{  j} = e^{ \pi i \frac { {  2 j + a - 2}  } n} \in \BX_n \lp (-)^a \rp $, with   $j = 1, ..., n$. 
It follows from Corollary \ref{cor: J(z; lambda; xi) and the J-Bessel functions} that
\begin{equation*}
X  (z; \ulambda) 
= \sqrt n \lp \frac  {\pi } {2  } \rp^{\frac {n-1} 2} e^{- \frac 1 4 \pi i (n-1)}  \cdot  D  V(\ulambda) S (\ulambda) E  (\ulambda) 
Y (z; \ulambda),
\end{equation*}
with
\begin{align*}
	& X  (z; \ulambda)   = \big(J  \lp  z; \ulambda ; \xi _{  j}  \rp \big)_{j=1}^n, \hskip 10 pt   Y  (z; \ulambda) = \big(J_l       \lp  z; (-)^a , \ulambda \rp \big)_{l      =1}^n, \\ 
& D    = \mathrm{diag} \Big( \xi_{  j} ^{\frac {n-1} 2} \Big)_{j = 1}^n, \hskip 10 pt   E  (\ulambda)  = \mathrm{diag} \left( e^{ \pi i \lp \frac 1 2 n -  a \rp \lambda_l     } \right)_{l      = 1}^n,  \hskip 10 pt
S (\ulambda)   = \mathrm{diag}  \textstyle \big(    S_l      (\ulambda)  \big)_{l      = 1}^n, \\
& V(\ulambda)  = \lp e^{- 2 \pi i (j - 1) \lambda_{l     }} \rp_{j,\, l      = 1}^n.
\end{align*}
Observe that $V(\ulambda)$ is a {\it Vandermonde matrix}. 

\begin{lem}
	\label{8lem: Vandermonde}
	For an $n$-tuple $\ux = (x_1, ..., x_n) \in \BC^n$ we define the  Vandermonde matrix $V = \big(  x_{l     }^{ j - 1 } \big)_{j,\, l      = 1}^n$.
	For $d = 0, 1, ..., n-1$ and $m = 1, ..., n$, let $\sigma_{m, d} $ denote the elementary symmetric polynomial in $x_1, ..., \widehat {x_m}, ..., x_n$ of degree $d$, and let $\tau_m = \prod_{k \neq m} (x_m - x_k) $.
	If $\ux$ is generic in the sense that all the components of $\ux$ are distinct, then $V$ is invertible, and furthermore, the inverse of $V$ is $\lp (-)^{n-j } \sigma_{m, n-j}  \tau_m\- \rp_{m,\, j = 1}^n $.
\end{lem}
\begin{proof}[Proof of Lemma \ref{8lem: Vandermonde}]
	It is a well-known fact that $V$ is invertible whenever $\ux$ is generic.
	If one denotes by $w_{m, j}$ the $(m, j)$-th entry of $V\-$, then
	$$\sum_{j=1}^n w_{m, j} x_l     ^{j-1} = \delta_{m, l     }.$$
	The Lagrange interpolation formula implies the following identity of polynomials
	$$\sum_{j=1}^n w_{m, j} x^{j-1} = \prod_{k \neq m} \frac {x - x_k} {x_m - x_k}.$$
	Identifying the coefficients of $x^{j-1}$ on both sides yields the desired formula of $w_{m, j}$.
\end{proof}

\begin{cor}\label{8cor: inverse connection}
	Let $a$ be a given integer. For $j = 1, ..., n$ define $\xi _{  j} = e^{ \pi i \frac { {2 j  + a - 2}  } n}$. For $d = 0, 1, ..., n-1$ and $ l      = 1, ..., n$, let $\sigma_{l     , d} (\ulambda) $ denote the elementary symmetric polynomial in $e^{- 2 \pi i \lambda_1}, ..., \widehat {e^{- 2 \pi i \lambda_l     }}, ..., e^{- 2 \pi i \lambda_n}$ of degree $d$. Then 
	\begin{equation*}
	J_l      \big( z; (-)^a , \ulambda \big) = \frac {e^{\frac 3 4 \pi i (n-1)}} {\sqrt n (2\pi)^{\frac {n-1} 2}}   e^{\pi i \lp \frac 1 2 n +  a   - 2 \rp \lambda_l     } \sum_{j=1}^n (-)^{n-j} \xi_{  j} ^{- \frac {n-1} 2}  \sigma_{l     , n-j} (\ulambda) J \lp z; \ulambda; \xi_{  j} \rp.
	\end{equation*} 
\end{cor}

\begin{proof}
	Choosing $x_{l     } = e^{-2\pi i \lambda_l     }$ in Lemma \ref{8lem: Vandermonde}, one sees that if  $\ulambda$ is generic then the matrix $V(\ulambda)$ is invertible and its inverse is given by $$ \lp (-2i)^{1-n} \cdot (-)^{n-j }  \sigma_{l     , n-j} (\ulambda) e^{\pi i (n-2) \lambda_l     } S_{l     } (\ulambda) \rp_{l     ,\, j = 1}^n .$$
	Some straightforward calculations then complete the proof.
\end{proof}

\begin{rem}
	In view of Proposition {\rm \ref{prop: Classical Bessel functions}}, Remark {\rm \ref{rem: n=2 J-Bessel function}} and {\rm \ref{rem: n=2, asymptotics}}, when $n=2$, Corollary {\rm \ref{cor: J(z; lambda; xi) and the J-Bessel functions}} corresponds to the connection formulae {\rm ({\it \cite[3.61(5, 6), 3.7 (6)]{Watson}})}, 
	\begin{align*} 
	&  	H^{(1)}_\nu (z) = \frac {J_{-\nu} (z) - e^{- \pi i \nu} J_\nu (z) }{i \sin ( \pi \nu)}, \hskip 10 pt 
	H^{(2)}_\nu (z) = \frac { e^{\pi i \nu} J_\nu (z) - J_{-\nu} (z)}{i \sin ( \pi \nu)},\\
	&  K_{\nu} (z) =  \frac {\pi \lp I_{-\nu} (z) - I_\nu (z) \rp} {2 \sin (\pi \nu)}, \hskip 27 pt 
	\pi I_\nu (z) - i e^{\pi i \nu} K_\nu (z) = 
	\frac {\pi i \lp e^{-\pi i \nu} I_{ \nu} (z) - e^{\pi i \nu} I_{-\nu} (z) \rp} {2 \sin (\pi \nu)},
	\end{align*}
	whereas Corollary {\rm \ref{8cor: inverse connection}}, with $a = 0$ or $1$, amounts to the formulae  {\rm ({\it see \cite[3.61(1, 2), 3.7 (6)]{Watson}})}
	\begin{align*}
	& J_\nu (z) = \frac {H_\nu^{(1)} (z) + H_\nu^{(2)} (z)} 2, \hskip 61 pt 
	J_{-\nu} (z) =  \frac {e^{\pi i \nu} H_\nu^{(1)} (z) + e^{-\pi i \nu} H_\nu^{(2)} (z) } 2, \\
	& I_{\nu} (z) = \frac {i e^{ \pi i \nu} K_{\nu} (z) + \lp \pi I_{\nu} (z) - i e^{\pi i \nu} K_{\nu} (z)  \rp } {\pi},  
	I_{- \nu} (z) = \frac {i e^{-\pi i \nu} K_{\nu} (z) + \lp \pi I_{\nu} (z) - i e^{\pi i \nu} K_{\nu} (z)  \rp } {\pi} .
	\end{align*}
\end{rem}


\section{$H$-Bessel functions and $K$-Bessel functions, II}\label{sec: H-Bessel functions and K-Bessel functions, II}

In this concluding section, we apply Theorem  \ref{thm: error bound}  to  improve the results in \S \ref{sec: Bessel functions of K-type and H-Bessel functions} on the asymptotics of Bessel functions $J (x; \usigma, \ulambda)$ and the Bessel kernel $J_{(\ulambda, \udelta)} (\pm x)$ for $x \ggg \mathfrak C^2$. 

\subsection{Asymptotic expansions of $H$-Bessel functions} \label{8sec: Asymptotics of H-Bessel functions}

The following proposition is a direct  consequence of Theorem  \ref{thm: error bound} and \ref{thm: bridge between H + (z; lambda) and J( z; lambda; 1)}.

\begin{prop}\label{prop: improved asymptotic}
Let $0 < \vartheta < \frac 1 2 \pi  $.

{ \rm (1).} Let $M$ be a positive integer. We have 
\begin{equation} \label{7eq: asymptotic expansion H pm, improved}
\begin{split}
H^{\pm}  (z; \ulambda) = n^{- \frac 1 2}  &  (\pm 2 \pi i)^{ \frac { n-1} 2} e^{ \pm i n z} z^{ - \frac { n-1} 2} \\
& \lp \sum_{m=0}^{M-1} ( \pm i )^{ - m} B_{m} (\ulambda) z^{- m} + O_{M, \vartheta, n} \left( \mathfrak C^{2 M} |z|^{-M} \right) \rp,
\end{split}
\end{equation}
for all $z \in  \BS'_{\pm 1} (\vartheta)$ such that $|z| \ggg_{M, \vartheta, n} \fC^2$.

{ \rm (2).} Define $W^{\pm} (z;  \ulambda) = \sqrt n (\pm 2 \pi i)^{- \frac { n-1} 2} e^{\mp i n z} H^{\pm}  (z; \ulambda)$. Let $M - 1 \geq j \geq 0$. We have
\begin{equation} 
W^{\pm, (j)} (z;  \ulambda) = z^{- \frac { n-1} 2} \lp \sum_{m = j }^{M-1} ( \pm i )^{j - m} B_{m, j} (\ulambda) z^{- m } + O_{M, \vartheta, j , n} \left( \fC^{2 M - 2j} |z|^{-M}\right) \rp,
\end{equation}
 for all $z \in  \BS'_{\pm 1} (\vartheta)$  such that $|z| \ggg_{M, \vartheta, n} \fC^2$.
\end{prop}

Observe that \begin{equation*}
\begin{split}
& \BH^{\pm} = \left\{ z \in \BC : 0 \leq \pm \arg z \leq \pi \right\} \\
\subset &\ \BS'_{\pm 1} (\vartheta) = \left\{ z \in \BU : - \lp \frac 1 2 - \frac 1 n \rp \pi - \vartheta < \pm \arg z < \lp \frac 3 2 + \frac 1 n \rp \pi + \vartheta \right\}.
\end{split}
\end{equation*} 
Fixing $\vartheta$ and restricting to the domain $\left \{ z \in \BH^\pm : |z| \ggg_{M, n} \fC^2 \right\}$, Proposition \ref{prop: improved asymptotic} improves Theorem \ref{thm: asymptotic expansion}.

\subsection{Exponential decay of $K$-Bessel functions}

Now suppose that $J (z; \usigma, \ulambda)$ is a $K$-Bessel function so that $0 < n_\pm (\usigma) < n$. Since $\BR_+ \subset \BS'_{\xi (\usigma)} (\vartheta)$, Corollary \ref{cor: connetions, 1} and Theorem \ref{thm: error bound} imply that $J (x; \usigma, \ulambda)$, as well as  all its derivatives, is not only a Schwartz function at infinity, which was shown in Theorem \ref{thm: Bessel functions of K-type}, but also a function of exponential decay on $\BR_+$. 

\begin{prop}\label{8prop: K}
	If $J (x; \usigma, \ulambda)$ is a $K$-Bessel function, then for all $ x \ggg_{ n} \fC^2$
	\begin{equation*}
	J^{(j)} (x; \usigma, \ulambda) \lll_{j, n} x^{-\frac { n-1} 2} e^ { {- \pi  \Im \Lambda (\usigma, \ulambda)  - n I (\usigma) x}  }, 
	\end{equation*}
	where  $\Lambda (\usigma, \ulambda) = \mp \sum_{l      \in L_\pm (\usigma)} \lambda_{l     }$ and $I (\usigma) = \Im \xi (\usigma) = \sin \lp \frac { n_\pm (\usigma)} n \pi \rp > 0$. In particular, we have  
	\begin{equation*}
	J^{(j)} (x; \usigma, \ulambda) \lll_{j, n} x^{-\frac { n-1} 2} e^ { {  \pi  \mathfrak I   - n \sin \lp \frac 1 n \pi \rp x}  }, 
	\end{equation*}
	for all $K $-Bessel functions $J  (x; \usigma, \ulambda)$ with given $\ulambda$, where $ \mathfrak I  = \max \left\{ \left| \Im \lambda_l      \right| \right\} $.
	
\end{prop}


\subsection{\texorpdfstring{The asymptotic of the Bessel kernel $J_{(\ulambda, \udelta)}$}{The asymptotic of the Bessel kernel $J_{(\lambda, \delta)}$}}


In comparison with Theorem \ref{thm: asymptotic Bessel kernel, 1}, we have the following theorem.

\begin{thm}\label{thm: improved  asymptotic Bessel kernel}
	Let notations be as in Theorem {\rm \ref{thm: asymptotic Bessel kernel, 1}}. Then, for $x \ggg_{ n} \mathfrak C^{2 }$, we have
		\begin{align*}
		W_{ \ulambda }^{\pm } (x) =    \sum_{m=0}^{M-1}  B^{\pm}_{m } (\ulambda) x^{-   m -  \frac {n-1} 2 }  
		+ O_{  M ,  n} \left( \fC^{2 M } x^{- M - \frac {n-1} 2 }\right), 
		\end{align*} 
		and
		\begin{equation*}
		E^{\pm }_{(\ulambda, \udelta)} \left( x \right) =  O_{ n} \lp x^{- \frac {n-1} {2 } } \exp \lp { \pi   \mathfrak I - 2 \pi n \sin \lp \tfrac 1 n \pi \rp  x } \rp \rp,
		\end{equation*}
		with $\mathfrak I = \max \left\{|\Im \lambda_l     | \right\}$.
\delete{	
	{\rm (1).} If $n$ is even, then for $x \ggg_{M, n} \mathfrak C^{2 }$  
	\begin{align*}
	J_{(\ulambda, \udelta)} \left(x^n \right) =  
	\sum_{\pm} (\pm)^{\sum  \delta_l     }  \frac { e \lp \pm   \frac {n-1} 8  \pm   n x   \rp }  {\sqrt{n} x^{\frac {n-1} {2 } }}  \sum_{m=0}^{M-1}  (\pm 2 \pi i)^{- m} & B_{m} (\ulambda) x^{-  m  }  \\
	& + O_{M, n} \lp \mathfrak C^{2M} x^{- \frac {n-1} {2 } -  M   } \rp,
	\end{align*}
	and  for $x \ggg_{ n} \mathfrak C^{2 }$  
	\begin{align*}
	J_{(\ulambda, \udelta)} \left(- x^n \right) =  O_{ n} \lp x^{- \frac {n-1} {2 } } \exp \lp {\tfrac 1 2 \pi  n \mathfrak I - 2 \pi n \sin \lp \tfrac 1 n \pi \rp  x } \rp \rp,
	\end{align*} 
	with $\mathfrak I = \max \left\{|\Im \lambda_l     | \right\}$.
	
	{\rm (2).} If $n$ is odd, then  for $x \ggg_{M, n} \mathfrak C^{2  }$  
	\begin{equation*}
	\begin{split}
	J_{(\ulambda, \udelta)} \left(\pm x^n \right) = (\pm)^{\sum \delta_l     }  \frac { e \lp \pm   \frac {n-1} 8  \pm   n x   \rp } {\sqrt{n} x^{\frac {n-1} {2 } }}   \sum_{m=0}^{M-1}  (\pm 2 \pi i)^{- m} & B_{m} (\ulambda) x^{- m  } \\
	& + O_{M, n} \lp \mathfrak C^{2M} x^{- \frac {n-1} {2 } -   M } \rp.
	\end{split}
	\end{equation*}
}

\end{thm}


{\appendix

\section{An alternative approach to asymptotic expansions}\label{appendix: asymptotic} 

When $n=3$, the application of Stirling's asymptotic formula in deriving the asymptotic expansion of  a Hankel transform was first found in \cite[\S 4]{Miller-Wilton}. The asymptotic was later formulated more explicitly in \cite[Lemma 6.1]{XLi}, where the author attributed the arguments in her proof to \cite{Ivic}. Furthermore, using similar ideas as in \cite{Miller-Wilton},  \cite{Blomer} simplified the proof of \cite[Lemma 6.1]{XLi} (see the proof of \cite[Lemma 6]{Blomer}).  
This method using Stirling's asymptotic formula is however the only known approach so far in the literature.

Closely following \cite{Blomer}, we shall  prove the asymptotic expansions of $H$-Bessel functions $H^{\pm} (x; \ulambda)$ of any rank $n$ by means of  Stirling's asymptotic formula.

From  (\ref{2eq: definition of J (x; sigma)}, \ref{2eq: definition of G(s;  sigma, lambda)}) we have
\begin{equation}\label{10def: H pm via Mellin inversion}
H^{\pm} (x; \ulambda) = \frac 1  {2 \pi i}\int_{\EuScript C} \lp \prod_{l      = 1}^{n} \Gamma (s - \lambda_l      ) \rp e \left( \pm \frac {n s } 4 \right) x^{- n s} d s.
\end{equation}
In view of the condition $\sum_{l      = 1}^n \lambda_l      = 0$,  Stirling's asymptotic formula yields
\begin{equation*}
\prod_{l      = 1}^{n} \Gamma (s - \lambda_l      ) = n^{-ns}   \Gamma \lp n s - \frac {n-1} 2 \rp   \exp \lp \sum_{m=0}^{M } C_m(\ulambda) s^{-m} \rp \lp 1+ R_{M+1} (s) \rp 
\end{equation*}
for some constants $C_m(\ulambda)$ and remainder term $R_{M+1} (s) = O_{\ulambda, M, n} \lp |s|^{-M-1} \rp$. Using the Taylor expansion for the exponential function and some straightforward algebraic manipulations, the right hand side can be written as
\begin{equation*}
n^{-ns}   \sum_{m=0}^{M } \widetilde C_m(\ulambda) \Gamma \lp n s - \frac {n-1} 2 - m \rp \lp 1+ \widetilde R_{M+1, m} (s) \rp
\end{equation*}
for certain constants $\widetilde C_m(\ulambda)$ and   similar functions $\widetilde R_{M+1, m}(s) = O_{\ulambda, M, n} \lp |s|^{-M-1} \rp$. Suitably choosing the contour $\EC$, it follows from \eqref{2eq: n = 1, Mellin inversion} that
\begin{equation*}
\begin{split}
&\frac 1  {2 \pi i}\int_{\EuScript C} \Gamma \lp n s - \frac {n-1} 2 - m \rp e \left( \pm \frac {n s } 4 \right) (n x)^{- n s} d s \\
= &\, \frac {e \lp \pm \lp \frac {n-1} 8  + \frac 14 m   \rp \rp} {n (n x)^{\frac {n-1} 2 + m}} \cdot \frac 1  {2 \pi i}\int_{ n \EuScript C - \frac {n-1} 2 - m} \Gamma (s) e \lp \pm \frac s 4 \rp (nx)^{-s} ds = \frac {  (\pm i)^{ \frac {n-1} 2 + m} } {  n^{\frac {n+1} 2 + m} } \cdot \frac { e^{\pm i n x} } {  x^{ \frac {n-1} 2 + m} }.
\end{split}
\end{equation*}
As for the error estimate, let us assume $x \geqslant 1$. Insert the part containing $\widetilde R_{M+1, m}(s)$ into \eqref{10def: H pm via Mellin inversion} and shift the contour to the vertical line of real part $ \frac 1 n (M - \frac 1 2) + \frac 1 2 $. By Stirling's formula, the integral remains absolutely convergent  and is of size $O_{\ulambda, M, n} \big( x^{- M  - \frac {n-1} 2 }\big)$. Absorbing the last main term into the error, 
we arrive at the following asymptotic expansion
\begin{equation}\label{10eq: asymptotic expansion of H pm (x; lambda)}
\begin{split}
H^{\pm}  (x; \ulambda)  = e^{ \pm i n x} x^{ - \frac { n-1} 2} 
\lp \sum_{m=0}^{M-1} C^{\pm}_{m} (\ulambda) x^{- m} + O_{ \ulambda, M, n} \left( x^{- M } \right) \rp, \hskip 10 pt x \geqslant 1,
\end{split}
\end{equation}
where $ C^{\pm}_{m} (\ulambda)$ is some constant depending on $\ulambda$.

\begin{rem}\label{rem: appendix only for R+}
	For the analytic continuation  $H^{\pm}  (z; \ulambda)$, we have   the Barnes type  integral representation as in Remark {\rm \ref{rem: Barnes integral}}. This however does not yield an asymptotic expansion of $H^{\pm}  (z; \ulambda)$ along with the above method. The obvious issue is with the error estimate, as $\left| z^{-ns} \right|$ is unbounded on the integral contour if $|z| \ra \infty$.
\end{rem}

Finally, we make some comparisons between the three asymptotic expansions \eqref{10eq: asymptotic expansion of H pm (x; lambda)}, \eqref{5eq: asymptotic expansion 1} and \eqref{7eq: asymptotic expansion H pm, improved} obtained from 
\begin{itemize}
\item[-] Stirling's asymptotic formula,
\item[-] the method of stationary phase,
\item[-] the asymptotic method of ordinary differential equations.
\end{itemize}
Recall that  $\mathfrak C =  \max  \left\{ | \lambda_l      | \right \} + 1$, $\mathfrak R = \max  \left\{ |\Re \lambda_l      | \right \}$. Firstly, the admissible domains of these asymptotic expansions are
\begin{equation*}
\begin{split}
& \{ x \in \BR_+ : x \geq 1 \},\\
& \left\{ z \in \BC : |z| \geq \mathfrak C, \ 0 \leq \pm \arg z \leq \pi \right \}, \\
& \left\{ z \in \BU : |z| \ggg_{M, \vartheta, n} \mathfrak C^2, \ - \lp \frac 1 2 - \frac 1 n \rp \pi - \vartheta < \pm \arg z < \lp \frac 3 2 + \frac 1 n \rp \pi + \vartheta \right \},
\end{split}
\end{equation*}
respectively. The range of argument is extending while that of modulus is reducing. Secondly, the  error estimates are
\begin{equation*}
O_{\ulambda, M,   n} \left( x^{- M - \frac { n-1} 2} \right),\
O_{\mathfrak R, M, n} \left( \mathfrak C^{2 M} |z|^{-M} \right),\
O_{M, \vartheta, n} \left( \mathfrak C^{2 M} |z|^{- M - \frac { n-1} 2 } \right),
\end{equation*}
respectively. Thus, in the error estimate, the dependence of the implied constant  on $\ulambda$ is improving in all aspects.

\bibliographystyle{alphanum}
\bibliography{references}

\end{document}